\documentclass[11pt]{amsart}
\usepackage{amscd, amssymb,latexsym,amsmath, amscd, amsmath}
\usepackage[all]{xy}
\usepackage{anysize}
\usepackage[sc]{mathpazo} 
\usepackage{color}
\usepackage{xcolor}
\usepackage{mathtools}
\marginsize{2.5cm}{2.5cm}{2.5cm}{2.5cm}
\newtheorem{theorem}{Theorem}[section] %(If you want theorem numbered
\newtheorem{lemma}[theorem]{Lemma}%               with section number.  Same
\newtheorem{corollary}[theorem]{Corollary}%       goes for lemmas, etc.)
\newtheorem{proposition}[theorem]{Proposition}
\newtheorem{algorithm}[theorem]{Algorithm}
\theoremstyle{definition}

\newtheorem{example}{Example}
\newtheorem{definition}{Definition}

\theoremstyle{remark}
\newtheorem{remark}{Remark}
\newtheorem*{ack}{Acknowledgements}
\begin{document}
\title[Derivatives and Integrals]{Algorithms for parabolic inductions and Jacquet modules in $\mathrm{GL}_n$}
\author[Chan and Pattanayak]{Kei Yuen Chan and Basudev Pattanayak}
\address{Department of Mathematics, University of Hong Kong, Hong Kong}
\email{kychan1@hku.hk and pbasudev93@gmail.com}
\subjclass{22E50}
\keywords{Parabolic inductions, Jacquet modules, Jantzen-M\'inguez algorithm, M\oe glin-Waldspurger algorithm, highest derivative multisegments}
\date{}
\begin{abstract}
In this article, we present algorithms for computing parabolic inductions and Jacquet modules for the general linear group $G$ over a non-Archimedean local field. Given the  Zelevinsky data or Langlands data of an irreducible smooth representation $\pi$ of $G$ and an essentially square-integrable representation $\sigma$, we explicitly determine the Jacquet module of $\pi$ with respect to $\sigma$ and the socle of the normalized parabolic induction $\pi \times \sigma$. Our result builds on and extends some previous work of M\oe glin-Waldspurger, Jantzen, M\'inguez, and Lapid-M\'inguez, and also uses other methods such as sequences of derivatives and an exotic duality. As an application, we give a simple algorithm for computing the highest derivative multisegment and an algorithm for computing the Langlands parameter of the highest Bernstein-Zelevinsky derivatives.
\end{abstract}
\maketitle
%
%\tableofcontents
%
%%%%%%%%%%%%%%%%%%%%%%%%%%%%%%%%%%%%%%%% MAIN PART %%%%%%%%%%%%%%%%%%%%%%%%%%%%%%%%%%%%%%%%%%
%
\section{Introduction}

Let $\mathrm{GL}_n(F)$ be the general linear group over a non-archimedean local field $F$. The smooth representation theory of $\mathrm{GL}_n(F)$ is primarily shaped by two essential tools: parabolic inductions and Jacquet modules. These two functors have been studied long in the literature and have given resolutions to several classical problems, including the classification of irreducible representations \cite{Zel}, unitary dual problem \cite{Ta86, LM16}, Zelevinsky dual \cite{MW, Jan}, theta correspondence \cite{Mi08}, and the branching laws \cite{Ch22}. Even investigation on their own properties, such as irreducibility of parabolic inductions \cite{BLM13, LM16, LM19, LM_inven, LM_pamq}, composition factors of parabolic inductions \cite{Ta15, Gu21}, and homological properties of the functors \cite{Ch24}, is also interesting and has connections to other subjects.

Our primary goal is to establish some efficient computational tools for problems on branching laws and Bernstein-Zelevinsky derivatives (see  Section \ref{ss application} for more discussions). This relies on two essential notions: derivatives from Jacquet functors and integrals from parabolic inductions. We shall now introduce more notations to explain our results.

%In this paper, we investigate these fundamental components from a combinatorial perspective. Our goal is twofolds: one is to explain some hidden ingredients in the study of simple quotients of Bernstein-Zelevinsky derivatives and another is to provide explicit practical means to determine the quotient branching laws in \cite{Cha_qbl}. 

\subsection{Notion of some representations}
In his work \cite{Zel}, Zelevinsky classifies all irreducible smooth complex representations of $\mathrm{GL}_n(F)$ in terms of combinatorial objects known as multisegments. These multisegments consist of a finite number of segments attached to some irreducible supercuspidal representation, along with a pair of integers. Each segment $\Delta$ can be associated with a segment representation $\langle \Delta \rangle$ and a generalized Steinberg representation $\mathrm{St}(\Delta)$ through parabolic induction. For any irreducible smooth complex representation $\pi$ of $\mathrm{GL}_n(F)$, there exists a multisegment $\mathfrak{m}$ such that $\pi$ takes the form $Z(\mathfrak{m})$ (or $L(\mathfrak{m}$)), which is the unique irreducible submodule of the parabolic induction of tensor product of $\langle \Delta \rangle$ (resp. $\mathrm{St}(\Delta)$) for $\Delta \in \mathfrak{m}$ in a certain order. Refer to Sections \ref{sec:seg} and \ref{sec:LZ_classification} for more details of the notations.

\subsection{Notion of derivatives}
Let $\pi$ be an irreducible smooth representation of $\mathrm{GL}_n(F)$ and $\sigma$ be an essentially square integrable representation of $\mathrm{GL}_\ell(F)$ for $\ell < n$. Then $\sigma=\mathrm{St}(\Delta)$ for some segment $\Delta$. Let $N_\ell \subset \mathrm{GL}_n(F)$ be the unipotent radical of the standard parabolic subgroup corresponding to the partition $(n-\ell, \ell)$, i.e. $N_\ell$ is the unipotent subgroup containing matrices of the form $\begin{pmatrix}
	I_{n-\ell} & u\\
	& I_\ell
\end{pmatrix}$, 
where $u$ is a $(n-\ell) \times \ell$ matrix over $F$. There exists at most one irreducible smooth representation $\tau$ of $\mathrm{GL}_{n-\ell}(F)$ such that 
\begin{equation*}\label{def:right_der}
	\tau \boxtimes \sigma   \hookrightarrow \mathrm{Jac}_{N_\ell}(\pi)  \quad( \text{alternatively, } \mathrm{Jac}_{N_\ell}(\pi) \twoheadrightarrow \tau \boxtimes \sigma ), 
\end{equation*}
where $\mathrm{Jac}_{N_\ell}(\pi)$ is the normalized Jacquet module of $\pi$ associated to $N_\ell$. If such $\tau$ exists, that $\tau$ is called the derivative of $\pi$ under $\mathrm{St}(\Delta)$ and is denoted by $\mathrm{D}^\mathrm{R}_{\Delta}(\pi)$. If no such $\tau$ exist, we set $\mathrm{D}^\mathrm{R}_{\Delta}(\pi)=0$. Similarly, there exists at most one irreducible smooth representation $\tau^\prime$ of $\mathrm{GL}_{n-\ell}(F)$ such that $\sigma \boxtimes \tau^\prime   \hookrightarrow \mathrm{Jac}_{N_{n-\ell}}(\pi)$.  The left derivative $\mathrm{D}^\mathrm{L}_{\Delta}(\pi)$ is defined as $\tau^\prime$ if such $\tau^\prime$ exists;  otherwise, we set $\mathrm{D}^\mathrm{L}_{\Delta}(\pi)=0$.  The derivatives under irreducible supercuspidal representations $\rho$ are called $\rho$-derivatives, and under essentially square integrable representations $\mathrm{St}(\Delta)$ are called St-derivatives.

We remark that in some contexts e.g. \cite{Jan} (also see \cite{Jan14, Xu17, Jan18, Ato, Ta22} for other classical groups), the notion of derivatives is (roughly) defined to be a collection of the composition factors of the form $\tau \boxtimes \sigma$ appearing in the semisimplification of $\mathrm{Jac}_{N_{\ell}}(\pi)$, which is related but different notion from the one used in this article. However, when one considers the so-called highest $\rho$-derivatives, those notions coincide with certain multiplicity one results \cite{Jan, Min}. The notion we use here fits our needs better in the applications discussed below. We also remark that in general, those factors $\tau\boxtimes \sigma$ may not appear in a submodule or quotient of $\mathrm{Jac}_{N_{\ell}}(\pi)$, and some of such higher structure issue is studied in \cite{Ch24}.

The complexity of the algorithms presented in this article is about sorting segments in certain orderings. It is quite computable by hand, even in some large cases (see examples given in the article), and we hope this can give insights and is more applicable to the problems mentioned. We also remark that for all composition factors in parabolic inductions and Jacquet modules, it can in general be computed by (variations of) Kazhdan-Lusztig algorithms, but such information does not directly give much information on socles and cosocles, as our algorithms do. 

\subsection{Notion of integrals}
Let $\pi$ be an irreducible smooth representation of $\mathrm{GL}_n(F)$ and let $\sigma=\mathrm{St}(\Delta)$ be an essentially square integrable representation of $\mathrm{GL}_\ell(F)$ for some segment $\Delta$. Then, there exists a unique simple submodule  
\[\mathrm{I}^\mathrm{R}_{\Delta}(\pi)  \hookrightarrow \pi \times \mathrm{St}(\Delta) ~ \big(\text{resp. } \mathrm{I}^\mathrm{L}_{\Delta}(\pi)  \hookrightarrow \mathrm{St}(\Delta) \times \pi \big)\]
of the normalized parabolic induction $\pi \times \mathrm{St}(\Delta)$ (resp. $\mathrm{St}(\Delta) \times \pi$). We call $\mathrm{I}^\mathrm{R}_{\Delta}(\pi)$ (resp. $\mathrm{I}^\mathrm{L}_{\Delta}(\pi)$) the right (resp. left) integral of $\pi$ under $\mathrm{St}(\Delta)$. According to \cite[Corollary 2.4]{LM19} and the second adjointness of parabolic induction, we have: 
\begin{equation*}\label{rel:der_int}
	\mathrm{D}^\mathrm{R}_{\Delta} \circ \mathrm{I}^\mathrm{R}_{\Delta}(\pi)\cong \pi \text{ and if } \mathrm{D}^\mathrm{R}_{\Delta}(\pi) \neq 0,  \mathrm{I}^\mathrm{R}_{\Delta} \circ \mathrm{D}^\mathrm{R}_{\Delta}(\pi)\cong \pi.    
\end{equation*}
A similar result holds for left derivatives and left integrals. When $\sigma=\mathrm{St}(\Delta)$ is a supercuspidal representation $\rho$, the integrals $\mathrm{I}^\mathrm{R/L}_{\Delta}(\pi)$ under $\sigma$ are called $\rho$-integral.

The problem of determining when $\mathrm{I}^{\mathrm R}_{\Delta}(\pi)\cong \mathrm{I}^{\mathrm L}_{\Delta}(\pi)$ is explored in \cite{LM16}, and also in a series of work of \cite{LM19,LM_inven,LM_pamq} for more general situation of $\square$-irreducible representations.

\subsection{Main results} \label{ss main results}
Let $\pi$ be an irreducible smooth representation of $\mathrm{GL}_n(F)$ and let  $\sigma=\mathrm{St}(\Delta)$, which is a generalized Steinberg representation of $\mathrm{GL}_\ell(F)$ for some segment $\Delta$. Then, there exist multisegments $\mathfrak{m}$, and $\mathfrak{n}$ such that $\pi=L(\mathfrak{m})$ (in Langlands classification) and $\pi=Z(\mathfrak{n})$ (in  Zelevinsky classification).  By Algorithm \ref{alg:der:Lang} and Algorithm \ref{alg:int:Lang} (resp. Algorithm \ref{alg:der_Zel} and Algorithm \ref{alg:int:Zel}),  we can attach multisegments $\mathcal{D}^\mathrm{Lang}_\Delta(\mathfrak{m})$ and $\mathcal{I}^\mathrm{Lang}_\Delta(\mathfrak{m})$
(resp. $\mathcal{D}^\mathrm{Zel}_\Delta(\mathfrak{n})$ and $\mathcal{I}^\mathrm{Zel}_\Delta(\mathfrak{n})$) respectively to the multisegment $\mathfrak{m}$ (resp. $\mathfrak{n}$). We then show the following main results of this paper:

\begin{enumerate}
	\item For derivatives in Langlands classification: \[\mathrm{D}^\mathrm{R}_\Delta(L(\mathfrak{m}))\cong \begin{cases}
		L \left(\mathcal{D}^\mathrm{Lang}_\Delta(\mathfrak{m})\right) &\mbox{ if } \mathcal{D}^\mathrm{Lang}_\Delta(\mathfrak{m}) \neq \infty\\
		0 &\mbox{ otherwise},
	\end{cases}\]
	and for derivatives in Zelevinsky classification: \[\mathrm{D}^\mathrm{R}_\Delta(Z(\mathfrak{n}))\cong \begin{cases}
		Z \left(\mathcal{D}^\mathrm{Zel}_\Delta(\mathfrak{n})\right) &\mbox{ if } \mathcal{D}^\mathrm{Zel}_\Delta(\mathfrak{n}) \neq \infty\\
		0 &\mbox{ otherwise}.
	\end{cases}\]
	\item For integrals in Langlands classification: $\mathrm{I}^\mathrm{R}_\Delta(L(\mathfrak{m}))\cong
		L \left(\mathcal{I}^\mathrm{Lang}_\Delta(\mathfrak{m})\right),$  and for integrals in Zelevinsky classification:  $\mathrm{I}^\mathrm{R}_\Delta(Z(\mathfrak{n})) \cong 
		L \left(\mathcal{I}^\mathrm{Zel}_\Delta(\mathfrak{n})\right) .$
\end{enumerate}
Here, we denote $\mathcal{D}^\mathrm{Lang}_\Delta(\mathfrak{m})$ or $\mathcal{D}^\mathrm{Zel}_\Delta(\mathfrak{n})$ as $\infty$  if some steps of the respective algorithms fail to construct the multisegment $\mathcal{D}^\mathrm{Lang}_\Delta(\mathfrak{m})$ or $\mathcal{D}^\mathrm{Zel}_\Delta(\mathfrak{n})$. A similar result holds for left derivatives and left integrals in both classifications.

The algorithms are mainly formulated in terms of linked relations of segments, which is probably not so surprising, as already seen in the work of \cite{MW, Jan, Min, LM16}. Our results can be viewed as extensions of theirs, and for more comparison of results/methods in \cite{MW, Jan, Min, LM16}, see Remarks \ref{rmk lang alg matching} and \ref{rmk zel alg mw}. However, as seen in this article, it takes much effort to obtain our formulation. We shall explain some of our key inputs in the following section.

\subsection{Methods of proofs}

We shall use rather different perspectives to deal with each case of the main results in Section \ref{ss main results}. We now highlight some key ingredients of our proofs:
\begin{enumerate}
 \item (Derivative algorithm for Langlands classification, Section \ref{sec:der_Lang}) We exploit the commutativity of derivatives and sequences of derivatives to reduce to $\rho$-derivatives. The representation-theoretic counterpart of such ideas is studied in \cite{Cha_csq}.
 \item (Derivative algorithm for Zelevinsky classification, Section \ref{sec:der_Zel}) We use the M\oe glin-Waldspurger (MW) algorithm as a basic case and introduce the notion of minimally linked multisegments to systematically study the combinatorics arising from multiple MW algorithms.
 \item (Integral algorithm for Langlands classification, Section \ref{sec:int}) We establish an exotic duality between the right integral algorithm and the left derivative algorithm. This duality allows one to transfer properties between two algorithms and then use those properties to reduce to the case of $\rho$-integrals. Here we refer the duality to be exotic because it comes from the left derivative and right integral, which have no obvious relation from a representation-theoretic viewpoint.
 \item (Integral algorithm for Zelevinsky classification, Section \ref{sec:int_Zel}) The proof uses an idea of gluing minimally linked multisegments, explained in Appendix \ref{s glue minimal multiseg}, in addition to the MW algorithm. 
\end{enumerate}

\subsection{Applications/motivations} \label{ss application}

\begin{enumerate}
\item Let $\pi$ be an irreducible smooth representation of $\mathrm{GL}_n(F)$ and $\Delta$ be a segment. Define $\varepsilon^\mathrm{R}_{\Delta}(\pi)$ to be the largest non-negative integer $k$
 such that $\left(\mathrm{D}^\mathrm{R}_{\Delta}\right)^k(\pi) \neq 0$. For a segment $[a,b]_\rho$, define the (right) $\eta$-invariant by \[\eta^\mathrm{R}_{[a,b]_\rho}(\pi)=\left(\varepsilon^\mathrm{R}_{[a,b]_\rho}(\pi),\varepsilon^\mathrm{R}_{[a+1,b]_\rho}(\pi),...,\varepsilon^\mathrm{R}_{[b,b]_\rho}(\pi)\right)\] 
Let $\Delta, \Delta^\prime$ be two segments. A triple $\left(\Delta, \Delta^\prime, \pi\right)$ is called combinatorially RdLi-commutative if $\mathrm{D}^\mathrm{R}_\Delta(\pi) \neq 0$ and $\eta^\mathrm{R}_{\Delta}\left(\mathrm{I}^\mathrm{L}_{\Delta^\prime}(\pi)\right)=\eta^\mathrm{R}_{\Delta}(\pi)$. The notion of generalized GGP relevant pair in \cite{Cha_duality} is an extension of this notion to a multisegment version. Results in this article allow one to check the GGP relevance condition, and, for example, are practical to recover some classical known branching laws such as generic representations. While the algorithms in this article are sufficient to determine quotient branching law in finite processes, one can still improve the efficiency of determining quotient branching laws by incorporating the ideas in \cite{Cha_qbl}. We hope to address this elsewhere.
\item It is shown in \cite{Cha_qbl} and  \cite{Cha_csq} that a sequence of derivatives of essentially square-integrable representations can be used to compute simple quotients of Bernstein-Zelevinsky derivatives of irreducible representations. It is also shown in \cite{Cha_csq} that those simple quotients can be classified in terms of the highest derivative multisegment and removal process (see Section \ref{ss hd removal}). As a consequence of our study, we also explain a simple algorithm to compute the highest derivative multisegment of an irreducible representation in Section \ref{s alg hd der}, and hence provide an effective solution to that classification problem.
\item Using \cite{CW25}, one can also compute explicitly simple quotients and submodules of certain parabolically induced modules and simple quotients of some translation functors for $\mathrm{GL}_n(\mathbb C)$. Those algebraic structures also have applications to branching laws. One may expect to obtain similar applications for $\mathrm{GL}_n(\mathbb R)$ via the Schur-Weyl duality constructed in \cite{CT12}.
\end{enumerate}

\begin{ack}
Some of the results are announced in Tianyuan Conference on Real Reductive Groups and Theta Correspondence in August 2024. The authors would like to thank the organizers Ning Li, Jiajun Ma, Binyong Sun, and Lei Zhang for their kind invitations. This project is supported in part by the Research Grants Council of the Hong Kong Special Administrative Region, China (Project No: 17305223, 17308324 ) and the National Natural Science Foundation of China (Project No. 12322120). 
\end{ack}

%%%%%%%%%%%%%%%%%%%%%%%%%%%%%% Preliminaries  %%%%%%%%%%%%%%%%%%%%%%%%%%%%%%%%%%%%%%%%%
%
\section{Preliminaries}\label{prelim}
Let $F$ be a non-archimedean local field with normalized
 absolute value $|\cdot|_F$. For every integer $n \geq 0$, let $G_n=\mathrm{GL}_n(F)$, where $G_0$ is considered as the trivial group. The character $\nu_n:G_n \rightarrow \mathbb{C}^\times$ is defined by $\nu_n(g)=|\mathrm{det}(g)|_F$ for $g \in G_n$. For any integer $n \geq 0$, let $\mathrm{Rep}(G_n)$ be the category of smooth complex representations of $G_n$ of finite length and let $\mathrm{Irr}(G_n)$ be the set of irreducible objects of $\mathrm{Rep}(G_n)$ up to equivalence. For every integer $n \geq 1$, let $\mathrm{Irr}^c(G_n)$ be the set of irreducible supercuspidal representations of $G_n$. We set
 \[ ~\mathrm{Irr}=\bigsqcup\limits_{n\geq 0}\mathrm{Irr}(G_n), \text{ and } \mathrm{Irr}^c=\bigsqcup\limits_{n \geq 1}\mathrm{Irr}^c(G_n).\]
Let $P=LN$ be a standard parabolic subgroup of $G_n$, where the Levi subgroup $L$ is isomorphic to $G_{n_1} \times \cdots \times G_{n_r}$ for some composition $n=n_1+\cdots+n_r$. Let $\pi_i$ be a smooth representation of $G_{n_i}$ for $1 \leq i \leq r$ and let $\pi$ denote a smooth representation of $G_n$. The normalized parabolic-induced representation is denoted by
\[\pi_1 \times \cdots \times \pi_r = \mathrm{Ind}^{G_n}_{P} (\pi_1 \boxtimes \cdots \boxtimes \pi_r),\]
and the normalized Jacquet module of $\pi$ with respect to $P$ is denoted by
\[\mathrm{Jac}_N(\pi)= \frac{\delta^{-\frac{1}{2}}\cdot \pi}{\mathrm{span}\{ x \cdot v -v \mid x \in N, v \in \pi\}},\] where $\delta$ is the modular character of $P$. For $\pi \in \mathrm{Rep}(G_n)$, the socle of $\pi$, denoted by $\mathrm{soc}(\pi)$, is the maximal semisimple subrepresentation of $\pi$ and the cosocle of $\pi$, denoted by $\mathrm{cosoc}(\pi)$, is the maximal semisimple quotient of $\pi$. %For $\pi \in \mathrm{Rep}$, $\pi$ is called socle irreducible (abbreviated as `SI') if $\mathrm{soc}(\pi)$ is irreducible and occurs with multiplicity one in the semisimplification $[\pi]$. 
\subsection{Segments and multisegments}\label{sec:seg}
We recall the notion of segments and multisegments introduced in \cite{Zel}.  Let $a,b \in \mathbb{Z}$ such that $b-a \in \mathbb{Z}_{\geq 0}$ and let $\rho \in \mathrm{Irr}^c(G_k)$. A segment in the cuspidal line $\rho$ is denoted either by a void set or by $\Delta = [a,b]_\rho$, which is essentially the set $\{ \nu^a \rho, \nu^{a+1}\rho, ..., \nu^b \rho \}$ with the character $\nu=\nu_k$. The segment  $[a,a]_\rho$ is written as $[a]_\rho$. We denote $\mathrm{Seg}$ for the set of all segments and $\mathrm{Seg}_\rho$ for the set of segments in the cuspidal line $\rho$. We set $[a,a-1]_\rho=\emptyset$ for $a \in \mathbb{Z}$.  For a segment $\Delta = [a,b]_\rho$, the starting (or begining) element $ \nu^a \rho$ is denoted by $s(\Delta)$, and the ending element $\nu^b \rho$ is denoted by $e(\Delta)$. The relative length of $\Delta = [a,b]_\rho$ is denoted by $\ell_{rel}(\Delta)=b- a + 1$. By convention, the length of the void segment is $0$. Two non-void segments $\Delta$ and $\Delta^\prime$ are said to be linked if  $\Delta \nsubseteq \Delta^\prime$, $\Delta^\prime \nsubseteq \Delta$ and  $\Delta \cup \Delta^\prime$ remains a segment. If not, they are considered unlinked or not linked. For two linked segments $\Delta=[a,b]_\rho$ and $\Delta^\prime=[a^\prime,b^\prime]_\rho$,  $\Delta$ is said to precede $\Delta^\prime$ if $a < a^\prime$, $b < b^\prime$ and $a^\prime \leq b+1$. If $\Delta$ precedes $\Delta^\prime$, we denote $\Delta \prec \Delta^\prime$. If $\Delta = [a,b]_\rho$ is a non-void segment, we define \[\Delta^+ =[a,b+1]_\rho,~\Delta^- = [a,b-1]_\rho,~ {^+}\Delta =[a-1,b]_\rho, \text{ and }{^-}\Delta = [a+1,b]_\rho\] with
the convention that $\Delta^-$ and ${^-}\Delta$ are void if $a=b$.

A multisegment, denoted as $\mathfrak{m}=\{\Delta_1,...,\Delta_r\}$, a multiset of non-void segments and is represented as $\mathfrak{m}=\Delta_1+\cdots +\Delta_r$. Let $\mathrm{Mult}$ be the set of all multisegments and $\mathrm{Mult}_\rho$ be the set of those multisegments consisting of segments in the cuspidal line $\rho$. The relative length of a multisegment $\mathfrak{m} \in \mathrm{Mult}_\rho$ is defined by $\ell_{rel}(\mathfrak{m})=\sum_{\Delta \in \mathfrak{m}}\ell_{rel}(\Delta)$ and is $0$ if $\mathfrak{m}$ is void. For a multisegment $\mathfrak{m}$, the number of non-void segments in $\mathfrak{m}$ is denoted by $|\mathfrak{m}|$. The support of a multisegment
$\mathfrak{m}$ is the multiset of integers obtained by taking the union (with multiplicities) of the segments in $\mathfrak{m}$. For two multisegments $\mathfrak{m}, \mathfrak{m}^\prime \in \mathrm{Mult}$, we write $\mathfrak{m} + \mathfrak{m}^\prime$ for the union $\mathfrak{m} $ and $\mathfrak{m}^\prime$ counting multiplicities. For a segment $\Delta$, we set $\mathfrak{m}+\Delta=\mathfrak{m}+\{\Delta\}$ if $\Delta \neq \emptyset$, and $\mathfrak{m}+\Delta=\mathfrak{m}$ if $\Delta=\emptyset$. Similarly, we define $\mathfrak{m} - \mathfrak{m}^\prime$ and $\mathfrak{m}-\Delta$. For $\mathfrak{m} \in \mathrm{Mult}_\rho$ and $i \in \mathbb{Z}$, we define, $\mathfrak{m}[i]=\left\{[a,b]_\rho \in \mathfrak{m} \mid a=i \right\}$ and $\mathfrak{m}\langle i \rangle=\left\{[a,b]_\rho \in \mathfrak{m} \mid b=i \right\}$. For a multisegment $\mathfrak{m}=\Delta_1+ \cdots + \Delta_r \in \mathrm{Mult}_\rho$ (with all $\Delta_i\neq \emptyset$), we define
\[\mathfrak{m}^+=\Delta_1^++ \cdots + \Delta_r^+ \text{ and } \mathfrak{m}^-=\Delta_1^-+ \cdots + \Delta_r^-.\]

\subsection{Zelevinsky and Langlands classification}\label{sec:LZ_classification}
Let $\Delta=[a,b]_\rho$ be a segment for some $\rho \in \mathrm{Irr}^c$. The normalized parabolic-induced representation $\nu^a \rho \times \nu^{a+1}\rho \times \cdots \times \nu^b \rho $ has a unique irreducible submodule denoted by  $\langle \Delta \rangle=\mathrm{soc}\left(\nu^a \rho \times \nu^{a+1}\rho \times  \cdots \times \nu^b \rho \right)$, and a unique irreducible quotient denoted by the generalized Steinberg representation  \[\mathrm{St}(\Delta)=\mathrm{cosoc}\left(\nu^a \rho \times \nu^{a+1}\rho \times  \cdots \times \nu^b \rho \right).\]
\subsubsection{Zelevinsky classification}Consider an ordered multisegment $\mathfrak{m}=\Delta_1+\Delta_2+ \cdots+\Delta_r$ with $\Delta_i \nprec \Delta_j$ for $i < j$. Then, the normalized parabolic-induced representation $\zeta(\mathfrak{m})=\langle \Delta_1 \rangle \times \langle \Delta_2 \rangle \times \cdots \times \langle \Delta_r \rangle$ has a unique irreducible submodule, denoted by 
\[Z(\mathfrak{m})=\mathrm{soc}\left(\zeta(\mathfrak{m}) \right).\]
If $\pi$ is any irreducible smooth representation of $G_n$, there exists a unique multisegment $\mathfrak{m}$ such that $\pi$ is isomorphic to $Z(\mathfrak{m})$.
\subsubsection{Langlands classification}
Consider an ordered multisegment $\mathfrak{m}=\Delta_1+\Delta_2+ \cdots+\Delta_r$ with $\Delta_i \nprec \Delta_j$ for $i > j$. We denote the unique irreducible subrepresentation of the normalized parabolic-induced representation $\lambda(\mathfrak{m})=\mathrm{St}(\Delta_1) \times \mathrm{St}(\Delta_2) \times  \cdots \times \mathrm{St}(\Delta_r)$ by
\[L(\mathfrak{m})=\mathrm{soc}\left(\lambda(\mathfrak{m}) \right).\]
 Again, for any irreducible smooth representation $\pi$ of $G_n$, there exists a unique multisegment $\mathfrak{m}$ such that $\pi$ is isomorphic to $L(\mathfrak{m})$.
\subsubsection{Zelevinsky involution} \label{ss zel invol}
Le $\mathcal{R}(G_n)$ be the Grothendieck group of $\mathrm{Rep}(G_n)$ and put $\mathcal{R}= \bigoplus\limits_{n \geq 0} \mathcal{R}(G_n)$. The normalized parabolic induction gives a product map
\[\mathcal{R} \times \mathcal{R} \longrightarrow \mathcal{R} \text{ defined by } \left([\pi], [\pi^\prime]\right) \mapsto [\pi \times \pi^\prime],\]
which transforms $\mathcal{R}$ into a ring of polynomials in the indeterminates $\langle \Delta \rangle$ (as well as $\mathrm{St}(\Delta)$) for $\Delta \in \mathrm{Seg}$ over $\mathbb{Z}$. The map $\imath: \mathrm{Irr} \longrightarrow  \mathrm{Irr}$ defined by the involution $\imath: \mathrm{St}(\Delta) \mapsto \langle \Delta \rangle$ can be extended uniquely to a ring endomorphism $\imath:\mathcal{R} \longrightarrow \mathcal{R},$ such that $\imath$ is an involution and for any multisegment $\mathfrak{m} \in \mathrm{Mult}$, we have 
\[\imath(Z(\mathfrak{m}))\cong L(\mathfrak{m}) \text{ and }\imath(L(\mathfrak{m})) \cong Z(\mathfrak{m}).\]
In \cite{MW},  M\oe glin and Waldspurger provide an algorithm to compute the multisegment $\mathfrak{m}^\#$ associated to each multisegment $\mathfrak{m} \in \mathrm{Mult}$ such that \[Z(\mathfrak{m})\cong \imath(L(\mathfrak{m}))\cong L(\mathfrak{m}^\#) \text{ and } L(\mathfrak{m})\cong \imath(Z(\mathfrak{m})) \cong Z(\mathfrak{m}^\#).\]

\subsubsection{Gelfand-Kazhdan involution} \label{sss gk invol}

Let $\theta: G_n \rightarrow G_n$ given by $\theta(g)=g^{-T}$, the inverse transpose of $g$. This induces a covariant auto-equivalence, still denoted by $\theta$, on $\mathrm{Rep}(G_n)$. On the combinatorial side, we define $\theta: \mathrm{Seg}_{\rho}\rightarrow \mathrm{Seg}_{\rho^{\vee}}$ given by $\theta\left([a,b]_{\rho}\right)=[-b,-a]_{\rho^{\vee}}$. Define 
\[ \Theta: \mathrm{Mult}_{\rho}\rightarrow \mathrm{Mult}_{\rho^{\vee}}, \quad \Theta(\left\{ \Delta_1, \ldots, \Delta_k\right\})=\left\{ \theta(\Delta_1), \ldots, \theta(\Delta_k) \right\}. \]

Gelfand-Kazhdan showed that for $\pi \in \mathrm{Irr}$, $\theta(\pi)$ is isomorphic to the smooth dual of $\pi$. In particular, we have:
\[  \theta(L(\mathfrak m))=L(\Theta(\mathfrak m)), \quad \theta(Z(\mathfrak m))=Z(\Theta(\mathfrak m)) .
\]

Using the relation: for any $\pi_1 \in \mathrm{Rep}(G_{n_1})$ and $\pi_2 \in \mathrm{Rep}(G_{n_2})$, $\theta(\pi_1\times \pi_2) \cong \theta(\pi_2)\times \theta(\pi_1)$, one can relate left and right integrals/derivatives as follows:
\[  \mathrm I^{\mathrm L}_{[a,b]_{\rho}}(\theta(\pi)) \cong  \theta \left(I^{\mathrm R}_{[-b,-a]_{\rho^{\vee}}}(\pi )\right), \quad  \mathrm D^{\mathrm L}_{[a,b]_{\rho}}(\theta(\pi)) \cong  \theta \left(D^{\mathrm R}_{[-b,-a]_{\rho^{\vee}}}(\pi )\right).
\]
The above isomorphisms can be reformulated as: for $\mathfrak m \in \mathrm{Mult}_{\rho}$ and $[a,b]_\rho \in \mathrm{Seg}_\rho$,
\begin{align}
\mathrm I^{\mathrm L}_{[a,b]_{\rho}}(L(\mathfrak m)) \cong \theta(\mathrm I^{\mathrm R}_{[-b,-a]_{\rho^{\vee}}}(L(\Theta(\mathfrak m)))), ~  \mathrm D^{\mathrm L}_{[a,b]_{\rho}}(L(\mathfrak m)) \cong  \theta(\mathrm D^{\mathrm R}_{[-b,-a]_{\rho^{\vee}}}(L(\Theta(\mathfrak m) )))
\end{align}
\begin{align}
\mathrm I^{\mathrm L}_{[a,b]_{\rho}}(Z(\mathfrak m)) \cong \theta(\mathrm I^{\mathrm R}_{[-b,-a]_{\rho^{\vee}}}(Z(\Theta(\mathfrak m)))), ~  \mathrm D^{\mathrm L}_{[a,b]_{\rho}}(Z(\mathfrak m)) \cong  \theta(\mathrm D^{\mathrm R}_{[-b,-a]_{\rho^{\vee}}}(Z(\Theta(\mathfrak m) ))).
\end{align}
We shall use these later to deduce the algorithms from right derivatives/integrals to those for left derivatives/integrals.
 
 \subsubsection{Representations along a fixed cuspidal line} We fix a cuspidal representation $\rho \in \mathrm{Irr}^c$. Define the set of irreducible representations along the $\rho$-line by \[\mathrm{Irr}_\rho=\left\{\pi \in \mathrm{Irr} \mid \pi=L(\mathfrak{m}) \text{ for some } \mathfrak{m} \in \mathrm{Mult}_\rho\right\}.\]
 In other words, $\mathrm{Irr}_\rho$ consists of the  elements of $\mathrm{Irr}$ which are an irreducible quotient of $\nu^{a_1}\rho \times \nu^{a_2}\rho \times \cdots \times \nu^{a_r}\rho,$
 for some integers $a_1,a_2,...,a_r$. According to Zelevinsky \cite{Zel}, it is most interesting to study the parabolic inductions and Jacquet modules for representations in $\mathrm{Irr}_\rho$ for a fixed supercuspidal representation $\rho$. The general case of our results can be deduced from this.

\subsection{Highest derivative multisegment and removal process} \label{ss hd removal}
Fix an integer $c$. Let $\Delta=[c,d]_\rho$, and $\Delta^\prime=[c,d^\prime]_\rho$ be non-void segments. We define the ordering $\Delta \leq_c^a \Delta^\prime$ if $d \leq d^\prime$. A multisegment $\mathfrak{m}$ is said to be at the point $\nu^c \rho$ if every non-empty segment of $\mathfrak{m}$ is of the form $[c,d]_\rho$ for some $d \geq c$. For $\pi \in \mathrm{Irr}_\rho$, there exists a unique $\leq^a_c$-maximal multisegment $\mathfrak{h}_c$ at the point $\nu^c\rho$ such that $\mathrm{D}^\mathrm{R}_{\mathfrak{h}_c}(\pi) \neq 0$ (see \cite{Cha_csq} for the notion $\mathrm{D}^\mathrm{R}_{\mathfrak{h}_c}$). The highest derivative multisegment of $\pi$ is defined by $\mathfrak{hd}(\pi)=\sum_{c \in \mathbb{Z}} \mathfrak{h}_c.$ In \cite{Cha_csq, Cha_csq_iii}, the author shows that $\mathrm{D}^\mathrm{R}_{\mathfrak{hd}(\pi)}(\pi)=\pi^-$, the highest Bernstein-Zelevinsky derivative of $\pi$, where $\pi^-\cong Z(\mathfrak{m}^-)$ if $\pi=Z(\mathfrak{m})$ (see \cite{Zel} for more details) and for any $\mathfrak{n} \in \mathrm{Mult}_\rho$, there exists $\sigma \in \mathrm{Irr}_\rho$ such that $\mathfrak{hd}(\sigma)=\mathfrak{n}$.

\begin{lemma} \cite[Corollary 9.5]{Cha_csq} \label{lem non-zero der hd}
Let $\pi \in \mathrm{Irr}_{\rho}$ and let $[a,b]_{\rho} \in \mathrm{Seg}_{\rho}$. Then $\mathrm{D}^\mathrm{R}_{[a,b]_{\rho}}(\pi)\neq 0$ if and only if $\mathfrak{hd}(\pi)$ contains a segment of the form $[a,c]_{\rho}$ for some $c \geq b$.
\end{lemma}

To provide a proof of Lemmas \ref{lem:D[ab]=D[a+1,b]_D[a]_pi}, \ref{lem:comm_D[ab]_D[a]_pi} and \ref{lem:comm_D_ab:D_a+1_pi} later, we have utilized a combinatorial removal process, denoted as $\mathfrak{r}(\Delta, \pi)$ (see \cite[Definition 8.2]{Cha_csq}), and the first segment of the process, denoted as $\Upsilon(\Delta,\pi)$ for a segment $\Delta \in \mathrm{Seg}_\rho$ satisfying $\varepsilon^{\mathrm{R}}_{\Delta}(\pi)\neq 0$. Here, $\Upsilon(\Delta,\pi)$ is the shortest length segment among all the segments $\Delta' \in \mathfrak{hd}(\pi)$ such that $s(\Delta)=s(\Delta')$ and $e(\Delta)\leq e(\Delta')$. We also utilized the derivative resultant multisegment, denoted as $\mathfrak{r}(\mathfrak{n}, \pi)$, for a multisegment $\mathfrak{n} \in \mathrm{Mult}_\rho$ that is admissible to $\pi$. For further details, we refer to Section 8 in \cite{Cha_csq}, and we only mention few properties we frequently need:

\begin{lemma} \label{lem removal process singleton}
Let $\pi \in \mathrm{Irr}_{\rho}$ and $a \in \mathbb{Z}$. Then the following holds:
\begin{enumerate}
\item[(i)] $\Upsilon([a]_{\rho}, \pi) \in \mathfrak{hd}(\pi)[a]$; and 
\item[(ii)] $\mathfrak r([a]_{\rho}, \pi)=\mathfrak{hd}(\pi)-\Upsilon([a]_{\rho},\pi)+{}^-\Upsilon([a]_{\rho},\pi)$.
\end{enumerate}
\end{lemma}

\begin{proof}
This follows directly from the definition of the removal process in \cite[Definition 8.2]{Cha_csq}.
\end{proof}

The relation to derivatives is the following:
\begin{lemma} \cite[Theorem 9.3]{Cha_csq} \label{lem derivative removal}
Let $\pi \in \mathrm{Irr}_{\rho}$ and $[a,b]_{\rho}\in \mathrm{Seg}_{\rho}$. Then, for any $c \geq a$, \[\mathfrak{hd}(\mathrm{D}_{[a,b]_{\rho}}^\mathrm{R}(\pi))[c]=\mathfrak r([a,b]_{\rho}, \pi)[c].\]
\end{lemma}

%%%%%%%%%%%%%%%%%%%%%%%%%%%%%%%%% Derivatives in Langlands Classification %%%%%%%%%%%%%%%%%%%%%%%%%%%%%%
%
\section{Derivatives in Langlands classification} \label{sec:der_Lang}
In this section, we introduce an algorithm to calculate the derivatives of the irreducible representations of $\mathrm{GL}_n(F)$, wherein the representations are expressed in terms of Langlands data. The algorithm of Jantzen and M\'inguez serves as the initial step in an inductive argument to validate the Algorithm \ref{alg:der:Lang}.

\subsection{Algorithm for $\rho$-derivatives.}

We first state an algorithm for computing the right $\rho$-derivative of an irreducible representation in Langlands classification, which is already obtained in \cite{Jan, Min, LM16} in different words/terminology.

\subsubsection*{$\mathfrak{tds}$-process:} To illustrate an algorithm for the $\rho$-derivative, we initially propose a removal process for two linked segments (abbreviated as the $\mathfrak{tds}$-process) in $\mathfrak{m} \in \mathrm{Mult}_\rho$ for a fixed integer $c$. This process is executed through the following steps: 
\begin{itemize} 
\item[(i)] Select the longest segment $\Delta^{\prime \prime}$ from $\mathfrak{m}[c+1]$.
\item[(ii)] Choose the longest segment $\Delta^{\prime}$ from $\mathfrak{m}[c]$ such that $\Delta^{\prime}$ precedes $\Delta^{\prime \prime}$. \item[(iii)] If both $\Delta^{\prime}$ and $\Delta^{\prime\prime}$ exist, remove them to define a new multisegment as $\mathfrak{tds}(\mathfrak{m}, c)= \mathfrak{m}-\Delta^{\prime} - \Delta^{\prime \prime}.$ 
 \end{itemize}
We say that $\Delta^{\prime}$ (resp. $\Delta^{\prime\prime}$) participates in (the removal step of) the $\mathfrak{tds}(-, c)$ process on $\mathfrak{m}$.

\begin{algorithm}[Right $\nu^a\rho$-derivative]\label{alg:rho_der:Lang}
Suppose $\mathfrak{m} \in \mathrm{Mult}_{\rho}$ and $a \in \mathbb{Z}$. Define a new multisegment $\mathcal{D}^\mathrm{Lan}_{[a]_{\rho}} (\mathfrak{m})$ by the following steps:

Step 1. Set $\mathfrak{m}_{0}=\mathfrak{m}$ and recursively define $\mathfrak{m}_{i}= \mathfrak{tds}(\mathfrak{m}_{i-1}, a)$ until the process terminates. Suppose this $\mathfrak{tds}(-,a)$ process terminates after $k$ times and the final multisegment is $\mathfrak{m}_{k}$.

Step 2. Choose the shortest segment ${\Delta}_* \in \mathfrak{m}_{k}[a]$ (if it exists) and define the multisegment 
\begin{equation}\label{eq:rho_der}
\mathcal{D}^\mathrm{Lan}_{[a]_{\rho}} (\mathfrak{m}) :=\mathfrak{m}- {\Delta}_* + {^-}{\Delta}_*.    
\end{equation}
If such segment ${\Delta}_*$ does not exist, we write 
$\mathcal{D}^\mathrm{Lan}_{[a]_{\rho}} (\mathfrak{m}) :=\infty.$ 
\end{algorithm}

\begin{example}
Let $\mathfrak{m}=\left\{ [0,4]_\rho, [1,5]_\rho, [1,4]_\rho, [1,3]_\rho, [1,2]_\rho, [2,5]_\rho,[2,3]_\rho \right\}$ and $a=1$. Then $\mathfrak{m}_1=\mathfrak{tds}(\mathfrak{m},1)=\mathfrak{m}-[1,4]_\rho-[2,5]_\rho$, and $\mathfrak{m}_2=\mathfrak{tds}(\mathfrak{m}_1,1)=\mathfrak{m}_1-[1,2]_\rho-[2,3]_\rho$.
The $\mathfrak{tds}(-,1)$ process terminates on $\mathfrak{m}_2$, and  $[1,3]_\rho$ is the shortest segment in $\mathfrak{m}_2[1]=\left\{[1,5]_\rho,[1,3]_\rho\right\}$. Therefore, we have 
\[\mathcal{D}^\mathrm{Lan}_{[1]_\rho} (\mathfrak{m}) =\mathfrak{m} - [1,3]_\rho+[2,3]_\rho =\left\{ [0,4]_\rho, [1,5]_\rho, [1,4]_\rho, [2,3]_\rho, [1,2]_\rho, [2,5]_\rho,[2,3]_\rho \right\}.\] 
If $a=0$, we have $\mathfrak{tds}(\mathfrak{m}, a)=\mathfrak{m}-[0,4]_\rho-[1,5]_\rho$, and so $\mathcal{D}^\mathrm{Lan}_{[0]_\rho} (\mathfrak{m}) = \infty$ since $\mathfrak{tds}(\mathfrak{m}, a)[0]=\emptyset$.
\end{example}

\begin{theorem}(cf. \cite[Theorem 2.2.1]{Jan}, \cite[Th\'eor\`eme 7.5]{Min})\label{thm:der:Lang:length=1}
 Suppose $\mathfrak{m} \in \mathrm{Mult}_\rho$ and $a \in \mathbb{Z}$. Then, the right $\nu^a\rho$-derivative of $L(\mathfrak{m})$ is given by:
\[\mathrm{D}^\mathrm{R}_{[a]_\rho}(L(\mathfrak{m})) \cong \begin{cases}
  L \left( \mathcal{D}_{[a]_\rho}^\mathrm{Lan}(\mathfrak{m})\right) &\mbox{ if }  \mathcal{D}_{[a]_\rho}^\mathrm{Lan}(\mathfrak{m}) \neq \infty\\
  0 &\mbox{ otherwise} .
\end{cases}\] 
%and the highest right $\nu^a\rho$-derivative
%\[\left(\mathrm{D}^\mathrm{R}_{[a]_\rho}\right)^{\varepsilon^\mathrm{R}_{[a]_\rho}(L(\mathfrak{m}))}(L(\mathfrak{m})) \cong \begin{cases}
%  L \left( \left(\mathcal{D}^\mathrm{Lan}_{\nu^a\rho} \right)^{r_a(\mathfrak{m})} (\mathfrak{m})\right) &\mbox{ if }  \mathcal{D}_{[a]_\rho}^\mathrm{Lan}(\mathfrak{m}) \neq \infty\\
%  0 &\mbox{ otherwise}.
%\end{cases}\]
\end{theorem}

\subsection{The number $\varepsilon^\mathrm{R}_{[a]_{\rho}}(\mathfrak m)$}

Recall that, for $\pi \in \mathrm{Irr}_{\rho}$, the number $\varepsilon^{\mathrm{R}}_{[a]_{\rho}}(\pi)$ is defined in Section \ref{ss application}. For $\mathfrak m \in \mathrm{Mult}_{\rho}$ and $a \in \mathbb{Z}$, we define $\varepsilon^{\mathrm{R}}_{[a]_{\rho}}(\mathfrak m)=\varepsilon^{\mathrm{R}}_{[a]_{\rho}}(L(\mathfrak m))$. We give a combinatorial characterization on the number $\varepsilon^{\mathrm{R}}_{[a]_{\rho}}$:

\begin{lemma} \label{lem combinatorial varepsilon}
Let $\mathfrak m \in \mathrm{Mult}_{\rho}$. We use the notations in Algorithm \ref{alg:rho_der:Lang}. Then $\varepsilon^{\mathrm{R}}_{[a]_{\rho}}(\mathfrak m)$ is equal to $|\mathfrak m_k[a]|$ i.e. the number of segments in $\mathfrak m_k$ starting with $\nu^a\rho$.
\end{lemma}

\begin{proof}
It follows by applying Theorem \ref{thm:der:Lang:length=1} multiple times.
\end{proof}

\subsection{Algorithm for $\mathrm{St}$-derivatives}\label{subsec:der:lang}
To establish an algorithm for derivatives, we need to define the following upward sequence of maximally linked segments to arrange the segments of the multisegment corresponding to the given irreducible representation.

{\it Upward sequence $\underline{\mathfrak{Us}}$:} We define the upward sequence of (maximally linked) segments in a multisegment $\mathfrak{n} \in \mathrm{Mult}_\rho$ as follows: identify the smallest number $a_{1}$ for which $\mathfrak{n}[a_{1}] \neq \emptyset$ and choose the longest segment $\Delta_{1}\in \mathfrak{n}[a_{1}]$. Recursively for $j \geq 2$, find the smallest number $a_{j}$ (if it exists) such that $a_{j-1} < a_{j}$ and there exists a segment $\Delta^\prime_j \in \mathfrak{n}[a_{j}]$ with $\Delta_{j-1} \prec \Delta^\prime_j$. Then, we pick a longest segment $\Delta_{j}\in \mathfrak{n}[a_{j}]$ such that $\Delta_{j-1} \prec \Delta_j$. This process terminates after a finite number of steps, say $r$, and  $\Delta_{1}, \Delta_{2}, \ldots, \Delta_{r}$ are all obtained in this process. We then define the following upward sequence:
\[\underline{\mathfrak{Us}}(\mathfrak{n}):=\left\{\Delta_{1} , \Delta_{2}, \ldots ,\Delta_{r}\right\}=\Delta_1 + \Delta_2+\ldots + \Delta_r.\]
We also regard $\underline{\mathfrak{Us}}(\mathfrak n)$ as a multisegment.

\begin{algorithm}\label{alg:der:Lang}
Given $\mathfrak{m} \in \mathrm{Mult}_\rho$ and $\Delta=[a,b]_\rho \in \mathrm{Seg}_\rho$, we consider the following multisegment \[\mathfrak{m}_{[a,b]}:=\{[a^\prime, b^\prime]_\rho \in \mathfrak{m} \mid a \leq a^\prime \leq b+1 \leq b^\prime +1 \}.\] 

Step 1. (Arrange upward sequences): Set $\mathfrak{m}_1=\mathfrak{m}_{[a,b]}$ and let $\underline{\mathfrak{Us}}(\mathfrak{m}_1)= \{\Delta_{1,1}, \Delta_{1,2}, \ldots, \Delta_{1,r_1}\}$ be the upward sequence of maximally linked segments on $\mathfrak{m}_{1}$ with $\Delta_{1,j} \prec \Delta_{1,j+1}$. Recursively for $2 \leq i \leq k$, we define $\mathfrak{m}_{i}= \mathfrak{m}_{i-1} - \underline{\mathfrak{Us}}(\mathfrak{m}_{i-1})$ and the corresponding upward sequence by
\[\underline{\mathfrak{Us}}(\mathfrak{m}_i)= \{\Delta_{i,1} , \Delta_{i,2} , \ldots, \Delta_{i,r_i}\} \text{ with }\Delta_{i,j} \prec \Delta_{i,j+1}.\]
Here $k$ is the smallest integer for which $\mathfrak{m}_{k+1}=\emptyset$. 

Step 2. (Removable free points): Denote $\Delta_{i,j}=\left[a_{i,j},b_{i,j}\right]_\rho$. We define the `removable free' section for the segment $\Delta_{i,j}$ for each $1 \leq i \leq k$ as:
\begin{align}\label{eqn removable free}
\mathfrak{rf}\left(\Delta_{i,j} \right)= \begin{cases}
    \left[a_{i,j},~ a_{i,j+1}-2 \right]_\rho &\mbox{ if } 1 \leq j < r_i\\
    \Delta_{i,r_i} &\mbox{ if } j=r_i.
\end{cases}
\end{align}
Here, $\mathfrak{rf}\left(\Delta_{i,j}\right)=\emptyset$ if $a_{i,j}> a_{i,j+1}-2$. For $y \in \mathbb{Z}$, we call $[y]_\rho$ a `removable free point' of $\Delta_{i,j} $  if $x \leq y \leq z$, where $ \mathfrak{rf}\left(\Delta_{i,j}\right)=[x,z]_\rho$.

Step 3. (Selection): We then select some segments $\Delta_{i,j}$ in the following way: 
\begin{enumerate}
    \item[(i)] Choose a segment $\Delta_{i_1,j_1} \in \mathfrak{m}_1$ (if it exists) where $i_1$ is the largest integer in $\{1,...,k\}$ such that $[a_{i_1,j_1},b]_\rho \subseteq \mathfrak{rf}\left(\Delta_{i_1,j_1} \right)$ for some $j_1 \in \left\{ 1,...,r_{i_1} \right\}$.
    \item[(ii)] Recursively for $t \geq 2$, Choose a segment $\Delta_{i_t,j_t} \in \mathfrak{m}_1$ (if it exists), where $i_t$ is the largest integer in $\{1,...,i_{t-1}\}$ such that 
    \begin{align} \label{eqn der algorithm free condition}
    \left[a_{i_t,j_t}, a_{i_{t-1}, j_{t-1}}-1\right]_\rho \subseteq \mathfrak{rf}\left(\Delta_{i_t,j_t} \right).
    \end{align}
    \item[(iii)] This process terminates (when no further such segment can be found) after a finite number of steps and suppose $\Delta_{i_\ell, j_\ell}$ is the last segment of the process.
\end{enumerate}

Step 4. (Truncation):  If $a_{i_\ell, j_\ell} = a,$ then we define new left truncated segments as follows:
\begin{equation*}%\label{trc_alg_der_lang}
\Delta_{i_1, j_1}^\mathrm{trc} = \left[b+1, ~ b_{i_1, j_1}\right]_\rho, \text{ and } \Delta_{i_t, j_t}^\mathrm{trc}= \left[a_{i_{t-1},j_{t-1}}, ~ b_{i_t, j_t}\right]_\rho ~ \text{ for } 2 \leq i \leq \ell. 
\end{equation*}
As convention, $[c, c-1]_\rho = \emptyset$. Then, the right derivative multisegment in the Langalands classification is defined by 
\begin{equation} \label{alg_der_lang}
\mathcal{D}_{[a,b]_\rho}^\mathrm{Lang}(\mathfrak{m}) = \mathfrak{m}- \sum\limits_{t=1}^{\ell} \Delta_{i_t, j_t} + \sum\limits_{t=1}^{\ell} \Delta_{i_t, j_t}^\mathrm{trc}.
\end{equation}
We shall call those segments $\Delta_{i_1,j_1}, \ldots, \Delta_{i_{\ell},j_{\ell}}$ participate in the truncation process for $\mathcal D^{\mathrm{Lang}}_{[a,b]_{\rho}}(\mathfrak m)$.

Step 4'. If $\Delta_{i_1,j_1}$ does not exist, or $a_{i_\ell, j_\ell} \neq a,$ we write \[\mathcal{D}_{[a,b]_{\rho}}^\mathrm{Lang}(\mathfrak{m}) = \infty.\]
\end{algorithm}
\begin{remark}\label{rem:der_equality}
    Note that $\mathcal{D}^\mathrm{Lang}_{[a]_\rho}(\mathfrak{m})=\mathcal{D}^\mathrm{Lan}_{[a]_\rho}(\mathfrak{m})$ for any $\mathfrak{m} \in \mathrm{Mult}_\rho$. We use the algorithm of $\mathcal{D}^\mathrm{Lan}_{[a]_\rho}(\mathfrak{m})$ to get $\mathcal{D}^\mathrm{Lang}_{[a]_\rho}(\mathfrak{m})$ without mentioning it further.
\end{remark}

\begin{example}
Let $\mathfrak{m}=\left\{ [0,5]_\rho, [0,4]_\rho, [1,2]_\rho, [2,6]_\rho,[2,3]_\rho \right\}$. Then, we have the following $\mathcal{D}^\mathrm{Lang}_\Delta(\mathfrak m )$ for various $\Delta$:

(i) Suppose $\Delta=[0,2]_\rho$.  Then, $\mathfrak{m}_1=\mathfrak{m}_{[0,2]}=\mathfrak{m}$ with $\underline{\mathfrak{Us}}(\mathfrak{m}_1)=\left\{[0,5]_\rho,[2,6]_\rho\right\}$:
\begin{align*}
   \xymatrix{  &    & \stackrel{2}{{\color{blue}{\bullet}}} \ar@{-}[r]   & \stackrel{3}{{\color{blue}{\bullet}}} \ar@{-}[r]  & \stackrel{4}{{\color{blue}{\bullet}}} \ar@{-}[r]  & \stackrel{5}{{\color{blue}{\bullet}}} \ar@{-}[r]  & \stackrel{6}{{\color{blue}{\bullet}}}\\
  \stackrel{0}{{\color{blue}{\bullet}}} \ar@{-}[r]   & \stackrel{1}{\bullet} \ar@{-}[r]   & \stackrel{2}{\bullet} \ar@{-}[r]   & \stackrel{3}{\bullet} \ar@{-}[r]  & \stackrel{4}{\bullet}\ar@{-}[r]  & \stackrel{5}{{{\bullet}}} & }
\end{align*}
$\mathfrak{m}_2=\mathfrak{m}_1- \underline{\mathfrak{Us}}(\mathfrak{m}_1)=\left\{ [0,4]_\rho, [1,2]_\rho, [2,3]_\rho \right\}$ with $\underline{\mathfrak{Us}}(\mathfrak{m}_2)=\left\{[0,4]_\rho\right\}$:  
\begin{align*}
 \xymatrix{  
  \stackrel{0}{{\color{red}{\bullet}}} \ar@{-}[r]   & \stackrel{1}{{\color{red}{\bullet}}} \ar@{-}[r]   & \stackrel{2}{{\color{blue}{\bullet}}} \ar@{-}[r]   & \stackrel{3}{{\color{blue}{\bullet}}} \ar@{-}[r]  & \stackrel{4}{{\color{blue}{\bullet}}} & & }
  \end{align*}
and $\mathfrak{m}_3=\mathfrak{m}_2- \underline{\mathfrak{Us}}(\mathfrak{m}_2)=\left\{[1,2]_\rho, [2,3]_\rho \right\}$ with $\underline{\mathfrak{Us}}(\mathfrak{m}_3)=\left\{[1,2]_\rho, [2,3]_\rho\right\}$: 
\begin{align*}
\xymatrix{  &    & \stackrel{2}{{\color{red}{\bullet}}} \ar@{-}[r]   & \stackrel{3}{{\color{blue}{\bullet}}}   &   &  & \\
    & \stackrel{1}{\bullet} \ar@{-}[r]   & \stackrel{2}{\bullet}  &  &   &  & }
\end{align*}
The blue and red points in the graphs represent the removable free points of the segments and the red points in the graphs represent those removable free points to be removed to get the derivative $\mathcal{D}^\mathrm{Lang}_{[0,2]_\rho}(\mathfrak m ) =\left\{[0,5]_\rho, [2,4]_\rho, [1,2]_\rho, [2,6]_\rho,[3]_\rho \right\}.$

(ii) Suppose $\Delta=[0,3]_\rho$.  Then $\mathfrak{m}_1=\mathfrak{m}_{[0,3]}=\left\{[0,5]_\rho, [0,4]_\rho, [2,6]_\rho,[2,3]_\rho \right\}$ with $\underline{\mathfrak{Us}}(\mathfrak{m}_1)=\left\{[0,5]_\rho,[2,6]_\rho\right\}$:
\begin{align*}
   \xymatrix{  &    & \stackrel{2}{{\color{blue}{\bullet}}} \ar@{-}[r]   & \stackrel{3}{{\color{blue}{\bullet}}} \ar@{-}[r]  & \stackrel{4}{{\color{blue}{\bullet}}} \ar@{-}[r]  & \stackrel{5}{{\color{blue}{\bullet}}} \ar@{-}[r]  & \stackrel{6}{{\color{blue}{\bullet}}}\\
  \stackrel{0}{{\color{blue}{\bullet}}} \ar@{-}[r]   & \stackrel{1}{\bullet} \ar@{-}[r]   & \stackrel{2}{\bullet} \ar@{-}[r]   & \stackrel{3}{\bullet} \ar@{-}[r]  & \stackrel{4}{\bullet}\ar@{-}[r]  & \stackrel{5}{{{\bullet}}} & }
\end{align*}
$\mathfrak{m}_2=\mathfrak{m}_1- \underline{\mathfrak{Us}}(\mathfrak{m}_1)=\left\{ [0,4]_\rho, [2,3]_\rho \right\}$ with $\underline{\mathfrak{Us}}(\mathfrak{m}_2)=\left\{[0,4]_\rho\right\}$:  
\begin{align*}
 \xymatrix{  
  \stackrel{0}{{\color{red}{\bullet}}} \ar@{-}[r]   & \stackrel{1}{{\color{red}{\bullet}}} \ar@{-}[r]   & \stackrel{2}{{\color{blue}{\bullet}}} \ar@{-}[r]   & \stackrel{3}{{\color{blue}{\bullet}}} \ar@{-}[r]  & \stackrel{4}{{\color{blue}{\bullet}}} & & }
  \end{align*}
and $\mathfrak{m}_3=\mathfrak{m}_2- \underline{\mathfrak{Us}}(\mathfrak{m}_2)=\left\{[2,3]_\rho \right\}$ with $\underline{\mathfrak{Us}}(\mathfrak{m}_3)=\left\{[2,3]_\rho\right\}$: 
\begin{align*}
\xymatrix{  &    & \stackrel{2}{{\color{red}{\bullet}}} \ar@{-}[r]   & \stackrel{3}{{\color{red}{\bullet}}}   &   &  & }
\end{align*}
The blue and red points in the graphs represent the removable free points of the segments and the red points in the graphs represent the removable free points to be removed to get the derivative $\mathcal{D}^\mathrm{Lang}_{[0,3]_\rho}(\mathfrak m )=\left\{ [0,5]_\rho, [2,4]_\rho, [1,2]_\rho, [2,6]_\rho\right\}.$

(iii) Let $\Delta=[0,5]_\rho$. Then, $\mathfrak{m}_1=\mathfrak{m}_{[0,5]}=\left\{[0,5]_\rho, [2,6]_\rho\right\}=\underline{\mathfrak{Us}}(\mathfrak{m}_1)$. Here, $[1]_\rho$ is not a removable free point of any segment in $\underline{\mathfrak{Us}}(\mathfrak{m}_1)$. Therefore, $\mathcal{D}^\mathrm{Lang}_{[0,5]_\rho}(\mathfrak m )= \infty$. 
\end{example}

The following simple observation is sometimes useful:

\begin{lemma} \label{lem either prec or overlap}
Let $\mathfrak m\in \mathrm{Mult}_{\rho}$ and let $[a,b]_{\rho}\in \mathrm{Seg}_{\rho}$. For any $\Delta, \Delta' \in \mathfrak m_{[a,b]}$, either one of the following hold:
\begin{enumerate}
    \item $\Delta \prec \Delta'$ or $\Delta' \prec \Delta$; or 
    \item $\Delta \subset \Delta'$ or $\Delta'\subset \Delta$.
\end{enumerate}
\end{lemma}

\begin{proof}
This is straightforward from the definition of $\mathfrak m_{[a,b]}$.
\end{proof}

%We also record one simple but useful consequence of the choice of the set $\mathfrak m_{[a,b]}$:

%\begin{lemma}
%Let $\mathfrak m \in \mathrm{Mult}_{\rho}$ and let $[a,b]_{\rho} \in \mathrm{Seg}_{\rho}$. Let $\Delta, \Delta' \in \mathfrak m_{[a,b]}$. Then one of the following holds:
%\begin{enumerate}
% \item $\Delta \subset \Delta'$ or $\Delta' \subset \Delta$; or
% \item $\Delta \prec \Delta'$ or $\Delta' \prec \Delta$.
%\end{enumerate}
%\end{lemma}

%\begin{proof}
%This follows from $\nu^b \rho \in \Delta \cap \Delta'$ and so $\Delta \cap \Delta'\neq \emptyset$.
%\end{proof}

\subsection{Composition of $\mathcal D_{[a+1,b]_{\rho}}^{\mathrm{Lang}}$ and $\mathcal D_{[a]_{\rho}}^{\mathrm{Lang}}$}

In the following lemmas, we shall compare properties of derivatives on algorithms and derivatives on the representation theory side. For algorithm side, we have to investigate the change of upward sequences between $\mathfrak m$ and $\mathcal D^{\mathrm{Lang}}_{[a]_{\rho}}(\mathfrak m)$, and it turns out the changes (see (\ref{eq:Us_a+1_b}), (\ref{eq:us_on_rho_der})) are reasonably simple to show some properties of derivatives. For the representation-theoretic side, one utilizes Section \ref{ss hd removal}.

\begin{lemma} \label{upward seq a a+1,b}
Let $\mathfrak m \in \mathrm{Mult}_{\rho}$  and $[a,b]_\rho \in \mathrm{Seg}_\rho$ with $a < b$. Suppose $\mathcal D^{\mathrm{Lang}}_{[a,b]_{\rho}}(\mathfrak{m})\neq \infty$ and $\varepsilon_{[a]_{\rho}}^{\mathrm{R}}(\mathfrak m)=1$. We use the notations in Algorithm \ref{alg:der:Lang} for $\mathcal D^{\mathrm{Lang}}_{[a,b]_{\rho}}(\mathfrak m)$. Let $k_0$ be the largest integer such that $\underline{\mathfrak{Us}}(\mathfrak m_{k_0})[a]\neq \emptyset$. Define $\mathfrak n_1=\mathfrak n_{[a+1,b]}$ and recursively define $\mathfrak n_{i+1}=\mathfrak n_i-\underline{\mathfrak{Us}}(\mathfrak n_i)$. Then 
\begin{align}\label{eq:Us_a+1_b}
 \underline{\mathfrak{Us}}(\mathfrak{n}_i)=\begin{cases}
     \underline{\mathfrak{Us}}(\mathfrak{m}_i)-\Delta_{i,1} &\mbox{ if } 1 \leq i \leq k_0 \text{ and } i \neq i_\ell\\
     \underline{\mathfrak{Us}}(\mathfrak{m}_i) - \Delta_{i,1} + {^-}\Delta_{i,1} &\mbox{ if } i = i_\ell\\
     \underline{\mathfrak{Us}}(\mathfrak{m}_i) &\mbox{ if } k_0+1 \leq i \leq k.
 \end{cases}   
\end{align}

\end{lemma}

\begin{proof}
It is a book-keeping exercise. Note that $\mathcal D^{\mathrm{Lang}}_{[a,b]_{\rho}}(\mathfrak{m})\neq \infty$ and $\varepsilon_{[a]_{\rho}}^{\mathrm{R}}(\mathfrak m)=1$ imply that $a_{i,2}=a+1$ for $1 \leq i \leq k_0$ except for $i=i_\ell$.
\end{proof}

\begin{example}\label{ex:demo_1}
Let $\mathfrak{m}=\left\{[1,5]_\rho, [1,4]_\rho, [1,3]_\rho, [2,6]_\rho, [2,4]_\rho, [2,3]_\rho, [3,7]_\rho, [3,5]_\rho, [3,4]_\rho \right\}$ with $a=1$ and $b=3$. Then, $\mathfrak{n}=\mathcal{D}_{[1]_\rho}^\mathrm{Lang}(\mathfrak{m})=\mathfrak{m}-[1,4]_\rho + [2,4]_\rho$ and  $\mathcal{D}_{[1]_\rho}^\mathrm{Lang}(\mathfrak{n})=\infty$. Set $\mathfrak{m}_1=\mathfrak{m}_{[a,b]}$ (resp. $\mathfrak{n}_1=\mathfrak{n}_{[a+1,b]}$) and recursively for $i>1$, $\mathfrak{m}_i=\mathfrak{m}_{i-1} - \underline{\mathfrak{Us}}(\mathfrak m_{i-1})$ (resp. $\mathfrak{n}_i=\mathfrak{n}_{i-1} - \underline{\mathfrak{Us}}(\mathfrak n_{i-1})$). Then, $\underline{\mathfrak{Us}}(\mathfrak{m}_{1})=\left\{[1,5]_\rho, [2,6]_\rho,[3,7]_\rho\right\}$, $\underline{\mathfrak{Us}}(\mathfrak{m}_{2})=\left\{[1,4]_\rho, [3,5]_\rho\right\}$, $\underline{\mathfrak{Us}}(\mathfrak{m}_{3})=\left\{[1,3]_\rho, [2,4]_\rho\right\}$, and $\underline{\mathfrak{Us}}(\mathfrak{m}_{4})=\left\{[2,3]_\rho, [3,4]_\rho\right\}$. On the other hand, $\underline{\mathfrak{Us}}(\mathfrak{n}_{1})=\left\{[2,6]_\rho,[3,7]_\rho\right\}$, $\underline{\mathfrak{Us}}(\mathfrak{n}_{2})=\left\{[2,4]_\rho, [3,5]_\rho\right\}$, $\underline{\mathfrak{Us}}(\mathfrak{n}_{3})=\left\{[2,4]_\rho\right\}$, and $\underline{\mathfrak{Us}}(\mathfrak{n}_{4})=\left\{[2,3]_\rho, [3,4]_\rho\right\}$. 
\end{example}

\begin{lemma}\label{lem:D[ab]=D[a+1,b]_D[a]}
Let $\mathfrak{m}\in \mathrm{Mult}_\rho$ and $[a,b]_\rho \in \mathrm{Seg}_\rho$ with $b>a$. Suppose $\varepsilon^\mathrm{R}_{[a]_\rho}(\mathfrak{m}) =1$.
\begin{enumerate}
\item[(i)] If $\mathcal{D}_{[a,b]_\rho}^\mathrm{Lang}(\mathfrak{m}) \neq \infty$, we have $\mathcal{D}^\mathrm{Lang}_{[a+1,b]_\rho} \circ \mathcal{D}_{[a]_\rho}^\mathrm{Lang}(\mathfrak{m}) =  \mathcal{D}_{[a,b]_\rho}^\mathrm{Lang}(\mathfrak{m}).$  
\item[(ii)] If $\varepsilon^\mathrm{R}_{[a+1]_{\rho}}(\mathfrak m)=0$ and $\mathcal D^\mathrm{Lang}_{[a+1,b]_{\rho}}\circ \mathcal D^\mathrm{Lang}_{[a]_{\rho}}(\mathfrak m)\neq \infty$, we have  \[ \mathcal{D}_{[a,b]_\rho}^\mathrm{Lang}(\mathfrak{m}) = \mathcal D^\mathrm{Lang}_{[a+1,b]_{\rho}}\circ \mathcal D^\mathrm{Lang}_{[a]_{\rho}}(\mathfrak m).\]
\end{enumerate}
\end{lemma}

\begin{example}
We continue Example \ref{ex:demo_1}. We have $\varepsilon^\mathrm{R}_{[1]_\rho}(\mathfrak{m})=1$ and $\varepsilon^\mathrm{R}_{[2]_\rho}(\mathfrak{m})=0$. Observe that $\mathcal{D}_{[a,b]_\rho}^\mathrm{Lang}(\mathfrak{m})=\mathfrak{m}-[1,4]_\rho-[2,4]_\rho-[3,4]_\rho+{^-}[1,4]_\rho +{^-}[2,4]_\rho + {^-}[3,4]_\rho$ and $\mathcal{D}_{[a+1,b]_\rho}^\mathrm{Lang}(\mathfrak{n})=\mathfrak{n}-[2,4]_\rho-[3,4]_\rho + {^-}[2,4]_\rho+ {^-}[3,4]_\rho=\left\{[1,5]_\rho, [2,4]_\rho, [1,3]_\rho, [2,6]_\rho, [2,3]_\rho, [3,4]_\rho, [4]_\rho, [3,7]_\rho, [3,5]_\rho \right\}$, which is equal to $\mathcal{D}_{[a,b]_\rho}^\mathrm{Lang}(\mathfrak{m})$.
\end{example}

\noindent
{\it Proof of Lemam \ref{lem:D[ab]=D[a+1,b]_D[a]}.}
Let's assume all the notations as mentioned in Algorithm \ref{alg:der:Lang} applied for $\mathcal D^{\mathrm{Lang}}_{[a,b]_{\rho}}(\mathfrak m)$. %We fix an integer $k_0$ with $1 \leq k_0 \leq k$ such that $a_{i,1}=a$ for $1 \leq i \leq k_0 $ and $ a_{i,1}\neq a $ for $  i \geq k_0+1.$ 

(i) As $\mathcal{D}_{[a,b]_\rho}^\mathrm{Lang}(\mathfrak{m}) \neq \infty$ with $[a]_\rho \subset \mathfrak{rf}\left(\Delta_{i_\ell,1} \right)$ (here $j_\ell =1$) and $\varepsilon^\mathrm{R}_{[a]_\rho}(\mathfrak{m}) =1$, we have
\begin{equation}\label{eq:n=der_a_m}
 \mathfrak{n}:=\mathcal{D}_{[a]_\rho}^\mathrm{Lang}(\mathfrak{m})=\mathfrak{m}-\Delta_{i_\ell,1}+ {^-}\Delta_{i_\ell,1}.   
\end{equation}
Let $\mathfrak{n}_1=\mathfrak{n}_{[a+1,b]}$. By Lemma \ref{upward seq a a+1,b}, we have   $\mathfrak{n}_1= \mathfrak{m}_1-\mathfrak{m}_1[a] + {^-}\Delta_{i_\ell,1}.$ For $i>1$, define $\mathfrak{n}_{i}=\mathfrak{n}_{i-1}-\underline{\mathfrak{Us}}(\mathfrak{n}_{i-1})$. By (\ref{eqn removable free}) and (\ref{eq:Us_a+1_b}), one sees that the segments participating in the truncation process for $\mathcal D^{\mathrm{Lang}}_{[a+1,b]_{\rho}}(\mathfrak n)$ and for $\mathcal D^{\mathrm{Lang}}_{[a,b]_{\rho}}(\mathfrak m)$ agree except that
\begin{enumerate}
    \item[(a)] $\Delta_{i_{\ell},1}$ participates in the truncation process for $\mathcal D^{\mathrm{Lang}}_{[a,b]_{\rho}}(\mathfrak m)$;
    \item[(b)] ${}^-\Delta_{i_{\ell},1}$ participates in the truncation process for $\mathcal D^{\mathrm{Lang}}_{[a+1,b]_{\rho}}(\mathfrak{n})$ if and only if any segment participating in the truncation process for $\mathcal D^{\mathrm{Lang}}_{[a,b]_{\rho}}(\mathfrak m)$ does not start with $\nu^{a+1}\rho$. 
    %For the only if direction, if a segment in $\mathfrak{m}_1$ starting with $\nu^{a+1}\rho$, participates in the truncation process for $\mathcal D^{\mathrm{Lang}}_{[a,b]_{\rho}}(\mathfrak{m})$, that segment is $\Delta_{i_{\ell-1}, j_{\ell-1}}$, and then, the fact $i_{\ell}<i_{\ell-1}$ and (\ref{eq:Us_a+1_b}) ensure that $\Delta_{i_{\ell-1}, j_{\ell-1}}$ also participates in the truncation process for $\mathcal D^{\mathrm{Lang}}_{[a+1,b]_{\rho}}(\mathfrak n)$. For the if direction, we first observe \[ \mathfrak{rf}({}^-\Delta_{i_{\ell},1})\cup \left\{ \nu^a\rho\right\} =\mathfrak{rf}(\Delta_{i_\ell,1}) ,\]
%where the first $\mathfrak{rf}$ is considered for $\mathfrak n$ and the second $\mathfrak{rf}$ is considered for $\mathfrak m$. Since $\mathfrak{rf}(\Delta_{i_\ell,1})$ (considered for $\mathfrak m$) satisfies the condition  (\ref{eqn der algorithm free condition}), $\mathfrak{rf}({}^-\Delta_{i_\ell,1})$ (considered for $\mathfrak{n}$) also satisfies the condition (\ref{eqn der algorithm free condition}) i.e. $[a+1,a_{i_{\ell-1},j_{\ell-1}}-1]_\rho \subset \mathfrak{rf}({}^-\Delta_{i_\ell,1})$. Now, if ${}^-\Delta_{i_{\ell},1}$ does not participate in the truncation process for  $\mathcal D^{\mathrm{Lang}}_{[a+1,b]_{\rho}}(\mathfrak n)$, then one can choose an index greater than $i_{\ell}$ for the last segment in the algorithm for $\mathcal D^{\mathrm{Lang}}_{[a+1,b]_{\rho}}(\mathfrak n)$. One sees that segment will then also participate in the truncation process for $\mathcal D^{\mathrm{Lang}}_{[a,b]_{\rho}}(\mathfrak m)$.
\end{enumerate}

 We divide it into two cases.
\begin{itemize}
    \item  Case 1: ${}^-\Delta_{i_{\ell},1}$ participates in the truncation process for $\mathcal D^{\mathrm{Lang}}_{[a+1,b]_{\rho}}(\mathfrak n)$. If we shorten the segment ${}^-\Delta_{i_\ell,1}$ by removing $\left[a+1, a_{i_{\ell-1},j_{\ell-1}}-1\right]_\rho$ from the left, the remaining part is $\left({^-}\Delta_{i_\ell,1}\right)^\mathrm{trc}=\Delta^\mathrm{trc}_{i_\ell,1}$. In this case, applying Algorithm \ref{alg:der:Lang} for $\mathcal{D}^\mathrm{Lang}_{[a+1,b]_\rho}(\mathfrak{n})$ we have: 
\begin{align*}
    \mathcal{D}^\mathrm{Lang}_{[a+1,b]_\rho} \circ \mathcal{D}_{[a]_\rho}^\mathrm{Lang}(\mathfrak{m}) &= \mathfrak{n}- \sum\limits_{t=1}^{\ell-1} \Delta_{i_t, j_t}-{^-}\Delta_{i_\ell,1} + \sum\limits_{t=1}^{\ell-1} \Delta_{i_t, j_t}^\mathrm{trc}+\left({^-}\Delta_{i_\ell,1}\right)^\mathrm{trc}\\
    &= \mathfrak{m}- \sum\limits_{t=1}^{\ell} \Delta_{i_t, j_t} + \sum\limits_{t=1}^{\ell} \Delta_{i_t, j_t}^\mathrm{trc} \quad (\text{by } \eqref{eq:n=der_a_m})\\ 
    &= \mathcal{D}_{[a,b]_\rho}^\mathrm{Lang}(\mathfrak{m}).
\end{align*}
\item Case 2: ${}^-\Delta_{i_{\ell,1}}$ does not participate in the truncation process for $\mathcal D^{\mathrm{Lang}}_{[a+1,b]_{\rho}}(\mathfrak n)$. Then $\Delta_{i_{\ell},1}^{\mathrm{trc}}={}^-\Delta_{i_{\ell},1}$. In this case, applying Algorithm \ref{alg:der:Lang} for $\mathcal{D}^\mathrm{Lang}_{[a+1,b]_\rho}(\mathfrak{n})$ we have:
\begin{align*}
    \mathcal{D}^\mathrm{Lang}_{[a+1,b]_\rho} \circ \mathcal{D}_{[a]_\rho}^\mathrm{Lang}(\mathfrak{m})
    &= \mathfrak{n}- \sum\limits_{t=1}^{\ell-1} \Delta_{i_t, j_t} + \sum\limits_{t=1}^{\ell-1} \Delta_{i_t, j_t}^\mathrm{trc} \\
    &= \mathfrak{m}- \sum\limits_{t=1}^{\ell} \Delta_{i_t, j_t}+{^-}\Delta_{i_\ell,1} + \sum\limits_{t=1}^{\ell-1} \Delta_{i_t, j_t}^\mathrm{trc}
    = \mathcal{D}_{[a,b]_\rho}^\mathrm{Lang}(\mathfrak{m}).
\end{align*}
\end{itemize}

(ii) The conditions indeed imply that  $\nu^a\rho$ must be in a removable free section of a segment in $\mathfrak m_1=\mathfrak m_{[a,b]}$. Otherwise, $\mathfrak m_{[a,b]}=\mathfrak n_{[a,b]}$ and so $\varepsilon^\mathrm{R}_{[a+1]_{\rho}}(\mathfrak m)=0$ implies that $\varepsilon^\mathrm{R}_{[a+1]_{\rho}}\left(\mathfrak n_{[a+1,b]}\right)=0$, where $\mathfrak n=\mathcal D_{[a]_{\rho}}^{\mathrm{Lang}}(\mathfrak m)$. This contradicts to $\mathcal D_{[a+1,b]_{\rho}}^{\mathrm{Lang}}(\mathfrak n)\neq \infty$. 

Now, one observes that the formula for the removal sequences for $\mathfrak n$ in Lemma \ref{upward seq a a+1,b} still applies. The remaining argument is the same as (i). 

%We denote $\mathfrak{n}=\mathcal{D}^\mathrm{Lang}_{[a]_\rho}(\mathfrak{m})=\mathfrak{m}-\Delta_* + {}^-\Delta_*$, for some $\Delta_* \in \mathfrak{m}[a]$. As $\mathcal{D}_{[a+1]_\rho}^\mathrm{Lang}(\mathfrak{m})= \infty$ and $\mathcal{D}_{[a+1,b]_\rho}^\mathrm{Lang}(\mathfrak{n})\neq  \infty$, we have $\Delta_* \in \mathfrak{m}_1,$ say $\Delta_*=\Delta_{i_*,1}$ for some $1 \leq i_*\leq k_0$ with $a_{i_*,2}\geq a+2$ and hence $\nu^a\rho \in \mathfrak{rf}(\Delta_*)$ (considered for $\mathfrak m$).

%As $\varepsilon^\mathrm{R}_{[a]_\rho}(L(\mathfrak{m}))=1$, by Remark \ref{rem:der_equality}, there exists only one $[a]_\rho$, which is a removable free point in $\mathfrak{m}_1$. Therefore, we have
%\begin{equation}
%    a_{i,2}=a+1 ~\text{ for } 1 \leq i \leq k_0 \text{ except for } i=i_*,
%\end{equation}
%and as $\mathcal{D}_{[a+1]_\rho}^\mathrm{Lang}(\mathfrak{m})= \infty$, we have
%\begin{equation}\label{eq:comp_E}
%    a_{i,3}=a+2 ~\text{ for all } i < i_*,
%\end{equation}
%and so $\nu^{a+1}\rho \notin \mathfrak{rf}(\Delta_{i,2})$ for all $i <i_*$. The rest of the arguments are similar to the proof of assertion (i) if we replace $i_\ell$ by $i_*$ and replace the fact $\mathcal{D}_{[a,b]_\rho}^\mathrm{Lang}(\mathfrak{m}) \neq \infty$ by $\mathcal{D}_{[a+1,b]_\rho}^\mathrm{Lang}(\mathfrak{n}) \neq \infty$ and use \eqref{eq:comp_E}. 
\qed

\subsection{Composition of $\mathrm{D}^{\mathrm{R}}_{[a]_{\rho}}$ and $\mathrm{D}^{\mathrm{R}}_{[a+1,b]_{\rho}}$}

\begin{lemma}\label{lem:D[ab]=D[a+1,b]_D[a]_pi}
Let $\pi\in \mathrm{Irr}_\rho$ and $[a,b]_\rho \in \mathrm{Seg}_\rho$ with $a <b$. Suppose $\varepsilon^\mathrm{R}_{[a]_\rho}(\pi) =1$. Then, 
\begin{enumerate}
    \item[(i)] If $\varepsilon^\mathrm{R}_{[a,b]_\rho}(\pi) \neq 0$, we have $\mathrm{D}^\mathrm{R}_{[a+1,b]_\rho} \circ \mathrm{D}^\mathrm{R}_{[a]_\rho}(\pi) \cong  \mathrm{D}^\mathrm{R}_{[a,b]_\rho}(\pi).$
    \item[(ii)] If $\varepsilon^\mathrm{R}_{[a+1]_\rho}(\pi) =0$ and  $\mathrm{D}^\mathrm{R}_{[a+1,b]_\rho} \circ \mathrm{D}^\mathrm{R}_{[a]_\rho}(\pi) \neq 0$, we have
    \[\mathrm{D}^\mathrm{R}_{[a,b]_\rho}(\pi)\cong \mathrm{D}^\mathrm{R}_{[a+1,b]_\rho} \circ \mathrm{D}^\mathrm{R}_{[a]_\rho}(\pi).\]
\end{enumerate}
\end{lemma}
\begin{proof}
(i) As $\varepsilon^\mathrm{R}_{[a]_\rho}(\pi)=1$, there is exactly one segment in $\mathfrak{hd}(\pi)[a]$. Moreover, by Lemma \ref{lem non-zero der hd}, there exists at least one segment $[a,c]_\rho \in \mathfrak{hd}(\pi)[a]$ such that $c \geq b$ because $\varepsilon^\mathrm{R}_{[a,b]_\rho}(\pi) \neq 0$. 
Therefore, $\mathfrak{hd}(\pi)[a]= \left\{[a,c]_\rho\right\}.$ Then, $[a, b]_\rho$ (as well as $[a]_\rho$) is a nonempty segment admissible to $\pi$, and the first segment (in this situation, the only segment) in the removal sequence is $\Upsilon\left([a, b]_\rho, \pi\right)=\Upsilon\left([a]_\rho, \pi\right)=[a, c]_\rho.$ By  \cite[Lemma 8.7]{Cha_csq}, 
\begin{align*}
 \mathfrak{r}([a]_\rho, \pi)&= \mathfrak{hd}(\pi) - [a, c]_\rho + [a+1, c]_\rho   \\
 &= \mathfrak{hd}(\pi) - \left\{\Upsilon\left([a, b]_\rho, \pi\right)\right\} + \left\{{^-}\Upsilon\left([a, b]_\rho, \pi\right)\right\}.
\end{align*}
Then, since $\varepsilon_{[a+1]_{\rho}}(\mathfrak m)=0$, $\mathfrak{r}([a]_{\rho}, \mathfrak m)[a+1]$ contains precisely one segment which is $[a+1,c]_{\rho}$. Now, by the definition of the removal process, we get
\[\mathfrak{r}([a, b]_\rho, \pi)= \mathfrak{r}\left([a+1, b]_\rho, \mathfrak{r}([a]_\rho, \pi)\right).\]
By \cite[Theorem 10.2]{Cha_csq}, we conclude that $\mathrm{D}^\mathrm{R}_{[a,b]_\rho}(\pi) \cong \mathrm{D}^\mathrm{R}_{[a+1,b]_\rho} \circ \mathrm{D}^\mathrm{R}_{[a]_\rho}(\pi).$

(ii) As $\varepsilon^\mathrm{R}_{[a+1]_{\rho}}(\pi)= 0$, Lemma \ref{lem non-zero der hd} implies that $\mathfrak{hd}(\pi)[a+1]=\emptyset$. On the other hand, $\mathrm{D}^\mathrm{R}_{[a+1,b]_{\rho}}\circ \mathrm{D}^\mathrm{R}_{[a]_{\rho}}(\pi)\neq 0$ and so Lemma \ref{lem non-zero der hd} implies that  $\mathfrak{hd}\left(\mathrm{D}^\mathrm{R}_{[a]_{\rho}}(\pi)\right)$ has a segment of the form $[a+1,d]_{\rho}$ for some $d \geq b$. Now, by Lemma \ref{lem derivative removal} and $\mathfrak{hd}(\pi)[a+1]=\emptyset$, we have $\mathfrak{r}([a]_{\rho}, \pi)[a+1]=\{[a+1,d]_{\rho}\}$. Then, since $\mathfrak{hd}(\pi)[a+1]=\emptyset$, $\mathfrak{hd}(\pi)$ must contain the segment $[a,d]_{\rho}$, in order to produce the segment $[a+1,d]_{\rho}$ in $\mathfrak r([a]_{\rho}, \pi)$.  Hence, $\mathrm{D}^\mathrm{R}_{[a,b]_{\rho}}(\pi)\neq 0$ by Lemma \ref{lem non-zero der hd}.
\end{proof}

\begin{remark}
Alternatively, one can observe that Lemma \ref{lem:D[ab]=D[a+1,b]_D[a]_pi}(ii) follows from the fact that $\mathrm{St}([a,b]_\rho)$ is the unique subquotient of $\mathrm{St}([a,a]_\rho) \times \mathrm{St}([a+1,b]_\rho)$ not admitting a right $\rho \nu^{a+1}$-derivative.
\end{remark}

\subsection{$\mathfrak{tds}$ after $\rho$-derivatives}

\begin{lemma} \label{lem tds on a ab commute}
Let $\mathfrak m \in \mathrm{Mult}_{\rho}$  and $[a,b]_\rho \in \mathrm{Seg}_\rho$. Suppose $\mathcal D_{[a,b]_{\rho}}^{\mathrm{Lang}}(\mathfrak m) \neq \infty$. Let 
\[ \mathfrak{n}:=\mathcal D^{\mathrm{Lang}}_{[a]_{\rho}}(\mathfrak m)=\mathfrak m-\Delta_a+{}^-\Delta_a
\]
for some $\Delta_a \in \mathfrak m[a]$. Suppose $\Delta_a \notin \mathfrak m_{[a,b]}$ and so $\mathfrak m_{[a,b]}=\mathfrak n_{[a,b]}$. Then after the removal steps for the $\mathfrak{tds}(-,a)$ process for both $\mathfrak m$ and $\mathcal D_{[a,b]_{\rho}}^{\mathrm{Lang}}(\mathfrak m)$, one obtains the same multiset of segments starting with $\nu^a\rho$ and not lying in $\mathfrak m_{[a,b]}$.
\end{lemma}

\begin{example}
     Consider $\mathfrak m=\left\{ [1]_{\rho}, [1,2]_{\rho}, [1,5]_{\rho}, [2,4]_{\rho}\right\}$ and $a=1$ and $b=3$. In such example, $\mathfrak m_{[1,3]}=\mathfrak n_{[1,3]}=\left\{ [1,5]_{\rho}, [2,4]_{\rho} \right\}$. Now, $\mathcal D^{\mathrm{Lang}}_{[1,3]_{\rho}}(\mathfrak m)=\left\{[1]_\rho, [1,2]_\rho, [2,5]_{\rho}, [4]_{\rho}\right\}$. Note that the segments participating in the removal step for the $\mathfrak{tds}(\mathfrak m, a)$-process are $[1,2]_{\rho}, [2,4]_{\rho}$, while the segments participating in the removal step for the $\mathfrak{tds}\left(\mathcal D_{[1,3]_{\rho}}^{\mathrm{Lang}}(\mathfrak{m}), a\right)$-process are $[1,2]_{\rho}$ and $[2,5]_{\rho}$. One may think that the segment $[2,4]_{\rho}$ in $\mathfrak{tds}(\mathfrak m,a)$ is truncated in the algorithm for $\mathcal D^{\mathrm{Lang}}_{[1,3]_{\rho}}(\mathfrak m)$, and so one has to replace with the $[2,5]_{\rho}$ in $\mathcal D_{[1,3]_{\rho}}^{\mathrm{Lang}}(\mathfrak m)$ (which is truncated from $[1,5]_{\rho}$ in $\mathfrak m$). The remaining segment starting with $\nu\rho$ after $\mathfrak{tds}(-,1)$-process on both $\mathfrak m$ and $\mathcal D^{\mathrm{Lang}}_{[a]_{\rho}}(\mathfrak m)$ is the segment $[1]_{\rho}$.
     
     %Now one can see to delete $[1]_{\rho}$ in $\mathcal D_{[1,3]_{\rho}}^{\mathrm{Lang}}(\mathfrak m)$ to obtain $\mathcal D_{[1]_{\rho}}^{\mathrm{Lang}}\circ \mathcal D_{[1,3]_{\rho}}^{\mathrm{Lang}}(\mathfrak m)$, and on the other hand, $[1]_{\rho}$ is also deleted in $\mathfrak m$ to obtain $\mathcal D_{[1]_{\rho}}^{\mathrm{Lang}}(\mathfrak m)$. Since $\mathfrak m_1=\mathfrak n_1$, we then also have $\mathcal D_{[1,3]_{\rho}}^{\mathrm{Lang}}\circ \mathcal D_{[1]_{\rho}}^{\mathrm{Lang}}(\mathfrak m)=\mathcal D_{[1]_{\rho}}^{\mathrm{Lang}}\circ \mathcal D_{[1,3]_{\rho}}^{\mathrm{Lang}}(\mathfrak m)$. The general argument follows this pattern, but it is not so educational to write down all details and so we omit.
\end{example}

\noindent
{\it Proof of Lemma \ref{lem tds on a ab commute}.}  Let's assume all the notations as mentioned in Algorithm \ref{alg:der:Lang} applied for $\mathcal D^{\mathrm{Lang}}_{[a,b]_{\rho}}(\mathfrak m)$.
\begin{enumerate}
\item When $a_{i_{\ell-1},j_{\ell-1}}=a+1$,  the segment ${^-}\Delta_{i_\ell,1}$ replaces $\Delta_{i_{\ell-1},j_{\ell-1}}$ to  participate in the removal steps of the $\mathfrak{tds}(-,a)$ process on $\mathcal{D}_{[a,b]_\rho}^\mathrm{Lang}(\mathfrak{m})$, whereas 
$\Delta_{i_{\ell-1},j_{\ell-1}}$ participates in the removal steps of the $\mathfrak{tds}(-,a)$ process on $\mathfrak{m}$. Now the remaining segments starting with $\nu^a\rho$ are the same. In particular, we have the lemma.
\item When $a_{i_{\ell-1},j_{\ell-1}}\neq a+1$, the segment $\Delta_{i_\ell, 1}$ is not in the removal step of the $\mathfrak{tds}(-,a)$-process for $\mathfrak m$ since $[a]_\rho$ is a removable free point. Then one has the same segments for the removal steps of the $\mathfrak{tds}(-, a)$-process for both $\mathfrak m$ and $\mathcal D^{\mathrm{Lang}}_{[a,b]_{\rho}}(\mathfrak m)$. Now the lemma follows.
\end{enumerate}
\qed

In the proofs of Lemmas \ref{lem:comm_D[ab]_D[a]} and \ref{lem:comm_D_ab:D_a+1} below, we shall need some variations of the above lemma. The details are similar, and we shall not provide full arguments each time.

\subsection{Commutativity of $\mathcal D_{[a]_{\rho}}^{\mathrm{Lang}}$ and $\mathcal D_{[a,b]_{\rho}}^{\mathrm{Lang}}$}
%The following Lemmas play an integral role in the inductive reasoning used to prove some of the key conclusions of this paper, such as in Theorem \ref{thm:zero_der_Lang} and \ref{thm:der:Lang}. 

We record the following observation, whose proof is straightforward:

\begin{lemma} \label{lem change upward seq  a ab}
Let $\mathfrak m \in \mathrm{Mult}_{\rho}$ and let $a \in \mathbb Z$. Suppose $\mathcal D^{\mathrm{Lang}}_{[a]_{\rho}}(\mathfrak m)\neq \infty$. Then
\[  \mathfrak n:= \mathcal D^{\mathrm{Lang}}_{[a]_{\rho}}(\mathfrak m)=\mathfrak m-\Delta_*+{}^-\Delta_*
\]
for some $\Delta_* \in \mathfrak m[a]$. We use the notations of Algorithm \ref{alg:der:Lang} for both $\mathcal D^{\mathrm{Lang}}_{[a,b]_{\rho}}(\mathfrak m)$ and $\mathcal D^{\mathrm{Lang}}_{[a,b]_{\rho}}(\mathfrak n)$. Analogously, $\mathfrak n_1=\mathfrak n_{[a,b]}$ and $\mathfrak n_{i+1}=\mathfrak n_i-\underline{\mathfrak{Us}}(\mathfrak n_i)$. Let $k_0$ be the largest integer such that $\underline{\mathfrak{Us}}(\mathfrak m_{k_0})[a]\neq \emptyset$. Suppose $\Delta_* \in \mathfrak m_{[a,b]}$ i.e. $\Delta_*=\Delta_{i_*,1} \in \underline{\mathfrak{Us}}(\mathfrak m_{i_*})$ for some $1 \leq i_* \leq k_0$. Then, $i_*$ is the largest integer $i \leq k_0$ such that $\Delta_{i,1}$ contains $\nu^a\rho$ in its removable free section. Furthermore, we have

\begin{align}\label{eq:us_on_rho_der}
 \underline{\mathfrak{Us}}(\mathfrak{n}_i)=\begin{cases}
     \underline{\mathfrak{Us}}(\mathfrak{m}_i) &\mbox{ if } 1 \leq i < i_* \text{ or } k_0+1 \leq i \leq k\\
     \underline{\mathfrak{Us}}(\mathfrak{m}_i) - \Delta_{i,1} + {^-}\Delta_{i,1} + \Delta_{i+1,1} &\mbox{ if } i = i_* \text{ and } i+1 \leq k_0\\
     \underline{\mathfrak{Us}}(\mathfrak{m}_i) - \Delta_{i,1} + {^-}\Delta_{i,1} &\mbox{ if } i = i_* = k_0\\
     \underline{\mathfrak{Us}}(\mathfrak{m}_i) - \Delta_{i,1} + \Delta_{i+1,1} &\mbox{ if }  i_* <i <k_0\\
     \underline{\mathfrak{Us}}(\mathfrak{m}_i) - \Delta_{i,1} &\mbox{ if }  i_*< i = k_0,
 \end{cases}   
\end{align}
\end{lemma}

\begin{proof}
By Algorithm \ref{alg:rho_der:Lang}, $i^*$ is the largest integer $\leq k_0$ such that the unique segment in $\underline{\mathfrak{Us}}(\mathfrak m_{i^*})[a]$ has $\nu^a\rho$ in its removable free section. With this choice of $i^*$, for $i_*<i \leq k_0$, we have $\underline{\mathfrak{Us}}(\mathfrak m_i)[a+1]\neq \emptyset$ (i.e. $a_{\ell-1, 2}=a+1$). Now one keeps track of the indices to obtain (\ref{eq:us_on_rho_der}).
\end{proof}

\begin{example}
Let $\mathfrak m=\left\{ [1,2]_{\rho}, [1,3]_{\rho}, [2,4]_{\rho}, [2,3]_{\rho}, [1,4]_{\rho}, [1,5]_{\rho}, [3,5]_{\rho} \right\}$. Suppose we consider $a=1$ and $b=2$ for the upward sequences.
\begin{enumerate}
\item Note that $\underline{\mathfrak{Us}}(\mathfrak m_1)=\left\{[1,5]_{\rho}\right\}$, $\underline{\mathfrak{Us}}(\mathfrak m_2)=\left\{ [1,4]_{\rho}, [3,5]_{\rho} \right\}$, $\underline{\mathfrak{Us}}(\mathfrak m_3)=\left\{ [1,3]_{\rho},[2,4]_{\rho}\right\}$ and $\underline{\mathfrak{Us}}(\mathfrak m_4)=\left\{ [1,2]_{\rho},[2,3]_{\rho}\right\}$.
\item Let $\mathfrak n=\mathcal D_{[1]_{\rho}}^{\mathrm{Lang}}(\mathfrak m)=\left\{ [1,2]_{\rho},[1,3]_{\rho}, [2,4]_{\rho}, [2,3]_{\rho},  [2,4]_{\rho}, [1,5]_{\rho}, [3,5]_{\rho}\right\}$ and $\mathfrak{n}_1=\mathfrak{n}_{[1,2]}$. Then, $\underline{\mathfrak{Us}}(\mathfrak n_1)=\left\{ [1,5]_{\rho}\right\}$, $\underline{\mathfrak{Us}}(\mathfrak n_2)=\left\{ [1,3]_{\rho}, [2,4]_{\rho}, [3,5]_{\rho} \right\}$, $\underline{\mathfrak{Us}}(\mathfrak n_3)=\left\{[1,2]_{\rho}, [2,4]_{\rho}\right\}$ and $\underline{\mathfrak{Us}}(\mathfrak n_4)=\left\{[2,3]_{\rho}\right\}$ (Here $\mathfrak{n}_i=\mathfrak{n}_{i-1}-\underline{\mathfrak{Us}}(\mathfrak n_{i-1})$ for $i >1$).
\end{enumerate}
This example illustrates changes in the upward sequences after taking $\mathcal D^{\mathrm{Lang}}_{[1]_{\rho}}$. For example, the segments in $\underline{\mathfrak{Us}}(\mathfrak n_2)$ is obtained from $\underline{\mathfrak{Us}}(\mathfrak m_2)$ by truncating $[1,4]_{\rho}$, and adding$[1,3]_{\rho}$ (the segment starting $[1]_{\rho}$ in $\underline{\mathfrak{Us}}(\mathfrak m_3)$). Further, the segments in $\underline{\mathfrak{Us}}(\mathfrak n_3)$ is obtained from $\underline{\mathfrak{Us}}(\mathfrak m_3)$ by replacing the segment $[1,3]_{\rho}$ in $\underline{\mathfrak{Us}}(\mathfrak m_3)$ with the segment $[1,2]_{\rho}$ in $\underline{\mathfrak{Us}}(\mathfrak m_4)$. 
\end{example}

The upshot of Lemma \ref{lem change upward seq  a ab} is the following:

\begin{corollary}
We use the notations in Lemma \ref{lem change upward seq  a ab}. Let $[a_{i_1, j_1},b_{i_1, j_1}]_{\rho}, \ldots, [a_{i_\ell, j_{\ell}}, b_{i_{\ell}, j_{\ell}}]_{\rho}$ be the segments participating in the truncation process for $\mathcal D^{\mathrm{Lang}}_{[a,b]_{\rho}}(\mathfrak m)$ as in  Algorithm \ref{alg:der:Lang} . If such $a_{i_p, j_p}\neq a, a+1$, then the segment $[a_{i_p,j_p}, b_{i_p,j_p}]$ also participates in the truncation process for  $\mathcal D^{\mathrm{Lang}}_{[a,b]_{\rho}}(\mathfrak n)$.
\end{corollary}

With the above corollary, one has to investigate how to pick the last one or two segments participating in the truncation process for $\mathcal D^{\mathrm{Lang}}_{[a,b]_{\rho}}(\mathfrak n)$ and the key is the following lemma:

\begin{lemma} \label{lem free point after Da}
We use the notations in Lemma \ref{lem change upward seq  a ab}. Suppose further that $\varepsilon^{\mathrm{R}}_{[a]_{\rho}}(\mathfrak m)\geq 2$. Then
\begin{enumerate}
\item There exists $i_{**}<i_*$ such that $\underline{\mathfrak{Us}}(\mathfrak m_{i_{**}})[a]\neq \emptyset$ and $\nu^a\rho$ is in the removable free section of the unique segment $\Delta_{i_{**},1}$ in $\underline{\mathfrak{Us}}(\mathfrak m_{i_{**}})[a]$.
\item Suppose $\mathcal D^{\mathrm{Lang}}_{[a,b]_{\rho}}(\mathfrak m)\neq \infty$. If $\Delta_*$ does not participate in the truncation process for $\mathcal D^{\mathrm{Lang}}_{[a,b]_{\rho}}(\mathfrak m)$, then $i_{\ell}<i_*$. Otherwise, $i_{\ell}=i_*$.
\end{enumerate}
\end{lemma}

\begin{proof}
(1) follows from the conditions $\varepsilon^{\mathrm{R}}_{[a]_{\rho}}(\mathfrak m)\geq 2$ and $\Delta_* \in \mathfrak m_{[a,b]}$, and Lemma \ref{lem combinatorial varepsilon}. For (2), it follows from Lemma \ref{lem change upward seq  a ab} that $i_{\ell} \leq i_*$, since  $[a_{i_\ell,1}, b_{i_\ell,1}]$ is a segment in $\mathfrak{m}_{[a,b]}$ with $\nu^a\rho$ in its removable free section. Then one has the two situations according to the given condition.
\end{proof}

We can now state and prove our main lemma in this subsection:

\begin{lemma}\label{lem:comm_D[ab]_D[a]}
Let $\mathfrak{m}\in \mathrm{Mult}_\rho$ and $[a,b]_\rho \in \mathrm{Seg}_\rho$. Suppose $\varepsilon^\mathrm{R}_{[a]_\rho}(\mathfrak{m}) \geq 2$.
\begin{itemize}
\item[(i)] If $\mathcal{D}^\mathrm{Lang}_{[a,b]_\rho}(\mathfrak{m}) \neq \infty$, we have $\mathcal{D}^\mathrm{Lang}_{[a]_\rho} \circ \mathcal{D}_{[a,b]_\rho}^\mathrm{Lang}(\mathfrak{m}) =  \mathcal{D}_{[a,b]_\rho}^\mathrm{Lang} \circ \mathcal{D}^\mathrm{Lang}_{[a]_\rho}(\mathfrak{m})\neq \infty.$ 
\item[(ii)] If $\mathcal D^{\mathrm{Lang}}_{[a,b]_{\rho}}\circ \mathcal D^{\mathrm{Lang}}_{[a]_{\rho}}(\mathfrak m)\neq \infty$, we have $\mathcal{D}^\mathrm{Lang}_{[a,b]_\rho}(\mathfrak{m}) \neq \infty$.
\end{itemize}
\end{lemma}

Before giving a proof of Lemma \ref{lem:comm_D[ab]_D[a]}, we give an example:

\begin{example}
Let $\mathfrak m =\left\{ [2]_{\rho}, [2,4]_{\rho}, [2,5]_{\rho}, [3]_{\rho}, [4,5]_{\rho} \right\}$. Note that 
\begin{enumerate}
\item $\mathcal D_{[2]_{\rho}}^{\mathrm{Lang}}(\mathfrak m)=\mathfrak m-[2,4]_{\rho}+[3,4]_{\rho}=\left\{ [2]_{\rho}, [3,4]_{\rho}, [2,5]_{\rho}, [3]_{\rho}, [4,5]_{\rho} \right\}$ and $\mathcal D_{[2,3]_{\rho}}^{\mathrm{Lang}}\circ \mathcal D_{[2]_{\rho}}^{\mathrm{Lang}}(\mathfrak m)=\mathcal D_{[2]_{\rho}}^{\mathrm{Lang}}(\mathfrak m)-[2,5]_{\rho}+[3,5]_{\rho}-[3]_{\rho}=\left\{ [2]_{\rho}, [3,4]_{\rho}, [3,5]_{\rho}, [4,5]_{\rho} \right\}$.
\item $\mathcal D^{\mathrm{Lang}}_{[2,3]_{\rho}}(\mathfrak m)=\mathfrak m-[2,4]_{\rho}+[3,4]_{\rho}-[3]_{\rho}=\left\{ [2],[3,4]_{\rho},[2,5]_{\rho},[4,5]_{\rho} \right\}$ and $\mathcal D^{\mathrm{Lang}}_{[2]_{\rho}}\circ \mathcal D^{\mathrm{Lang}}_{[2,3]_{\rho}}(\mathfrak m)=\mathcal D^{\mathrm{Lang}}_{[2,3]_{\rho}}(\mathfrak m)-[2,5]_{\rho}+[3,5]_{\rho}=\left\{ [2]_{\rho},[3,4]_{\rho},[3,5]_{\rho},[4,5]_{\rho}\right\}$.
\end{enumerate}
\end{example}

\noindent
{\it Proof of Lemma \ref{lem:comm_D[ab]_D[a]}.} We assume all the notations as mentioned in Algorithm \ref{alg:der:Lang}. Let $k_0$ be the largest integer such that $\underline{\mathfrak{Us}}(\mathfrak m_{k_0})[a]\neq \emptyset$. As $\varepsilon^\mathrm{R}_{[a]_\rho}(\mathfrak{m}) \neq 0$, by Algorithm \ref{alg:rho_der:Lang}, there exists a non-empty segment $\Delta_a \in \mathfrak{m}[a]$ such that   
$\mathfrak{n}:=\mathcal{D}^\mathrm{Lang}_{[a]_\rho}  (\mathfrak{m})=\mathfrak{m}-\Delta_a + {^-}\Delta_a.$ Consider the multisegment $\mathfrak{n}_1=\mathfrak{n}_{[a,b]}$ and recursively for $i > 1$, we set $\mathfrak{n}_{i}=\mathfrak{n}_{i-1}-\underline{\mathfrak{Us}}(\mathfrak{n}_{i-1})$. 

We first assume that $\Delta_a \notin \mathfrak{m}_1=\mathfrak{m}_{[a,b]}$, that is $\mathfrak{n}_1=\mathfrak{m}_1$. Then, the assertion (i) follows from Lemma \ref{lem tds on a ab commute}, and the assertion (ii) follows immediately as  $\mathcal D^{\mathrm{Lang}}_{[a,b]_{\rho}}(\mathfrak{n}) \neq \infty$ implies $\mathcal D^{\mathrm{Lang}}_{[a,b]_{\rho}}(\mathfrak{m}_1)=\mathcal D^{\mathrm{Lang}}_{[a,b]_{\rho}}(\mathfrak{n}_1) \neq \infty$.

%If $a_{i_{\ell-1},j_{\ell-1}} \neq a+1$,  the segments participating in the removal steps of the $\mathfrak{tds}(-,a)$ process on $\mathcal{D}_{[a,b]_\rho}^\mathrm{Lang}(\mathfrak{m})$ are exactly same as in the removal steps of the $\mathfrak{tds}(-,a)$ process on $\mathfrak{m}$, since $[a]_\rho \subset \mathfrak{rf}(\Delta_{i_\ell,1})$ and $\Delta_{i_\ell,1}$ does not participate in the later one. If $a_{i_{\ell-1},j_{\ell-1}}=a+1$,  the segment ${^-}\Delta_{i_\ell,1}$ replaces $\Delta_{i_{\ell-1},j_{\ell-1}}$ to  participate in the removal steps of the $\mathfrak{tds}(-,a)$ process on $\mathcal{D}_{[a,b]_\rho}^\mathrm{Lang}(\mathfrak{m})$, whereas 
%$\Delta_{i_{\ell-1},j_{\ell-1}}$ participates in the removal steps of the $\mathfrak{tds}(-,a)$ process on $\mathfrak{m}$ and all other removal steps remain unchanged in the $\mathfrak{tds}(-,a)$ process on $\mathcal{D}_{[a,b]_\rho}^\mathrm{Lang}(\mathfrak{m})$ compared to the process on $\mathfrak{m}$. Therefore, if $\Delta_a \notin \mathfrak{m}_1$, both the assertions (i) and (ii) hold. 

For the remainder of the proof, we assume that $\Delta_a \in \mathfrak{m}_1=\mathfrak{m}_{[a,b]}$. Then, $\Delta_a=\Delta_{i_*,1}$ for some $1 \leq i_* \leq k_0$.

(i) Suppose $\mathcal{D}_{[a,b]_\rho}^\mathrm{Lang}(\mathfrak{m}) = \mathfrak{m}- \sum\limits_{t=1}^{\ell} \Delta_{i_t, j_t} + \sum\limits_{t=1}^{\ell} \Delta_{i_t, j_t}^\mathrm{trc} \neq \infty$ as in Algorithm \ref{alg:der:Lang}. By Lemma \ref{lem free point after Da}, we have $i_* \geq i_\ell$. Further, as $\varepsilon^\mathrm{R}_{[a]_{\rho}}(\mathfrak{m})\geq 2$, there exists largest positive integer $i_{**} < i_*$ such that 
\[\mathcal{D}^\mathrm{Lang}_{[a]_\rho} \circ \mathcal{D}^\mathrm{Lang}_{[a]_\rho}(\mathfrak{m})=\mathcal{D}^\mathrm{Lang}_{[a]_\rho}(\mathfrak{m})-\Delta_{i_{**},1}+ {^-}\Delta_{i_{**},1}.\]
%Therefore, if $i_*=i_\ell$, we have $\nu^a \rho \in \mathfrak{rf}(\Delta_{i_{**},1})$ where $\mathfrak{rf}$ is considered both in $\mathfrak{m}$ and $\mathfrak{n}$.

If $i_*=i_\ell$, the segment ${^-}\Delta_{i_\ell,1} \in \mathfrak{n}_{[a,b]}$ and it participates in the truncation process for $\mathcal{D}_{[a,b]_\rho}^\mathrm{Lang}(\mathfrak{n})$ when $a_{i_{\ell-1},j_{\ell-1}} \neq a+1$ in $\mathfrak{m}_{[a,b]}$. In that case, we shorten the segment ${^-}\Delta_{i_\ell,1}$ by removing $\left[a+1, a_{i_{\ell-1},j_{\ell-1}}-1\right]_\rho$ from left and the remaining part is denoted by $\left({^-}\Delta_{i_\ell,1}\right)^\mathrm{trc}=\Delta^\mathrm{trc}_{i_\ell,1}$. Therefore, using \eqref{alg_der_lang},
Lemmas \ref{lem change upward seq  a ab} and \ref{lem free point after Da}, we have $\mathcal{D}^\mathrm{Lang}_{[a,b]_\rho} \circ \mathcal{D}^\mathrm{Lang}_{[a]_\rho}(\mathfrak{m})$
\begin{align*}
    &=\mathcal{D}^\mathrm{Lang}_{[a]_\rho}(\mathfrak{m})-\sum\limits_{t=1}^{\ell-1} \Delta_{i_t, j_t} + \sum\limits_{t=1}^{\ell-1} \Delta_{i_t, j_t}^\mathrm{trc}+\begin{cases}
     -{^-}\Delta_{i_\ell,1} +  \left({^-}\Delta_{i_\ell,1}\right)^\mathrm{trc} - \Delta_{i_{**},1} + {^-}\Delta_{i_{**},1} &\mbox{ if } i_*=i_\ell,\\
     -\Delta_{i_\ell,1} +  \Delta^\mathrm{trc}_{i_\ell,1}  &\mbox{otherwise}
    \end{cases}\\
     &=\mathfrak{m}-\sum\limits_{t=1}^{\ell} \Delta_{i_t, j_t} + \sum\limits_{t=1}^{\ell} \Delta_{i_t, j_t}^\mathrm{trc}+\begin{cases}
      - \Delta_{i_{**},1} + {^-}\Delta_{i_{**},1} &\mbox{ if }  i_*=i_\ell,\\
      - \Delta_{i_{*},1} + {^-}\Delta_{i_{*},1} &\mbox{otherwise}
    \end{cases}\\
    &=\mathcal{D}^\mathrm{Lang}_{[a,b]_\rho}(\mathfrak{m})+\begin{cases}
      - \Delta_{i_{**},1} + {^-}\Delta_{i_{**},1} &\mbox{ if } i_*=i_\ell, \\
     - \Delta_{i_{*},1} + {^-}\Delta_{i_{*},1}  &\mbox{ otherwise}.
    \end{cases}
\end{align*}

We now turn to compute $\mathcal{D}^\mathrm{Lang}_{[a]_\rho} \circ \mathcal{D}^\mathrm{Lang}_{[a,b]_\rho}(\mathfrak{m})$. It is similar to Lemma \ref{lem tds on a ab commute} (also see Example \ref{ex tds a commute aa+1 Delta in m1} in Appendix \ref{s example tds process}). Indeed, if $i_* > i_\ell$ and $a_{i_{\ell-1},j_{\ell-1}}=a+1$, we have $i_* > i_{\ell-1}$ and the segment ${^-}\Delta_{i_\ell,1}$ replaces $\Delta_{i_{\ell-1},j_{\ell-1}}$ to participate in the removal step of the $\mathfrak{tds}(-,a)$ process on $\mathcal{D}_{[a,b]_\rho}^\mathrm{Lang}(\mathfrak{m})$, whereas 
$\Delta_{i_{\ell-1},j_{\ell-1}}$ participates in the removal step of the $\mathfrak{tds}(-,a)$ process on $\mathfrak{m}$. Otherwise the $\mathfrak{tds}(-,a)$ process on both $\mathcal{D}_{[a,b]_\rho}^\mathrm{Lang}(\mathfrak{m})$ and $\mathfrak{m}$ removes  same set of segments starting with $\nu^a\rho$ and $ \nu^{a+1}\rho$. Hence, using 
\eqref{eq:us_on_rho_der},
\begin{align*}
\mathcal{D}^\mathrm{Lang}_{[a]_\rho} \circ \mathcal{D}^\mathrm{Lang}_{[a,b]_\rho}(\mathfrak{m})=\mathcal{D}^\mathrm{Lang}_{[a,b]_\rho}(\mathfrak{m})+\begin{cases}
      - \Delta_{i_{**},1} + {^-}\Delta_{i_{**},1} &\mbox{ if } i_*=i_\ell,\\
     - \Delta_{i_{*},1} + {^-}\Delta_{i_{*},1}  &\mbox{ if } i_*>i_\ell.
    \end{cases}    
\end{align*}
Combining the above two expressions, we have the lemma.

(ii) Suppose $\mathcal D^{\mathrm{Lang}}_{[a,b]_{\rho}}\circ D^{\mathrm{Lang}}_{[a]_{\rho}}(\mathfrak m)\neq \infty$. Then, by Lemmas \ref{lem change upward seq  a ab} and \ref{lem free point after Da}, we can trace the algorithm to see that $\mathcal D^{\mathrm{Lang}}_{[a,b]_{\rho}}(\mathfrak m)\neq \infty$.

\subsection{Commutativity of $\mathrm{D}^{\mathrm R}_{[a,b]_{\rho}}$ and $\mathrm{D}^{\mathrm R}_{[a]_{\rho}}$}

\begin{lemma}\label{lem:comm_D[ab]_D[a]_pi}
Let $\pi \in \mathrm{Irr}_\rho$ and $[a,b]_\rho \in \mathrm{Seg}_\rho$. Suppose $\varepsilon^\mathrm{R}_{[a]_\rho}(\pi) \geq 2$. Then, 
\begin{enumerate}
    \item[(i)] If $\varepsilon^\mathrm{R}_{[a,b]_\rho}(\pi) \neq 0$, we have  $\mathrm{D}^\mathrm{R}_{[a]_\rho} \circ \mathrm{D}^\mathrm{R}_{[a,b]_\rho}(\pi) \cong \mathrm{D}^\mathrm{R}_{[a,b]_\rho} \circ \mathrm{D}^\mathrm{R}_{[a]_\rho}  (\pi)\neq 0.$
    \item[(ii)] If $\mathrm{D}^\mathrm{R}_{[a,b]_\rho} \circ \mathrm{D}^\mathrm{R}_{[a]_\rho}  (\pi)\neq 0$, we have $\varepsilon^\mathrm{R}_{[a,b]_\rho}(\pi) \neq 0$.
\end{enumerate}
\end{lemma}
\begin{proof}
(i) The commutativity part follows from \cite[Lemma 4.4]{Cha_csq}, and the non-zero part follows from \cite[Proposition 5.5]{Cha_csq}.

(ii) As $\mathrm{D}^\mathrm{R}_{[a,b]_{\rho}}\circ \mathrm{D}^\mathrm{R}_{[a,a]_{\rho}}(\pi)\neq 0$, it follows from Lemma \ref{lem non-zero der hd} that $\mathfrak{hd}(\mathrm{D}^\mathrm{R}_{[a,a]_{\rho}}(\pi))$ contains a segment of the form $[a,c]_{\rho}$ for some $c \geq b$. Now Lemmas \ref{lem removal process singleton} and \ref{lem derivative removal} imply $\mathfrak{hd}(\pi)$ has the segment $[a,c]_{\rho}$, and so $\mathrm{D}^\mathrm{R}_{[a,b]_{\rho}}(\pi)\neq 0$ by Lemma \ref{lem non-zero der hd}.
\end{proof}

\subsection{Commutativity of $\mathcal D_{[a,b]_{\rho}}^{\mathrm{Lang}}$ and $\mathcal D_{[a+1]_{\rho}}^{\mathrm{Lang}}$}

We now compare the upward sequences for computing $\mathcal D^{\mathrm{Lang}}_{[a,b]_{\rho}}(\mathfrak m)$ and $\mathcal D^{\mathrm{Lang}}_{[a,b]_{\rho}}(\mathfrak n)$, where $\mathfrak n=\mathcal D^{\mathrm{Lang}}_{[a+1]_{\rho}}(\mathfrak m)$.

\begin{lemma} \label{lem upward seq expression a+1}
Let $\mathfrak m \in \mathrm{Mult}_{\rho}$. Suppose $\mathcal D^{\mathrm{Lang}}_{[a+1]_{\rho}}(\mathfrak m)\neq \infty$. Then one has
\[ \mathfrak n :=  \mathcal D^{\mathrm{Lang}}_{[a+1]_{\rho}}(\mathfrak m)=\mathfrak m-\Delta^*+{}^-\Delta^*
\]
for some $\Delta^* \in \mathfrak m[a+1]$. We use the notations in Algorithm (\ref{alg:der:Lang}), and in particular, let $\mathfrak m_1=\mathfrak m_{[a,b]}$ and $\mathfrak n_1=\mathfrak n_{[a,b]}$ , and $\mathfrak n_{i+1}=\mathfrak n_i-\underline{\mathfrak{Us}}(\mathfrak n_i)$. Suppose $\Delta^*\in\mathfrak m_{[a,b]}$. Let $i^*$ be the largest integer such that $\Delta^* \in \underline{\mathfrak{Us}}(\mathfrak m_{i^*})$. Then $\nu^{a+1}\rho$ is in the removable free section of $\Delta^*$ (considered as a segment in $\underline{\mathfrak{Us}}(\mathfrak m_{i^*})$) and furthermore,
\begin{enumerate}
\item for $i<i^*$, $\underline{\mathfrak{Us}}(\mathfrak n_i)=\underline{\mathfrak{Us}}(\mathfrak m_i)$
\item for $i=i^*$, 
\begin{itemize}
    \item Suppose there exists a segment $\Delta'$ in $\underline{\mathfrak{Us}}(\mathfrak m_{i^*+1})[a+1]$. Suppose furthermore that either $\underline{\mathfrak{Us}}(\mathfrak m_{i^*})[a]=\emptyset$ or $\Delta'$ is linked to the unique segment $\Delta_{i^*,1}$ in $\underline{\mathfrak{Us}}(\mathfrak m_{i^*})[a]$ ,
\[  \underline{\mathfrak{Us}}(\mathfrak n_{i^*})=\underline{\mathfrak{Us}}(\mathfrak m_{i^*})-\Delta^*+{}^-\Delta^*+\Delta'
\] 
   \item otherwise, 
   \[  \underline{\mathfrak{Us}}(\mathfrak n_{i^*})=\underline{\mathfrak{Us}}(\mathfrak m_{i^*})-\Delta^*+{}^-\Delta^* \]
\end{itemize}
\item Suppose we are in the first bullet of (2). For $i = i^*+1$, we have $\underline{\mathfrak{Us}}(\mathfrak n_i)$ as follows: let $\Delta_i'$ be the unique segment in $\underline{\mathfrak{Us}}(\mathfrak m_i)[a+1]$
\begin{itemize}
 \item Suppose there exists a segment $\Delta_{i+1}'$ in $\underline{\mathfrak{Us}}(\mathfrak m_{i+1})[a+1]$. Suppose, furthermore that either $\underline{\mathfrak{Us}}(\mathfrak m_i)[a]=\emptyset$ or $\Delta_{i+1}'$ is linked to the unique segment in $\underline{\mathfrak{Us}}(\mathfrak m_i)[a]$, then one has
 \[   \underline{\mathfrak{Us}}(\mathfrak n_i)=\underline{\mathfrak{Us}}(\mathfrak m_i)-\Delta'_i+\Delta'_{i+1}
 \]
 \item If any condition in the first bullet fails, one has
  \[   \underline{\mathfrak{Us}}(\mathfrak n_i)=\underline{\mathfrak{Us}}(\mathfrak m_i)-\Delta'_i
 \]
 \end{itemize}
 \item One recursively has the above description of $\underline{\mathfrak{Us}}(\mathfrak n_i)$ until it reaches the second bullet case, say at the index $i'$. If we are in the second bullet of (2), set $i'=i^*$.
 \item For $i>i'$, one has $\underline{\mathfrak{Us}}(\mathfrak n_i)=\underline{\mathfrak{Us}}(\mathfrak m_i)$.
\end{enumerate}

%Each upward sequence $\underline{\mathfrak{Us}}(\mathfrak n_i)$ is either equal to
%\begin{enumerate}
%\item $\underline{\mathfrak{Us}}(\mathfrak m_{j_i})$ (for some $j$); or
%\item $\underline{\mathfrak{Us}}(\mathfrak m_{j_i})-\Delta'+\Delta''$, where $\Delta'$ (resp. $\Delta''$) is the unique segment in $\underline{\mathfrak{Us}}(\mathfrak m_j)$ (resp. $\underline{\mathfrak{Us}}(\mathfrak m_{j_i+1})$) with starting point $\nu^{a+1}\rho$; or
%\item $\underline{\mathfrak{Us}}(\mathfrak m_{j_i})-\Delta'$, where $\Delta'$ is the unique segment in $\underline{\mathfrak{Us}}(\mathfrak m_{j_i})$ with starting point $\nu^{a+1}\rho$; or
%\item $\underline{\mathfrak{Us}}(\mathfrak m_i)-\Delta^*+{}^-\Delta^*+\Delta''$, where $\Delta''$ is the unique segment in $\underline{\mathfrak{Us}}(\mathfrak m_{i+1})$ with the starting point $\nu^a\rho$ (if such segment exists); or
%\item $\underline{\mathfrak{Us}}(\mathfrak m_i)-\Delta^*+{}^-\Delta^*$.
%\end{enumerate}
%Either Case (3) or (4) can happen and happen only precisely for one $i$, say $i^*$.   Indeed, if $|\mathfrak n|=|\mathfrak m|$; or more general the number of upward sequences for $\mathfrak m_{[a,b]_{\rho}}$ is equal to the number of upward sequences for $\mathfrak n_{[a,b]_{\rho}}$, we have $j_i=i$ for all $i$. For general situation, $j_i=i$ or $i+1$. 
\end{lemma}

\begin{proof}
We first briefly explain the part that $\nu^{a+1}\rho$ is in the removable free section of $\Delta^*$. Suppose not. Let $\widetilde{\Delta}$ be the unique segment in $\underline{\mathfrak{Us}}(\mathfrak m_{i^*})[a+2]$. Now, for $i<i^*$, if there exists a (unique) segment in $\underline{\mathfrak{Us}}(\mathfrak m_i)[a+1]$ such that it is linked to $\widetilde{\Delta}$, then our choice on $\Delta^*$ guarantees that $\underline{\mathfrak{Us}}(\mathfrak m_i)[a+2] \neq \emptyset$. However, now one readily uses the segments in $\underline{\mathfrak{Us}}(\mathfrak m_i)[a+2]$ to carry out the removal step in the $\mathfrak{tds}(\mathfrak m, a+1)$-process, and sees that $\Delta^*$ is also removed. This again gives contradiction to our choice of $\Delta^*$.  

We now consider the general formula of $\underline{\mathfrak{Us}}(\mathfrak n_i)$. One has to observe that if the first bullet of (3) happens, then $\underline{\mathfrak{Us}}(\mathfrak m_{i})[a+2]\neq \emptyset$, which one can prove inductively. We remark that it is possible that for some $i>i^*$, $\underline{\mathfrak{Us}}(\mathfrak m_{i})[a+1]\neq \emptyset$ and $\underline{\mathfrak{Us}}(\mathfrak m_{i})[a+2]=\emptyset$. However, in such case, one has some $i^*<i'<i$ such that $\underline{\mathfrak{Us}}(\mathfrak m_{i'})[a+1]=\emptyset$, and so one will get to the case of the second bullet of (3) before reaching such $i$. Proving such situation is again quite straightforward, while it is not immediate in some cases. 
\end{proof}

\begin{example} For $\mathfrak{m}\in \mathrm{Mult}_\rho$ with $[a,b]_\rho \in \mathrm{Seg}_\rho$, let $\mathfrak{n}=\mathcal{D}_{[a+1]_\rho}^\mathrm{Lang}(\mathfrak{m}) \neq \infty$. Set $\mathfrak{m}_1=\mathfrak{m}_{[a,b]}$ (resp. $\mathfrak{n}_1=\mathfrak{n}_{[a,b]}$) and recursively for $i>1$, $\mathfrak{m}_i=\mathfrak{m}_{i-1} - \underline{\mathfrak{Us}}(\mathfrak m_{i-1})$ (resp. $\mathfrak{n}_i=\mathfrak{n}_{i-1} - \underline{\mathfrak{Us}}(\mathfrak n_{i-1})$).
\begin{enumerate}
\item[(i)] Let $\mathfrak m=\left\{[1,5]_{\rho}, [2,4]_{\rho}, [2,5]_{\rho}, [3,5]_{\rho}, [5,6]_{\rho} \right\}$ with $a=1$ and $b=4$. Let $\mathfrak n=\mathcal D^{\mathrm{Lang}}_{[2]_{\rho}}(\mathfrak m)$. In this case, $\Delta^*=[2,5]_{\rho}$ and $i^*=2$. Then we have $\underline{\mathfrak{Us}}(\mathfrak m_1)=\left\{ [1,5]_{\rho}, [5,6]_{\rho}\right\}$, $\underline{\mathfrak{Us}}(\mathfrak m_2)=\left\{ [2,5]_{\rho} \right\}$ and $\underline{\mathfrak{Us}}(\mathfrak m_3)=\left\{[2,4]_{\rho}, [3,5]_{\rho}\right\}$. On the other hand, $\underline{\mathfrak{Us}}(\mathfrak n_1)=\left\{ [1,5]_{\rho}, [5,6]_{\rho}\right\}$, $\underline{\mathfrak{Us}}(\mathfrak n_2)=\left\{  [2,4]_{\rho}, [3,5]_{\rho}\right\}$ and $\underline{\mathfrak{Us}}(\mathfrak n_3)=\left\{ [3,5]_{\rho}\right\}$. 
\item[(ii)] Let $\mathfrak m=\left\{ [1,5]_{\rho}, [2,6]_{\rho}, [2,3]_{\rho}, [3,4]_{\rho}\right\}$ with $a=1$ and $b=3$, and let $\mathfrak n=\mathcal D_{[2]_{\rho}}^{\mathrm{Lang}}(\mathfrak m)$. In this case, $\Delta^*=[2,6]_{\rho}$ and $i^*=1$. We also have $\underline{\mathfrak{Us}}(\mathfrak m_1)=\left\{ [1,5]_{\rho}, [2,6]_{\rho}\right\}$, $\underline{\mathfrak{Us}}(\mathfrak m_2)=\left\{ [2,3]_{\rho}, [3,4]_{\rho}\right\}$. On the other hand, $\underline{\mathfrak{Us}}(\mathfrak n_1)=\left\{ [1,5]_{\rho}, [3,6]_{\rho} \right\}$ and $\underline{\mathfrak{Us}}(\mathfrak n_2)=\left\{ [2,3]_{\rho}, [3,4]_{\rho}\right\}$. 
%\item Let $\mathfrak m=\left\{ [1,6]_{\rho}, [1,5]_{\rho}, [2,5]_{\rho}, [2,7]_{\rho},  [3,7]_{\rho}\right\}$. In such case, $\Delta^*=[2,7]_{\rho}$ and $i^*=1$. We then have $\mathfrak{Us}(\mathfrak m_1)=\left\{ [1,6]_{\rho}, [2,7]_{\rho} \right\}$, $\mathfrak{Us}(\mathfrak m_2)=\left\{ [1,5]_{\rho}, [3,7]_{\rho}\right\}$, $\mathfrak{Us}(\mathfrak m_3)=\left\{ [2,5] \right\}$. On the other hand, $\mathfrak{Us}(\mathfrak n_1)=\left\{ [1,6]_{\rho}, [3,7]_{\rho} \right\}$, $\mathfrak{Us}(\mathfrak n_2)=\left\{ [1,5]_{\rho}, [3,7]_{\rho} \right\}$ and $\underline{\mathfrak{Us}}(\mathfrak n_3)=\left\{[2,5]_{\rho}\right\}$.
\end{enumerate}
\end{example}

We also record the following observation:

\begin{lemma} \label{lem lower for [a,b] a+1}
We use the notations in Lemma \ref{lem upward seq expression a+1}.  Recall that $[a_{i_{\ell-1},j_{\ell-1}}, b_{i_{\ell-1},j_{\ell-1}}]_{\rho}$ is the second last segment (if exists) in the truncation process for $\mathcal D^{\mathrm{Lang}}_{[a,b]_{\rho}}(\mathfrak m)$. If $a_{i_{\ell-1}, j_{\ell-1}}=a+1$, then 
either one of the following holds:
\begin{enumerate}
    \item $i_{\ell-1}\leq i^*$; or
    \item there exists $i^*< i' < i_{\ell-1}$ such that $\underline{\mathfrak{Us}}(\mathfrak m_{i'})[a+1]=\emptyset$. In particular, $i_{\ell-1}>i^*+1$.
   % \item there exists $i^*<i' \leq i_{\ell-1}$ such that $\underline{\mathfrak{Us}}(\mathfrak m_{i'})[a+1]\neq \emptyset$, $\underline{\mathfrak{Us}}(\mathfrak m_{i'-1})[a]\neq \emptyset$ and the unique segment in $\underline{\mathfrak{Us}}(\mathfrak m_{i'})[a+1]$ is not linked to the unique segment in $\underline{\mathfrak{Us}}(\mathfrak m_{i'-1})[a]$.
\end{enumerate} 

\end{lemma}

\begin{proof}
Suppose not. Then, for all $i^*<i' \leq i_{\ell-1}$, $\underline{\mathfrak{Us}}(\mathfrak m_{i'})[a+1] \neq \emptyset$. 

From how we pick $\Delta^*$ and the $\mathfrak{tds}(-,a+1)$-process, we must have $i_{\ell-1}-i^*$ many segments $\Delta^1, \ldots, \Delta^{i_{\ell-1}-i^*}$ in $\mathfrak m[a+2]$ satisfying the following two properties:
\begin{enumerate}
\item each $\Delta^i$ is not linked to $\Delta^*$; and
\item for each $i$, $\Delta^i$ is linked to the unique segment in $\underline{\mathfrak{Us}}(\mathfrak m_i)[a+1]$.
\end{enumerate}
However, those segments $\Delta^i$ then force that $\underline{\mathfrak{Us}}(\mathfrak m_i)[a+2]\neq \emptyset$ for $i^*<i \leq i_{\ell}$. This contradicts that the segment $[a_{i_{\ell-1}, j_{\ell-1}}, b_{i_{\ell-1},j_{\ell-1}}]_{\rho}$ has $\nu^{a+1}\rho$ in its removable free section.
\end{proof}

\begin{lemma}\label{lem:comm_D_ab:D_a+1}
Let $\mathfrak{m}\in \mathrm{Mult}_\rho$ and $[a,b]_\rho \in \mathrm{Seg}_\rho$ with $a < b$ and $\mathcal{D}_{[a+1]_\rho}^\mathrm{Lang}(\mathfrak{m}) \neq \infty$. Then,
\begin{itemize}
    \item[(i)] If $\mathcal{D}_{[a,b]_\rho}^\mathrm{Lang}(\mathfrak{m}) \neq \infty$, we have 
$\mathcal{D}^\mathrm{Lang}_{[a+1]_\rho} \circ \mathcal{D}_{[a,b]_\rho}^\mathrm{Lang}(\mathfrak{m}) =  \mathcal{D}_{[a,b]_\rho}^\mathrm{Lang} \circ \mathcal{D}^\mathrm{Lang}_{[a+1]_\rho}  (\mathfrak{m})\neq \infty.$ 
   \item[(ii)] If $\mathcal{D}_{[a,b]_\rho}^\mathrm{Lang} \circ \mathcal{D}^\mathrm{Lang}_{[a+1]_\rho}  (\mathfrak{m})\neq \infty$, we have $\mathcal{D}_{[a,b]_\rho}^\mathrm{Lang}(\mathfrak{m}) \neq \infty.$
\end{itemize}
\end{lemma}

Since an argument for Lemma \ref{lem:comm_D_ab:D_a+1} has a similar nature to the one of Lemma \ref{lem:comm_D[ab]_D[a]}, we shall be a bit sketchy.  \\
\noindent
{\it Sketch of a proof of Lemma \ref{lem:comm_D_ab:D_a+1}.} We have that 
\[ \mathfrak n:=\mathcal D_{[a+1]_{\rho}}^{\mathrm{Lang}}(\mathfrak m) =\mathfrak m-\Delta_{a+1}+{}^-\Delta_{a+1}
\]
for some $\Delta_{a+1} \in \mathfrak m[a+1]$. We use the notations as mentioned in Algorithm \ref{alg:der:Lang}. The collection of upward sequences for $\mathfrak n_{[a,b]}$ is described in Lemma \ref{lem upward seq expression a+1}.

The first situation is that $\Delta_{a+1}\notin \mathfrak m_{[a,b]}$.  In such case, $\mathfrak m_{[a,b]}= \mathfrak n_{[a,b]}$. Thus, the upward sequences for $\mathfrak m_{[a,b]}$ are the same as the upward sequences for $\mathfrak n_{[a,b]}$. Thus, one remains to investigate the $\mathfrak{tds}\left(\mathcal D^{\mathrm{Lang}}_{[a,b]_{\rho}}(\mathfrak m), a+1\right)$-process. Now, one carries out similar considerations as in Lemma \ref{lem tds on a ab commute}. Some examples are given in the Appendix \ref{s example tds process}.

The second situation is that $\Delta_{a+1} \in \mathfrak m_{[a,b]}$. In this case, one further considers whether $\Delta_{a+1}$ is a segment participating in the truncation process for $\mathcal D^{\mathrm{Lang}}_{[a,b]_{\rho}}(\mathfrak m)$. A complete argument is again routine, but slightly tedious. However, the general principle is that if one wants to compare the segments participating in the truncation process for $\mathcal D^{\mathrm{Lang}}_{[a,b]_{\rho}}(\mathfrak m)$ and $\mathcal D^{\mathrm{Lang}}_{[a,b]_{\rho}}(\mathfrak n)$, one uses Lemmas \ref{lem upward seq expression a+1} and \ref{lem lower for [a,b] a+1}. If one wants to compare the segments participating in the removal steps of the $\mathfrak{tds}(\mathfrak m, a+1)$ process and $\mathfrak{tds}(\mathcal D^{\mathrm{Lang}}_{[a,b]_{\rho}}(\mathfrak m), a+1)$ process, one applies a similar consideration in the previous paragraph and Lemma \ref{lem tds on a ab commute}. Now, with the segments participating in both processes, one can apply definitions to show the lemma.

\begin{example} \label{ex commute on a+1 case}
\begin{enumerate}
\item Let $\mathfrak m=\left\{ [1,5]_{\rho}, [2,6]_{\rho}, [1,4]_{\rho}, [2,3]_{\rho}, [2]_{\rho}\right\}$ with $a=1$ and $b=3$. Then 
\[\mathcal D^{\mathrm{Lang}}_{[2]_{\rho}}(\mathfrak m)=\mathfrak m -\left\{ [2]_{\rho} \right\}=\left\{ [1,5]_{\rho}, [2,6]_{\rho}, [1,4]_{\rho}, [2,3]_{\rho} \right\}\] and $\mathcal D^{\mathrm{Lang}}_{[1,3]_{\rho}}\circ \mathcal D_{[2]_{\rho}}^{\mathrm{Lang}}(\mathfrak m)=\mathcal D_{[2]_{\rho}}^{\mathrm{Lang}}(\mathfrak m)-[2,3]_{\rho}-[1,4]_{\rho}+[2,4]_{\rho}=\left\{[1,5]_{\rho}, [2,6]_{\rho}, [2,4]_{\rho}\right\}$. 
\item $\mathcal D^{\mathrm{Lang}}_{[1,3]_{\rho}}(\mathfrak m)=\mathfrak m-[2,3]_{\rho}-[1,4]_{\rho}+[2,4]_{\rho}$ and  $\mathcal D^{\mathrm{Lang}}_{[2]_{\rho}}\circ \mathcal D^{\mathrm{Lang}}_{[1,3]_{\rho}}(\mathfrak m)=\mathcal D^{\mathrm{Lang}}_{[1,3]_{\rho}}(\mathfrak m)-[2]_{\rho}=\left\{[1,5]_{\rho}, [2,6]_{\rho}, [2,4]_{\rho}\right\}$.
\end{enumerate}
\end{example}

Note that the segments participating in the truncation process are the same no matter the order of the derivatives $\mathcal D_{[2]}^{\mathrm{Lang}}$ and $\mathcal D_{[1,3]}^{\mathrm{Lang}}$. We now present another example:

\begin{example}
Let $\mathfrak m=\left\{ [0,2]_{\rho}, [0,6]_{\rho}, [1,3]_{\rho}, [1,5]_{\rho}  \right\}$ with $a=0$ and $b=2$. Then 
\begin{enumerate}
 \item $\mathcal D^{\mathrm{Lang}}_{[1]_{\rho}}(\mathfrak m)=\mathfrak m-[1,3]_{\rho}+[2,3]_{\rho}=\left\{ [0,2]_{\rho}, [0,6]_{\rho}, [2,3]_{\rho}, [1,5]_{\rho}  \right\} $ and $\mathcal D^{\mathrm{Lang}}_{[0,2]_{\rho}}\circ \mathcal D^{\mathrm{Lang}}_{[1]_{\rho}}(\mathfrak m)=\mathcal D^{\mathrm{Lang}}_{[1]_{\rho}}(\mathfrak m)-[0,6]_{\rho}+[1,6]_{\rho}-[1,5]_{\rho}+[2,5]_{\rho}-[2,3]_{\rho}+[3]_{\rho}=\left\{ [0,2]_{\rho}, [1,6]_{\rho}, [3]_{\rho}, [2,5]_{\rho}  \right\} $.
 \item $\mathcal D^{\mathrm{Lang}}_{[0,2]_{\rho}}(\mathfrak m)=\mathfrak m-[0,6]_{\rho}+[1,6]_{\rho}-[1,3]_{\rho}+[3]_{\rho}=\left\{ [0,2]_{\rho}, [1,6]_{\rho}, [3]_{\rho}, [1,5]_{\rho}\right\}$ and $\mathcal D^{\mathrm{Lang}}_{[1]_{\rho}}\circ \mathcal D^{\mathrm{Lang}}_{[0,2]_{\rho}}(\mathfrak m)=\mathcal D^{\mathrm{Lang}}_{[0,2]_{\rho}}(\mathfrak m)-[1,5]_{\rho}+[2,5]_{\rho}=\left\{ [0,2]_{\rho}, [1,6]_{\rho}, [3]_{\rho}, [2,5]_{\rho}  \right\}$.
\end{enumerate}
\end{example}

\subsection{Commutativity of $\mathrm{D}^{\mathrm{Lang}}_{[a,b]_{\rho}}$ and $\mathrm{D}^{\mathrm{Lang}}_{[a+1]_{\rho}}$}

\begin{lemma}\label{lem:comm_D_ab:D_a+1_pi}
Let $\pi\in \mathrm{Irr}_\rho$ and $[a,b]_\rho \in \mathrm{Seg}_\rho$ with $a < b$ and $\varepsilon_{[a+1]_\rho}^\mathrm{R}(\pi) \neq 0$. 
\begin{itemize}
    \item[(i)] If $\mathrm{D}_{[a,b]_\rho}^\mathrm{R}(\pi) \neq 0$, we have 
$\mathrm{D}^\mathrm{R}_{[a+1]_\rho} \circ \mathrm{D}_{[a,b]_\rho}^\mathrm{R}(\pi) \cong  \mathrm{D}_{[a,b]_\rho}^\mathrm{R} \circ \mathrm{D}^\mathrm{R}_{[a+1]_\rho}  (\pi)\neq 0.$ 
   \item[(ii)] If $\mathrm{D}_{[a,b]_\rho}^\mathrm{R} \circ \mathrm{D}^\mathrm{R}_{[a+1]_\rho}  (\pi)\neq 0$, we have $\mathrm{D}_{[a,b]_\rho}^\mathrm{R}(\pi) \neq 0.$
\end{itemize}
\end{lemma}
\begin{proof}
    (i) The commutativity part follows from \cite[Lemma 4.4]{Cha_csq}. The non-zeroness part follows from the third bullet of \cite[Theorem 9.3]{Cha_csq}.

    (ii) One can argue similarly as in the proof of Lemma \ref{lem:D[ab]=D[a+1,b]_D[a]_pi}(ii) by using properties of the removal process in Section \ref{ss hd removal}. We omit the details.
\end{proof}

\subsection{Main results}

\begin{theorem}\label{thm:zero_der_Lang}
    Let $\pi=L(\mathfrak{m})$ for some $\mathfrak{m} \in \mathrm{Mult}_\rho$ and let $[a,b]_\rho\in \mathrm{Seg}_\rho$. Then, \[\mathcal{D}_{[a,b]_\rho}^\mathrm{Lang}(\mathfrak{m}) \neq \infty ~ \text{ if and only if }~ \mathrm{D}^\mathrm{R}_{[a,b]_\rho}(\pi)\neq 0.\]
\end{theorem}
\begin{proof}
We prove the result by induction on $\ell_{rel}([a,b]_\rho)$ and $\ell_{rel}(\mathfrak m)$. If $a = b$, the statement follows from Theorem \ref{thm:der:Lang:length=1}. Assume $a < b$. We divide the proof into the following cases:

Case 1: $\varepsilon^\mathrm{R}_{[a]_{\rho}}(L(\mathfrak m))\geq 2$. In such cases, we conclude by
	\begin{align*}
		\mathcal D^\mathrm{Lang}_{{[a,b]_{\rho}}}(\mathfrak m) \neq \infty 
		\Longleftrightarrow   & \mathcal D^\mathrm{Lang}_{{[a,b]_{\rho}}} \circ \mathcal D^\mathrm{Lang}_{[a]_{\rho}}(\mathfrak m) \neq \infty \quad (\text{by Lemma } \ref{lem:comm_D[ab]_D[a]})\\
		{\Longleftrightarrow}  &  \mathrm{D}^\mathrm{R}_{[a,b]_{\rho}} \left(L \left( \mathcal D^\mathrm{Lang}_{[a]_{\rho}}(\mathfrak m)\right)\right) \neq 0 \quad (\text{by induction assumption})\\
        {\Longleftrightarrow}  &  \mathrm{D}^\mathrm{R}_{[a,b]_{\rho}}\circ \mathrm{D}^\mathrm{R}_{[a]_{\rho}}(L(\mathfrak m)) \neq 0 \quad (\text{by Theorem } \ref{thm:der:Lang:length=1})\\
		\Longleftrightarrow & \mathrm{D}^\mathrm{R}_{[a,b]_{\rho}}(L(\mathfrak m))\neq 0  \quad(\text{by Lemma } \ref{lem:comm_D[ab]_D[a]_pi})
	\end{align*}

 Case 2: $\varepsilon^\mathrm{R}_{[a]_{\rho}}(L(\mathfrak m))= 1$ and $\varepsilon^\mathrm{R}_{[a+1]_{\rho}}(L(\mathfrak m))= 0$. In such cases, we conclude by
	\begin{align*}
		\mathcal D^\mathrm{Lang}_{[a,b]_\rho}(\mathfrak m)\neq \infty 
		\Longleftrightarrow & \mathcal D^\mathrm{Lang}_{[a+1,b]_\rho}\circ \mathcal D^\mathrm{Lang}_{[a]_{\rho}}(\mathfrak m) \neq \infty \quad (\text{by Lemma } \ref{lem:D[ab]=D[a+1,b]_D[a]})\\
		{\Longleftrightarrow} & \mathrm D^\mathrm{R}_{[a+1,b]_\rho}\circ \mathrm D^\mathrm{R}_{[a]_{\rho}}(L(\mathfrak m)) \neq 0 \quad (\text{by induction assumption})\\
		\Longleftrightarrow & \mathrm D^\mathrm{R}_{[a,b]_\rho}(L(\mathfrak m)) \neq 0 \quad (\text{by Lemma } \ref{lem:D[ab]=D[a+1,b]_D[a]_pi} )
	\end{align*}

Case 3: $\varepsilon^\mathrm{R}_{[a]_{\rho}}(L(\mathfrak m))= 1$ and $\varepsilon^\mathrm{R}_{[a+1]_{\rho}}(L(\mathfrak m))\neq 0$. In such cases, we conclude by
    \begin{align*}
    	\mathcal D^\mathrm{Lang}_{[a,b]}(\mathfrak m)\neq \infty 
    	\Longleftrightarrow & \mathcal D^\mathrm{Lang}_{[a,b]_\rho}\circ \mathcal D^\mathrm{Lang}_{[a+1]_\rho}(\mathfrak m) \neq \infty \quad (\text{by Lemma } \ref{lem:comm_D_ab:D_a+1})\\
      {\Longleftrightarrow} & \mathrm D^\mathrm{R}_{[a,b]_\rho}\circ \mathrm D^\mathrm{R}_{[a+1]_\rho}(L(\mathfrak m)) \neq 0 \quad (\text{by induction assumption}) \\
    	\Longleftrightarrow & \mathrm D^\mathrm{R}_{[a,b]_\rho}(L(\mathfrak m)) \neq 0 \quad (\text{by Lemma } \ref{lem:comm_D_ab:D_a+1_pi})
    \end{align*}

Case 4: $\varepsilon^\mathrm{R}_{[a]_{\rho}}(L(\mathfrak m))=0$. Then $\mathcal D^\mathrm{Lang}_{[a]_\rho}(\mathfrak m)=\infty$ and $\mathrm{D}^\mathrm{R}_{[a]_\rho}(L(\mathfrak m))=0$. The first equality and Theorem \ref{thm:der:Lang:length=1} imply that $\nu^a\rho$ is not in $\mathfrak{rf}(\Delta_{i,j})$ for any segment $\Delta_{i,j}$ involved in the upward sequences of Algorithm \ref{alg:der:Lang} applied to $\mathfrak m$ and so $\mathcal D_{[a,b]_{\rho}}^{\mathrm{Lang}}(\mathfrak m)=\infty$, while the second equality implies $\mathfrak{hd}(\pi)[a]=\emptyset$ and so  $\mathrm{D}^\mathrm{R}_{[a,b]_{\rho}}(\pi)=0$ by Lemma \ref{lem non-zero der hd}. 
\end{proof}

\begin{lemma}\label{lem:iso_der=>iso_rep}
    Let $\tau_1, \tau_2 \in \mathrm{Irr}_{\rho}$. Let $\Delta \in \mathrm{Seg}_{\rho}$ such that $\mathrm{D}^\mathrm{R}_{\Delta}(\tau_1) \cong \mathrm{D}^\mathrm{R}_{\Delta}(\tau_2)\neq 0$. Then, $\tau_1 \cong \tau_2.$
\end{lemma}
\begin{proof}
    Apply the integral $\mathrm{I}^\mathrm{R}_\Delta$ on both sides of $\mathrm{D}^\mathrm{R}_{\Delta}(\tau_1) \cong \mathrm{D}^\mathrm{R}_{\Delta}(\tau_2)$ to get the result.
\end{proof}

\begin{theorem}\label{thm:der:Lang}
Let $\mathfrak{m} \in \mathrm{Mult}_\rho$ and $\Delta=[a,b]_\rho \in \mathrm{Seg}_\rho$. If $\mathcal{D}_\Delta^\mathrm{Lang}(\mathfrak{m}) \neq \infty$, we have
\[\mathrm{D}^\mathrm{R}_\Delta \left(L(\mathfrak{m}) \right) \cong L \left( \mathcal{D}_\Delta^\mathrm{Lang}(\mathfrak{m})  \right).\]  
\end{theorem}
\begin{proof}
We use an induction argument on the relative length $\ell_{rel}(\Delta)$, and $\ell_{rel}(\mathfrak{m})$ of $\mathfrak{m}$ to prove this result. By Theorem \ref{thm:der:Lang:length=1}, for any $\mathfrak{m}^\prime \in \mathrm{Mult}_\rho$, we have 
\begin{equation}\label{Eq:der_Lang_A}
\mathrm{D}^\mathrm{R}_{[a]_\rho}(L(\mathfrak{m}^\prime)) \cong L \left(\mathcal{D}^\mathrm{Lang}_{[a]_\rho}(\mathfrak{m}^\prime)\right).   
\end{equation}
As $\mathcal{D}_{[a,b]_\rho}^\mathrm{Lang}(\mathfrak{m}) \neq \infty$, we have $\varepsilon^\mathrm{R}_{[a,b]_\rho}(L(\mathfrak{m})) \neq 0$,  $\mathcal{D}_{[a]_\rho}^\mathrm{Lang}(\mathfrak{m}) \neq \infty$, and so, 
$\varepsilon^\mathrm{R}_{[a]_\rho}(L(\mathfrak{m})) \neq 0$. These show the case when $\ell_{\mathrm{rel}}(\Delta)=1$.

Case 1. $\varepsilon^\mathrm{R}_{[a]_\rho}(L(\mathfrak{m})) \geq 2$. As an inductive step, we assume that 
\begin{equation*}
 \mathrm{D}^\mathrm{R}_{[a,b]_\rho}(L(\mathfrak{n})) \cong L \left(\mathcal{D}^\mathrm{Lang}_{[a,b]_\rho}(\mathfrak{n})\right).  
\end{equation*}
for any $\mathfrak{n} \in \mathrm{Mult}_\rho$ with $\ell_{rel}(\mathfrak{n}) < \ell_{rel}(\mathfrak{m})$. Put $\mathfrak{n}=\mathcal{D}^\mathrm{Lang}_{[a]_\rho}(\mathfrak{m})$. As $\ell_{rel}(\mathfrak{n})<\ell_{rel}(\mathfrak{m})$, we have
\begin{equation}\label{Eq:der_Lang_B}
\mathrm{D}^\mathrm{R}_{[a,b]_\rho}\left(L\left(\mathcal{D}^\mathrm{Lang}_{[a]_\rho}(\mathfrak{m})\right)\right) \cong L \left(\mathcal{D}^\mathrm{Lang}_{[a,b]_\rho}\circ\mathcal{D}^\mathrm{Lang}_{[a]_\rho}(\mathfrak{m})\right).  
\end{equation}
Then, we get
\begin{align*}
\mathrm{D}^\mathrm{R}_{[a]_\rho}\left( L\left(\mathcal{D}^\mathrm{Lang}_{[a,b]_\rho}(\mathfrak{m}) \right)\right) &\cong L\left(\mathcal{D}^\mathrm{Lang}_{[a]_\rho}\circ \mathcal{D}^\mathrm{Lang}_{[a,b]_\rho} (\mathfrak{m})\right) \quad (\text{by }\eqref{Eq:der_Lang_A})\\
 &\cong L\left(\mathcal{D}^\mathrm{Lang}_{[a,b]_\rho}\circ \mathcal{D}^\mathrm{Lang}_{[a]_\rho} (\mathfrak{m})\right) \quad (\text{by Lemma }\ref{lem:comm_D[ab]_D[a]})\\
 &\cong \mathrm{D}^\mathrm{R}_{[a,b]_\rho}\left(L\left(\mathcal{D}^\mathrm{Lang}_{[a]_\rho}(\mathfrak{m})\right)\right) \quad (\text{by }\eqref{Eq:der_Lang_B})\\
 &\cong \mathrm{D}^\mathrm{R}_{[a,b]_\rho} \circ  \mathrm{D}^\mathrm{R}_{[a]_\rho}\left( L\left(\mathfrak{m}\right)\right) \quad (\text{by }\eqref{Eq:der_Lang_A})\\
 &\cong \mathrm{D}^\mathrm{R}_{[a]_\rho} \circ  \mathrm{D}^\mathrm{R}_{[a,b]_\rho}\left( L\left(\mathfrak{m}\right)\right) \quad (\text{by Lemma } \ref{lem:comm_D[ab]_D[a]_pi}).
\end{align*}
Therefore, by Lemma \ref{lem:iso_der=>iso_rep}, we conclude that 
 $\mathrm{D}^\mathrm{R}_{[a,b]_\rho}\left( L\left(\mathfrak{m}\right)\right)\cong L\left(\mathcal{D}^\mathrm{Lang}_{[a,b]_\rho}(\mathfrak{m}) \right).$

Case 2. $\varepsilon^\mathrm{R}_{[a]_\rho}(L(\mathfrak{m})) =1$.
%As an inductive step, we assume that for $\mathfrak{n} \in \mathrm{Mult}_\rho$,
%\begin{equation}\label{Eq:der_Lang_C}
% \mathrm{D}^\mathrm{R}_{[a+1,b]_\rho}(L(\mathfrak{n})) \cong L \left(\mathcal{D}^\mathrm{Lang}_{[a+1,b]_\rho}(\mathfrak{n})\right). 
%\end{equation}
In such case, we conclude by
\begin{align}\label{Eq:der_Lang_E}
 L\left(\mathcal{D}^\mathrm{Lang}_{[a,b]_\rho}(\mathfrak{m}) \right) &= L\left(\mathcal{D}^\mathrm{Lang}_{[a+1,b]_\rho}\circ \mathcal{D}^\mathrm{Lang}_{[a]_\rho} (\mathfrak{m})\right) \quad (\text{by Lemma } \ref{lem:D[ab]=D[a+1,b]_D[a]}) \nonumber\\
 &\cong \mathrm{D}^\mathrm{R}_{[a+1,b]_\rho} \left( L\left( \mathcal{D}^\mathrm{Lang}_{[a]_\rho} (\mathfrak{m})\right)\right)\\
 &\cong \mathrm{D}^\mathrm{R}_{[a+1,b]_\rho} \circ \mathrm{D}^\mathrm{R}_{[a]_\rho} \left( L\left( \mathfrak{m}\right)\right) \quad (\text{by  }\eqref{Eq:der_Lang_A})\nonumber\\
 &\cong \mathrm{D}^\mathrm{R}_{[a,b]_\rho} \left( L\left(\mathfrak{m}\right)\right) \quad (\text{by Lemma } \ref{lem:D[ab]=D[a+1,b]_D[a]_pi}).\nonumber
\end{align}
Here, the isomorphism \eqref{Eq:der_Lang_E} follows from induction as $\ell_{rel}([a+1,b]_{\rho})<\ell_{rel}([a,b]_{\rho})$.
\end{proof}

\subsection{Left derivative algorithm}

For $\mathfrak m \in \mathrm{Mult}_{\rho}$ and $[a,b]_{\rho} \in \mathrm{Seg}_\rho$, we define 
\[ \mathcal D^{\mathrm{Lang}, \mathrm L}_{[a,b]_{\rho}}(\mathfrak m)=\Theta\left(\mathcal D^{\mathrm{Lang}}_{[-b,-a]_{\rho^{\vee}}}(\Theta(\mathfrak m))\right) .
\]
Now, with Theorems \ref{thm:zero_der_Lang} and \ref{thm:der:Lang} and discussions in Section \ref{sss gk invol}, we have:

\begin{theorem}
Suppose $\mathfrak{m} \in \mathrm{Mult}_\rho$ and $\Delta \in \mathrm{Seg}_\rho$. Then, the following holds:
\begin{enumerate}
\item $\mathcal{D}_\Delta^{\mathrm{Lang}, \mathrm{L}}(\mathfrak{m}) \neq \infty$ if and only if $\mathrm D^{\mathrm L}_{\Delta}(L(\mathfrak m))\neq 0$; and
\item if $\mathcal{D}_\Delta^{\mathrm{Lang}, \mathrm{L}}(\mathfrak{m}) \neq \infty$, we have $\mathrm{D}^\mathrm{L}_\Delta \left(L(\mathfrak{m}) \right) \cong L \left( \mathcal{D}_\Delta^{\mathrm{Lang}, \mathrm{L}}(\mathfrak{m})  \right)$.
\end{enumerate}
\end{theorem}

%\begin{remark}
%   An algorithm for the left derivative in Langlands classification can be easily written as a mirror symmetry to Algorithm \ref{alg:der:Lang}, we omit to describe it here. 
%\end{remark}
%%%%%%%%%%%%%%%%%%%%%%In Zelevinsky classification %%%%%%%%%%%%%%%%%%%%%%%%%%%%%%

\section{Derivatives in Zelevinsky classification}\label{sec:der_Zel}
In this section, we present an algorithm (refer to Algorithm \ref{alg:der_Zel}) for computing the $\mathrm{St}$-derivatives of irreducible representations of $\mathrm{GL}_n(F)$ in the Zelevinsky classification. Similar to the proof of the $\mathrm{St}$-derivative in the Langlands classification, we could have offered an inductive proof, where the $\rho$-derivative in the Zelevinsky classification and several results akin to Lemmas \ref{lem:D[ab]=D[a+1,b]_D[a]}, \ref{lem:comm_D[ab]_D[a]} and \ref{lem:comm_D_ab:D_a+1} would be required. Here, we avoid that approach and use the M\oe glin-Waldspurger (MW) algorithm for computing derivatives. By applying this algorithm, we can derive Corollary \ref{prop:derivative}, and then combine it with a reduction in Section \ref{ss reduction to max cus} to prove our main algorithm (refer to Theorem \ref{thm alg der zel}).

\subsection{MW algorithm for derivatives}
For $\mathfrak{m} \in \mathrm{Mult}_\rho$,  define the multisegment $\mathcal{D}^\mathrm{MW}(\mathfrak{m})$ associated to $\mathfrak{m}$ in the following way: let $b$ be the largest integer such that $\mathfrak{m}\langle b \rangle \neq \emptyset$ that means $\nu^b\rho$ is the maximal cuspidal support of $\mathfrak m$. Then, we choose the shortest segment $\Delta_0$ in $\mathfrak{m}\langle b \rangle$. Recursively for $1 \leq s \leq k$, we choose the shortest segment $\Delta_{s}$ in $\mathfrak{m}\langle b-s \rangle$ such that $\Delta_{s} \prec \Delta_{s-1}$ and $k$ is the largest possible integer for which such $\Delta_k$ exists. Define the first segment in $\mathfrak{m}^\#$ (produced by MW algorithm applied to $\mathfrak{m}$) as: 
\[\Delta(\mathfrak{m})=\{ \nu^{b-k}\rho, \nu^{b-k+1}\rho,...,\nu^b \rho\}=[b-k, b]_\rho,\] and the reduced multisegment by
\[\mathcal{D}^\mathrm{MW}(\mathfrak{m})=\mathfrak{m}- \sum\limits_{i=0}^k \Delta_i +\sum\limits_{i=0}^k \Delta_i^-.\]
By the MW algorithm in \cite{MW}, we get 
\begin{equation}\label{eq:MW}
	\mathfrak{m}^\#= \Delta(\mathfrak{m}) + \left(\mathcal{D}^\mathrm{MW}(\mathfrak{m})\right)^\#.
\end{equation}
We say that $\Delta_0, \Delta_1, \ldots, \Delta_k$ are the ordered segments participating in the MW algorithm for $\mathfrak m$.  For any $a \leq c$, define $\varepsilon_{[a,c]_{\rho}}^{\mathrm{MW}}(\mathfrak m)$ to be the multiplicity of the segment $[a,c]_{\rho}$ in $\mathfrak m^{\#}$. It can be shown that $\varepsilon_{[a,c]_{\rho}}^{\mathrm{MW}}(\mathfrak m)=\varepsilon^\mathrm{R}_{[a,c]_{\rho}}(Z(\mathfrak m))$, which explains the notion of $\varepsilon^{\mathrm{MW}}_{[a,c]_{\rho}}$.

%We may denote the multisegment $\mathcal{D}^\mathrm{MW}(\mathfrak{m})$ by $\mathcal{D}^\mathrm{MW}_{\Delta(\mathrm{m})}(\mathfrak{m})$.

\begin{proposition}\label{prop:MW_der} \cite{MW}
    Let $\mathfrak{m} \in \mathrm{Mult}_\rho$ and $\Delta(\mathfrak{m})$ be the first segment produced in the MW algorithm for $\mathfrak m$. Then, we have \[\mathrm{D}^\mathrm{R}_{\Delta(\mathfrak{m})}(Z (\mathfrak{m})) \cong Z \left( \mathcal{D}^\mathrm{MW}(\mathfrak{m}) \right).\]  
\end{proposition}
\begin{proof}
%By the definition of $\mathrm{D}^\mathrm{R}_{\Delta(\mathfrak m)}$ and Frobenius reciprocity, 
%\begin{align} \label{eqn derivative 1}
%Z(\mathfrak m) \hookrightarrow \mathrm{D}^{\mathrm R}_{\Delta(\mathfrak m)}(Z(\mathfrak m)) \times \mathrm{St}(\Delta(\mathfrak m)) .
%\end{align}

By Langlands classification and discussions in Section \ref{ss zel invol},
\[ Z(\mathfrak m) \cong L(\mathfrak m^{\#}) \hookrightarrow \lambda(\mathfrak m^{\#})=\lambda(\mathcal D^{\mathrm{MW}}(\mathfrak m)^{\#})\times \mathrm{St}(\Delta(\mathfrak m)) .
\]
As the submodule of $\lambda(\mathcal D^{\mathrm{MW}}(\mathfrak m)^{\#}) \times \mathrm{St}(\Delta(\mathfrak{m}))$ is unique and the submodule of $\lambda(\mathcal D^{\mathrm{MW}}(\mathfrak m^{\#}))$ is isomorphic to $L(\mathcal{D}^{\mathrm{MW}}(\mathfrak m)^{\#})\cong Z(\mathcal D^{\mathrm{MW}}(\mathfrak m))$, the above map factors through the map
\begin{align} \label{eqn mw deri}  Z(\mathfrak m) \hookrightarrow Z(\mathcal D^{\mathrm{MW}}(\mathfrak m)) \times \mathrm{St}(\Delta(\mathfrak m)) .
\end{align}
Now (\ref{eqn mw deri}) gives that $\mathrm I^{\mathrm R}_{\Delta(\mathfrak m)}(Z(\mathcal D^{\mathrm{MW}}(\mathfrak m))) \cong Z(\mathfrak m)$ and so applying $\mathrm D_{\Delta(\mathfrak m)}^{\mathrm R}$ on both sides, we obtain the proposition.
\end{proof}
%{\color{blue} (Eq. (19) is true when the derivative is non-zero, that means that we are assuming the derivative is non-zero, and giving the proof but actually the proof should imply that this derivative is always non-zero)}

%Let $\mathfrak{n}, \mathfrak{n}^\prime \in \mathrm{Mult}$ with $\mathfrak{n} \nprec \mathfrak{n}^\prime$. Then, we have (see \cite[Prop. 2.5]{LM16}, \& \cite[Def. 2]{LM_inven} for more details), 
%\begin{equation}\label{Eq:LM}
%Z(\mathfrak{n} + \mathfrak{n}^\prime)= \mathrm{soc} \left( Z(\mathfrak{n}) \times Z(\mathfrak{n}^\prime) \right),  \text{ and } L(\mathfrak{n} + \mathfrak{n}^\prime)= \mathrm{soc} \left( L(\mathfrak{n}^\prime) \times L(\mathfrak{n}) \right).    
%\end{equation}
%The M\oe glin-Waldspurger algorithm gives $\mathfrak{m}^\# = \mathcal{D}^\mathrm{MW}(\mathfrak{m})^\# + \Delta(\mathfrak{m}).$ Observe that $\{\Delta(\mathfrak{m})\} \nprec \mathcal{D}^\mathrm{MW}(\mathfrak{m})^\#$. Therefore, applying (\ref{Eq:LM}), we get
%\begin{align*}
% Z(m)=L(\mathfrak{m}^\#) &= \mathrm{soc} \left( L\left(\mathcal{D}^\mathrm{MW}(\mathfrak{m})^\#\right) \times L(\Delta(\mathfrak{m})) \right)\\
% &= \mathrm{soc} \left( Z\left(\mathcal{D}^\mathrm{MW}(\mathfrak{m})\right) \times \mathrm{St}(\Delta(\mathfrak{m})) \right).
%\end{align*}
%Using the Frobenius reciprocity theorem and \cite[Proposition 2.1]{Cha_csq}, we have \[Z\left(\mathcal{D}^\mathrm{MW}(\mathfrak{m})\right) \otimes \mathrm{St}(\Delta(\mathfrak{m})) = \mathrm{soc} \left(\mathrm{Jac}_{N_{\ell_a(\Delta(\mathfrak{m}))}}\left(Z(\mathfrak{m})\right) \right).\]

\begin{example}\label{ex:MW}
Let $\mathfrak{m}=\left\{ [0,2]_\rho, [2,4]_\rho, [2,5]_\rho, [3,5]_\rho, [4,6]_\rho \right\}$. Then, $\Delta(\mathfrak{m})=[4,6]_\rho$ and the multisegment  $\mathcal{D}^\mathrm{MW}(\mathfrak{m})=\left\{ [0,2]_\rho, [2,3]_\rho, [2,5]_\rho, [3,4]_\rho, [4,5]_\rho \right\}$. Therefore, we have $\mathrm{D}^\mathrm{R}_{[4,6]_{\rho}}(Z(\mathfrak m) )=Z\left(\left\{ [0,2]_\rho, [2,3]_\rho, [2,5]_\rho, [3,4]_\rho, [4,5]_\rho \right\}\right)$.\qed
\end{example}
%\begin{lemma}\label{lem:der_not=zero=>} 
%Let $\mathfrak{m} \in \mathrm{Mult}_\rho$ and $\Delta \in \mathrm{Seg}_\rho$. Suppose $e(\Delta) \geq e\left(\Delta(\mathfrak{m})\right)$. Then,  \[\mathrm{D}^\mathrm{R}_{\Delta}(Z(\mathfrak{m})) \neq 0 \Longrightarrow \Delta \in \mathfrak{m}^\#.\]
%\end{lemma}
%\begin{proof}
%Assume $\mathrm{D}^\mathrm{R}_{\Delta}(Z(\mathfrak{m}))=Z(\mathfrak{n})$ for some $\mathfrak{n} \in \mathrm{Mult}_\rho$. By definition of integral, \[Z(\mathfrak{m}) =\mathrm{I}^\mathrm{R}_{\Delta} \circ \mathrm{D}^\mathrm{R}_{\Delta}(Z(\mathfrak{m})) =\mathrm{soc}\left(Z(\mathfrak{n}) \times \mathrm{St}(\Delta) \right).\] Then, $L(\mathfrak{m}^\#) = \mathrm{soc}\left(L(\mathfrak{n}^\#) \times L(\Delta) \right).$ As the maximal cuspidal support of $\mathfrak{n}$ is less than or equal to $e(\Delta)$, it follows that $\Delta \nprec \Delta^\prime$ for any $\Delta^\prime \in \mathfrak{n}^\#$. Therefore, we conclude that
% \[L(\mathfrak{m}^\#) = \mathrm{soc}\left(L(\mathfrak{n}^\#) \times L(\Delta) \right)=L\left(\mathfrak{n}^\# + \Delta \right), \text{ and so, } \Delta \in \mathfrak{m}^\#.\]
%\end{proof}

\subsection{Algorithm for derivatives}
%Let $\mathfrak{m}$ be a multisegment and $\Delta$ be a segment. We associate a multisegment, denoted as $\mathcal{D}^\mathrm{Zel}_{\Delta}(\mathfrak{m})$ to the pair $(\mathfrak{m},\Delta)$ by the following Algorithm \ref{alg:der_Zel} such that the right derivatives $\mathrm{D}^\mathrm{R}_\Delta(Z(\mathfrak{m}))$ of the irreducible representation $Z(\mathfrak{m})$ is given by $Z \left(\mathcal{D}^\mathrm{Zel}_{\Delta}(\mathfrak{m}) \right)$. 

%The algorithm for the left derivative in the Zelevinsky classification will be a mirror symmetry to this right version.

\begin{algorithm}\label{alg:der_Zel}
Let $\mathfrak{m} \in \mathrm{Mult}_\rho$ and $\Delta=[a,b]_\rho \in \mathrm{Seg}_\rho$. Set $\mathfrak{m}_0=\mathfrak{m}$ and perform the following steps:

Step 1. (Upward sequence of maximal linked segments): Define the removable upward sequence of maximal linked segments in neighbors on $\mathfrak{m}_0$ ranging from $a-1$ to $b$ as follows: start with the longest segment ${\Delta}^{a-1}_1$ (if exists) in $\mathfrak{m}_0\left\langle a-1 \right \rangle$. Recursively for $a \leq i \leq b$, we choose the longest segment ${\Delta}^{i}_1$ (if exists) in $\mathfrak{m}_0\left\langle i \right \rangle$ such that ${\Delta}^{i-1}_1 \prec {\Delta}^{i}_1$. Then the sequence ${\Delta}^{a-1}_1 \prec {\Delta}^{a}_1 \prec \cdots \prec {\Delta}^{b}_1$ defines an upward sequence of maximal linked segments in neighbors on $\mathfrak{m}_0$ ranging from $a-1$ to $b$. 

Step 2. (Remove) Replace $\mathfrak{m}_0$ by $\mathfrak{m}_1$ defined by 
\[\mathfrak{m}_1= \mathfrak{m}-\sum\limits_{i=a-1}^b {\Delta}^{i}_1 .\]

Step 3. (Repeat Steps 1 and 2): Again find \big(if it exists, say ${\Delta}^{a-1}_2 \prec {\Delta}^{a}_2 \prec \cdots \prec {\Delta}^{b}_2$\big) the upward sequence of maximal linked segments in neighbors on $\mathfrak{m}_1$ ranging from $a-1$ to $b$, and remove it to get the multisegment    $\mathfrak{m}_2= \mathfrak{m}_1-\sum\limits_{i=a-1}^b {\Delta}^{i}_2 .$ Repeat this removal process until it terminates after a finite number of times, say $k$ times. 

Step 4. (Final selection): If $\mathfrak{m}_k\left\langle b \right \rangle \neq \emptyset$, choose the shortest length segment say $\widetilde{\Delta}_{b} \in \mathfrak{m}_k\left\langle b \right \rangle$. Otherwise, we set $\widetilde{\Delta}_{b} = \emptyset$ the void segment. Recursively for $b-1 \geq i \geq a$, we choose the shortest segment $\widetilde{\Delta}_{i} \in \mathfrak{m}_k\left\langle i \right \rangle$ (if exists) such that $\widetilde{\Delta}_{i} \prec \widetilde{\Delta}_{i+1}$. Otherwise, we set $\widetilde{\Delta}_{i}=\emptyset$. 

Step 5. (Truncation): If $\widetilde{\Delta}_{i}\neq \emptyset$ for all $a \leq i \leq b$, we say that a downward sequence of minimal linked segments in neighbors on $\mathfrak{m}_k$ ranging from $b$ to $a$ exists and we define the right derivative multisegment by 
\[\mathcal{D}^\mathrm{Zel}_{[a,b]_\rho}(\mathfrak{m}) := \mathfrak{m} - \sum\limits_{i=a}^{b} \widetilde{\Delta}_{i} + \sum\limits_{i=a}^{b} (\widetilde{\Delta}_{i})^-.\]

Step 5'. If $\widetilde{\Delta}_{i}= \emptyset$ for some $a \leq i \leq b$, we set $\mathcal{D}^\mathrm{Zel}_{[a,b]_\rho}(\mathfrak{m})=\infty.$
\end{algorithm}

\begin{example}
(i) Let $\mathfrak{m}=\left\{ [0,4]_\rho, [2,4]_\rho, [2,5]_\rho, [2,5]_\rho, [3,5]_\rho, [4,5]_\rho \right\}$ and $\Delta=[5,5]_\rho$. Then, $[0,4]_\rho \prec [2,5]_\rho$ (resp. $[2,4]_\rho \prec [3,5]_\rho$) is the removable upward sequence of maximal linked segments on $\mathfrak{m}$ (resp. $\mathfrak{m}_1$) ranging from $4$ to $5$, where $\mathfrak{m}_1=\mathfrak{m}-[0,4]_\rho - [2,5]_\rho$, $\mathfrak{m}_2=\mathfrak{m}_1-[2,4]_\rho - [3,5]_\rho$ and there is no such sequence on $\mathfrak{m}_2$ as $\mathfrak{m}_2\left\langle 4 \right\rangle = \emptyset$. Since $[4,5]_\rho$ is the shortest segment in $\mathfrak{m}_2\left\langle 5 \right\rangle=\left\{[2,5]_\rho,[4,5]_\rho \right\}$, we have $\mathcal{D}^\mathrm{Zel}_{[5]_\rho}(\mathfrak{m})=\mathfrak{m}-[4,5]_\rho + [4]_\rho=\left\{ [0,4]_\rho, [2,4]_\rho, [2,5]_\rho, [2,5]_\rho, [3,5]_\rho, [4]_\rho \right\}.$

(ii) Let $\mathfrak{m}=\left\{ [0,4]_\rho,[3,4]_\rho, [2,5]_\rho, [3,5]_\rho, [4,6]_\rho \right\}$ and $\Delta=[4,6]_\rho$. There is no segment ending with $\nu^3 \rho$ to produce a removable upward sequence of maximal linked segments ranging from $3$ to $6$. Here, $\{[4,6]_\rho,[3,5]_\rho,[0,4]_\rho\}$ is the downward sequence of minimal linked segments in the neighbors of $\mathfrak{m}$ ranging from $6$ to $4$. Therefore, 
\begin{align*}
  \mathcal{D}^\mathrm{Zel}_{[4,6]_\rho}(\mathfrak{m})&=\mathfrak{m}- [4,6]_\rho-[3,5]_\rho-[0,4]_\rho + [4,5]_\rho+[3,4]_\rho+[0,3]_\rho \\
  &=\left\{ [0,3]_\rho,[3,4]_\rho, [2,5]_\rho, [3,4]_\rho, [4,5]_\rho \right\}.  
\end{align*}

(iii) Let $\mathfrak{m}=\left\{ [0,4]_\rho, [2,5]_\rho, [3,5]_\rho, [4,6]_\rho \right\}$ and $\Delta=[5,6]_\rho$. Here, $\left\{[0,4]_\rho, [2,5]_\rho,[4,6]_\rho\right\}$ is the only upward sequence of maximal linked segments in $\mathfrak{m}$ ranging from $4$ to $6$. But $\mathfrak{m}-\left\{[0,4]_\rho, [2,5]_\rho,[4,6]_\rho\right\}$ does not have any segment ending with $\nu^6 \rho$ to get a complete downward sequence of minimal linked segments ranging from $6$ to $5$. Therefore, $\mathcal{D}^\mathrm{Zel}_{[5,6]_\rho}(\mathfrak{m})=\infty$. \qed
\end{example}

%\begin{proposition}\label{prop:der_Zel=MW}
% Let $\mathfrak{m} \in \mathrm{Mult}_\rho$ and $\Delta(\mathfrak{m})$ be the first segment in $\mathfrak{m}^\#$. Then,
% \[\mathrm{D}^\mathrm{R}_{\Delta(\mathfrak{m})}(Z(\mathfrak{m}))\cong Z\left(\mathcal{D}^\mathrm{Zel}_{\Delta(\mathfrak{m})}(\mathfrak{m})\right).\]   
%\end{proposition}
%\begin{proof}
%    Follows from Proposition \ref{prop:MW_der} and the fact that $\mathcal{D}^\mathrm{Zel}_{\Delta(\mathfrak{m})}(\mathfrak{m})=\mathcal{D}^\mathrm{MW}(\mathfrak{m}).$
%\end{proof}

\subsection{Combinatorial structure from multiple MW algorithms}
For convenience, let $\Delta_1, \ldots, \Delta_r  \in \mathrm{Seg}_\rho$ be segments such that $e(\Delta_1)=\ldots =e(\Delta_r)$. We say that $\Delta_1, \ldots, \Delta_r$ are in  increasing order if $\Delta_1\subseteq \ldots \subseteq \Delta_r$, equivalently $s(\Delta_1)\geq \ldots \geq s(\Delta_r)$.

\begin{definition} \label{def minimal linked}
\begin{enumerate}
\item[(i)] Let $\mathfrak{n}_1, \mathfrak{n}_2 \in \mathrm{Mult}_{\rho}$. The multisegments $\mathfrak{n}_1 $ and $\mathfrak{n}_2$ are said to be linked by mapping if there exists an injective map $f:\mathfrak{n}_1 \rightarrow \mathfrak{n}_2$ such that  $\Delta \prec f(\Delta)$ for all $\Delta \in \mathfrak{n}_1$. 
\item[(ii)] Fix a multisegment $\mathfrak m \in \mathrm{Mult}_{\rho}$ and an integer $k$. Let $\mathfrak n_1$ be a submultisegment of $\mathfrak m\langle k-1 \rangle$ and let $\mathfrak n_2$ be a submultisegment of $\mathfrak m\langle k \rangle$. We say that $\mathfrak n_1$ and $\mathfrak n_2$ are minimally linked (in $\mathfrak m$) if 
\begin{enumerate}
    \item[(a)] $\mathfrak n_1$ and $\mathfrak n_2$ are linked by mapping; and
    \item[(b)] there does not exist $\mathfrak n_1^{\prime} { \subseteq \mathfrak m\langle k-1 \rangle}$ such that $|\mathfrak n_1^{\prime}|>|\mathfrak n_1|$ and $\mathfrak n_1^{\prime}$ and $\mathfrak n_2$ are linked by mapping;
    \item[(c)] there does not exist $\mathfrak n_1^{\prime} { \subseteq \mathfrak m\langle k-1 \rangle}$ such that $|\mathfrak n_1^{\prime}|=|\mathfrak n_1|$, $\mathfrak n_1' < \mathfrak n_1$, and $\mathfrak n_1^{\prime}$ and $\mathfrak n_2$ are linked by mapping. Here, we write both the multisegments $\mathfrak n_1=\left\{ \Delta_1, \ldots, \Delta_r\right\}$ and $\mathfrak n_1'=\left\{ \Delta_1', \ldots, \Delta_r'\right\}$ in the increasing order, and $\mathfrak n_1'< \mathfrak n_1$ means $s(\Delta_i)\leq s(\Delta_i')$ for all $i$ and at least one inequality is strict.
\end{enumerate}
\end{enumerate}

The minimality refers to the condition (c) in Definition \ref{def minimal linked} while in certain sense, we also require the number of segments to be the largest in condition (b). It is straightforward to observe that in Definition \ref{def minimal linked}(ii), for a fixed $\mathfrak n_2 \subseteq \mathfrak m\langle k \rangle$, there is at most one submultisegment $\mathfrak n_1 \subseteq \mathfrak m\langle k-1 \rangle$ minimally linked to $\mathfrak n_2$. 
\end{definition}

\begin{remark}
One can find the above defined minimally linkedness is similar to the notion of the best matching function as introduced in \cite[Section 5.3]{LM16}.  To observe that we consider $Y= \mathfrak n_2\subset \mathfrak m\langle k \rangle$, $X =  \mathfrak m\langle k-1 \rangle$, and for $\Delta \in X, \Delta' \in Y$, define the relation $\rightsquigarrow$ by $\Delta' \rightsquigarrow \Delta$ if and only if $\Delta \prec \Delta'$. On both $X$ and $Y$, we consider the standard ordering $\Delta_1 \leq \Delta_2$ whenever $s(\Delta_1) \leq s(\Delta_2)$ and $\Delta_1, \Delta_1 \in X$ (or in $Y$). Then, $\rightsquigarrow$ is traversable (see \cite{LM16} for definition) and the domain of the best $\rightsquigarrow$-matching function $f$ is $\mathfrak{n}_1$, which is minimally linked to $Y$.
\end{remark}

%Let $\mathfrak{m} \in \mathrm{Mult}_\rho$ and $\Delta \in \mathrm{Seg}_\rho$. We denote $\varepsilon^\mathrm{Zel}_\Delta(\mathfrak{m})$ for the highest integer $t$ such that  $\left(\mathcal{D}^\mathrm{Zel}_{\Delta}\right)^t(\mathfrak{m})\neq \infty$. Let $\Delta(\mathfrak{m})=[a,c]_\rho$ be the first segment in $\mathfrak{m}^\#$ via M\oe glin-Waldspurger algorithm. We denote $\varepsilon^\mathrm{MW}_{[a,c]_\rho}(\mathfrak{m})$ for number of segment $[a,c]_\rho$ in $\mathfrak{m}^\#$.

\begin{lemma} \label{lem structure of segments from MW}
Let $\mathfrak m \in \mathrm{Mult}_{\rho}$. Let $\nu^c\rho$ be the maximal cuspidal support of $\mathfrak m$. Let ${ r} \in \mathbb Z_{>0}$ such that $\nu^c\rho$ is still the maximal cuspidal support of $(\mathcal D^{\mathrm{MW}})^{r-1}(\mathfrak m)$. For $1\leq i\leq r$, let $\Delta_{c, i},\Delta_{c-1, i}, \ldots, \Delta_{a_i, i}$ be all the ordered segments participating in the MW-algorithm for $(\mathcal D^{\mathrm{MW}})^{i-1}(\mathfrak m)$ with $e\left(\Delta_{k, i}\right)=k$. Set $\mathfrak{m}_0=\mathfrak{m}$ and define, recursively,
\[ \mathfrak m_i=\mathfrak m_{i-1}-\Delta_{c, i}-\ldots-\Delta_{a_i ,i} .
\]
Then 
\begin{enumerate}
\item[(i)] The ordered segments participating in the MW algorithm for $\mathfrak m_i$ are precisely the segments $\Delta_{c,i+1}, \ldots, \Delta_{a_{i+1},i+1}$. In particular,
\[  \left\{ \Delta_{c,1}, \ldots, \Delta_{a_1,1}, \ldots,  \Delta_{c,r}, \ldots, \Delta_{a_r,r} \right\} 
\]
is a submultisegment of $\mathfrak m$.
\item[(ii)] Let $[a,c]_\rho$ be the first segment produced in the MW algorithm for $\mathfrak{m}$. For $a \leq k \leq c$, define $MW_k(\mathfrak m)=\left\{ \Delta_{k,1}, \ldots, \Delta_{k, x_k} \right\}$, where $x_k$ is the largest integer such that $\Delta_{k,x_k}$ is defined. Then
\begin{enumerate}
\item[(a)] $x_c \geq \ldots \geq x_a$; 
\item[(b)] $\Delta_{k,1}, \ldots, \Delta_{k,x_k}$ are in the increasing order;
\item[(c)] for $a\leq k \leq c-1$, $MW_k(\mathfrak m)$ and $MW_{k+1}(\mathfrak m)$ are minimally linked in $\mathfrak m$.
\end{enumerate}
\end{enumerate}
\end{lemma}

\begin{proof}
When $i=1$, it follows from the definition of the MW algorithm. We consider $i \geq 2$. Then, inductively, we have:
\[ (\mathcal{D}^{\mathrm{MW}})^{i-1}(\mathfrak m) = \mathfrak m-\sum_{p=1}^{i-1}\sum_{k=a_p}^{c} \Delta_{k, p}+\sum_{p=1}^{i-1}\sum_{k=a_p}^{c}\Delta_{ k,p}^- .\]

To prove (i), one has to show that for all $k$, $\Delta_{k,i}\neq \Delta_{k+1,s}^-$ for all $1\leq s \leq i-1$. One way to show is an induction on $k$ for both (i) and (ii) together. The argument is straightforward (while not completely immediate) from the minimal choices from the MW algorithm. 

\end{proof}

Lemma \ref{lem structure of segments from MW} with the uniqueness of the minimal linkedness gives a characterization of segments participating in the MW algorithms. Below, we shall use this characterization to show the compatibility with the segments produced in Algorithm \ref{alg:der_Zel} (see Proposition \ref{prop fomrula for multiple mw} in the next section).

\subsection{MW algorithm and removal upward sequences of maximal linked segments} \label{ss prove claim 1 prop der} 

Our goal is to compute $\mathcal D^{\mathrm{Zel}}_{[b,c]_{\rho}}(\mathfrak m)$. For this, we first compute a more involved term $(\mathcal D^{\mathrm{MW}})^r\circ \mathcal D^{\mathrm{Zel}}_{[b,c]_{\rho}}(\mathfrak m)$ and show the term is equal to $(\mathcal D^{\mathrm{MW}})^{r+1}(\mathfrak m)$ in Propostion \ref{prop fomrula for multiple mw} below (see Lemma \ref{lem structure of w minus max} for more notations). We already have a combinatorial description of $(\mathcal D^{\mathrm{MW}})^{r+1}(\mathfrak m)$ in Lemma \ref{lem structure of segments from MW}, and we are going to analyse the combinatorial structure arising for algorithms in $(\mathcal D^{\mathrm{MW}})^r\circ \mathcal D^{\mathrm{Zel}}_{[b,c]_{\rho}}(\mathfrak m)$.

\subsubsection{Compare $\varepsilon^{MW}$ and the number of removal upward sequences of maximal linked segments}

\begin{lemma} \label{lem mw number=no max link zel}
Let $\mathfrak m \in \mathrm{Mult}_{\rho}$ and let $\Delta(\mathfrak m)=[a,c]_{\rho}$ be the first segment produced in the MW algorithm. Then, for $a \leq b \leq c$, the number of removal upward sequences of maximal linked segments in neighbors in $\mathfrak m$ ranging from $b-1$ to $c$ is equal to 
\[  \varepsilon^{\mathrm{MW}}_{[a,c]_{\rho}}(\mathfrak m)+\ldots +\varepsilon^{\mathrm{MW}}_{[b-1,c]_{\rho}}(\mathfrak m) .\]
(If $b=a$, then the number is equal to zero.)
\end{lemma}

\begin{proof}
Let $r_0=\varepsilon^{\mathrm{MW}}_{[a,c]_{\rho}}(\mathfrak m)+\ldots +\varepsilon^{\mathrm{MW}}_{[b-1,c]_{\rho}}(\mathfrak m)-1$. To prove the above lemma, one constructs from a collection, say $\mathfrak p$, of segments in removal upward sequences of maximal linked segments (in neighbors in $\mathfrak m$ ranging from $b-1$ to $c$) to a collection, say $\mathfrak q$, of all segments participating in the MW-algorithm (for $\mathfrak m, \mathcal D^\mathrm{MW}(\mathfrak m), \ldots, (\mathcal D^\mathrm{MW})^{r_0-1}(\mathfrak m)$), and vice versa. For the segments participating in the multiple MW-algorithms, one uses the description in Lemma \ref{lem structure of segments from MW}. 

Such constructions are elementary, and we give an example to illustrate the idea. Let 
\[  \mathfrak m=\left\{ [-4,0], [-2,1], [-3,1], [-1,2], [2,3] , [0,3], [3,4], [1,4]  \right\}
\]
Then, in this example with $a=0$, $b-1=2$ and $c=4$, we have $\varepsilon^{\mathrm{MW}}_{[a,c]_{\rho}}(\mathfrak m)=1$, $r_0=0$,
\[\mathfrak p=\left\{  [-1,2], [0,3], [1,4]\right\}, \text{ and } \mathfrak q=\left\{ [-4,0], [-2,1], [-1,2], [2,3], [3,4]\right\}.\]
Now, one can construct from $\mathfrak p$ to $\mathfrak q$ by first replacing $[1,4]$ with $[3,4]$, then replacing $[0,3]$ by $[2,3]$,  then keeping $[-1,2]$, and finally adding the segments $[-4,0]$ and $[-2,1]$. To construct from $\mathfrak q$ to $\mathfrak p$, one reverses the process.
\end{proof}

\subsubsection{Overlap between segments from MW algorithms and from the removal upward sequences} \label{sss overlap mw and removal}

\begin{lemma} \label{lem structure of w minus max}
Let $\mathfrak m \in \mathrm{Mult}_{\rho}$. Let $[a,c]_{\rho}$ be the first segment produced in the MW algorithm. Suppose, furthermore, the first segment produced in the MW-algorithm $(\mathcal D^{\mathrm{MW}})^r(\mathfrak m)$ is $[b,c]_{\rho}$. Set
\[  r= \varepsilon^{\mathrm{MW}}_{[a,c]_{\rho}}(\mathfrak m)+\ldots +\varepsilon^{\mathrm{MW}}_{[b-1,c]_{\rho}}(\mathfrak m) .
\]
For $i=1,\ldots, r+1$, let $\Delta_{c,i}, \ldots, \Delta_{a_i,i}$  be all the ordered segments participating in the MW algorithm for $(\mathcal D^{\mathrm{MW}})^{i-1}(\mathfrak m)$.
For $b\leq k \leq c$, let $MW_k(\mathfrak m)=\left\{ \Delta_{k,1}, \ldots, \Delta_{k,r+1} \right\}$. Let $\mathfrak r$ be the set of all segments in the removal upward sequences of maximal linked segments in neighbors in $\mathfrak m$ ranging from $b-1$ to $c$.

Define $p_k^*$ to be the least integer such that $\Delta_{k,p_k^*} \notin \mathfrak{r}\langle k\rangle$. (The existence of such integer is guaranteed by $|\mathfrak{r}\langle k \rangle|<|MW_k(\mathfrak m)|$, see Lemma \ref{lem mw number=no max link zel}.) Then the following conditions hold:
\begin{enumerate}
    \item[(a)] $p^*_c\leq \ldots \leq p^*_{b}$;
    \item[(b)] For $b\leq k \leq c$, $\Delta_{k, 1}, \ldots, \Delta_{k,p_k^*-1}$ are in $\mathfrak r\langle k \rangle$, and moreover they are the first $p_k^*-1$ segments in the increasing order of $\mathfrak{r}\langle k\rangle$;
    \item[(c)] For $b\leq k \leq c-1$, $MW_k(\mathfrak m)-\Delta_{k,p_k^*}$ and $MW_{k+1}(\mathfrak m)-\Delta_{k+1,p_{k+1}^*}$ are minimally linked in $\mathfrak m-\Delta_{c,p_c^*}-\ldots -\Delta_{b, p_b^*}$; 
    \item[(d)] $\Delta_{c,p^*_c}$ is the shortest segment in $\mathfrak m\langle c \rangle -\mathfrak{r}\langle c\rangle$, and for $b\leq k \leq c-1$, $\Delta_{k,p^*_k}$ is the shortest segment in $\mathfrak m\langle k \rangle -\mathfrak{r}\langle k\rangle$ that is linked to $\Delta_{k+1, p^*_{k+1}}$.
\end{enumerate}
\end{lemma}

\begin{example}
Let $\mathfrak m=\{[-4,2]_{\rho}, [-2,2]_{\rho}, [-4,3]_{\rho}, [-3,3]_{\rho},  [-1,3]_{\rho},[-5,4]_{\rho}, [-2,4]_{\rho}, [-1,4]_{\rho}, [2,4]_{\rho},$ $ [-1,5]_{\rho}, [3,5]_{\rho}, [4,5]_{\rho} \}.$ Then, $\Delta(\mathfrak{m})=[a,c]_\rho=[2,5]_\rho$.
We consider $b=3$. In such case,  $r=\varepsilon^{\mathrm{MW}}_{[2,5]_{\rho}}(\mathfrak m)=2$. Then in the notation of Lemma \ref{lem structure of w minus max}, for $k=2,3,5$,
\[ MW_k(\mathfrak m)= \mathfrak m\langle k \rangle  ,
\]
and
\[  MW_4(\mathfrak m)=\mathfrak m\langle 4\rangle- [-5,4]_{\rho}  .
\]
On the other hand,
\[ \mathfrak r\langle 5 \rangle =\left\{ [-1,5]_{\rho}, [3,5]_{\rho} \right\}, \mathfrak r\langle 4\rangle =\left\{ [-2,4]_{\rho}, [2,4]_{\rho} \right\} ,
\]
\[  \mathfrak r\langle 3 \rangle=\left\{ [-3,3]_{\rho}, [-1,3]_{\rho}\right\}, \text{ and } \mathfrak r\langle 2 \rangle= \mathfrak m \langle 2 \rangle .
\]
Then $p_5^*=1$, $p_4=2$, $p_3^*=3$, and $\Delta_{5,1}=[4,5]_{\rho}$, $\Delta_{4, 2}=[-1,4]_{\rho}$ and $\Delta_{3,3}=[-4,3]_{\rho}$. 
\end{example}

Before proving Lemma \ref{lem structure of w minus max}, we shall prove the following useful simple counting lemma:
\begin{lemma} \label{lem back to minimal link}
We shall use all the notations in Lemma \ref{lem structure of w minus max}. Let $b\leq k \leq c-1$ and $\mathfrak n$ be a submultisegment of $\mathfrak r\langle k \rangle$. Then $\mathfrak n$ and $\left\{ \Delta_{k+1,1},\ldots, \Delta_{k+1,|\mathfrak n|}\right\}$ are linked by mapping.
\end{lemma}

\begin{proof}
As $\mathfrak n$ is in $\mathfrak r\langle k \rangle$, this guarantees that there exist submultisegments $\mathfrak n_x\subset \mathfrak m\langle x \rangle$ $(x=k+1,\ldots, c)$ such that for all $x=k+1, \ldots, c$ 
\begin{enumerate}
\item $|\mathfrak n_x|=|\mathfrak n|$;
\item $\mathfrak n_{x-1}$ and $\mathfrak n_{x}$ are linked by mapping. (Here, $\mathfrak n_k=\mathfrak n$.)
\end{enumerate}

We can replace $\mathfrak n_{c}$  by $\widetilde{\mathfrak n}_c:=\left\{ \Delta_{c,1}, \ldots, \Delta_{c,|\mathfrak n|} \right\}$ so that $\mathfrak n_{c-1}$ and $\widetilde{\mathfrak n}_c$ are linked by mapping. Now, by using $MW_x(\mathfrak m)$ and $MW_{x+1}(\mathfrak m)$ are minimally linked, we inductively replace $\mathfrak n_x$ (where $k+1\leq x\leq c-1$) by $\widetilde{\mathfrak n}_x:=\left\{ \Delta_{x,1}, \ldots, \Delta_{x,|\mathfrak n|} \right\}$ such that $\mathfrak n_{x-1}$ and $\widetilde{\mathfrak n}_x$ are still linked by mapping.
So eventually, we also have
 $\mathfrak n$ and $\widetilde{\mathfrak n}_{k+1}$ is also linked by mapping, as desired.
\end{proof}

\noindent
{\it Proof of Lemma \ref{lem structure of w minus max}:} 
For $b \leq k \leq c$, let $\mathfrak{r}\langle k\rangle=\left\{\Delta_{k,1}', \ldots, \Delta_{k, r}'\right\}$ written in the increasing order, where the number of segments follows from Lemma \ref{lem mw number=no max link zel}. When $k=c$, it is clear that $\Delta_{c,1}, \ldots, \Delta_{c,p_c^*-1}$ are in $\mathfrak r$. This gives (b) and (d), and there is nothing to prove for (c).

We now assume $b\leq k<c$ and separately consider each condition:

\noindent
{\it Prove condition (a) and first part of (b):} We have to show that $\Delta_{k, 1}, \ldots, \Delta_{k, p^*_{k+1}-1}$ are in $\mathfrak r\langle k \rangle$. Suppose not, that means $\Delta_{k, j}$ is not in $\mathfrak r\langle k \rangle$ for some $1 \leq j \leq p_{k+1}^*-1$. Then Lemma \ref{lem back to minimal link} and the minimally linked condition imply that we must have $s(\Delta_{k, p_{k+1}^*-1}')\leq s(\Delta_{k, p_{k+1}^*})$. On the other hand, since $\Delta_{k,p_{k+1}^*} \prec \Delta_{k+1,p_{k+1}^*}$, $s(\Delta_{k,p_{k+1}^*})<s(\Delta_{k+1, p_{k+1}^*})$. Combining the two inequalities, we have $s(\Delta_{k,p_{k+1}^*-1}')< s(\Delta_{k+1,p_{k+1}^*})$. This implies 
\begin{align} \label{eqn linked replace condition (a)}
\Delta'_{k,p_{k+1}^*-1}\prec \Delta_{k+1,p_{k+1}^*}.
\end{align}
On the other hand, since $\mathfrak{r}\langle k\rangle$ and $\mathfrak{r}\langle k+1\rangle$ are linked by mapping, we can define a map $f: \mathfrak r\langle k \rangle \rightarrow \mathfrak r\langle k+1 \rangle$ satisfying the following properties: for all $1\leq j \leq r+1$
\begin{enumerate}
    \item[(i)] $f(\Delta_{k,j}')=\Delta_{k+1,j}'$
    \item[(ii)] $\Delta'_{k,j}\prec f(\Delta'_{k, j})$.
\end{enumerate}
Then, using the induction case of condition (b), we must have that 
\[ \Delta_{k+1, p_{k+1}^*-1}'=\Delta_{k+1, p_{k+1}^*-1}
\]
Now, combining above discussions, we can define another map $\widetilde{f}:\mathfrak r\langle k \rangle \rightarrow \mathfrak r\langle k+1 \rangle-\Delta_{k+1,p_{k+1}^*-1}+\Delta_{k+1,p_{k+1}^*}$ such that $\widetilde{f}(\Delta_{k,j}')=f(\Delta_{k+1,j}')$ if $j\neq p_{k}^*-1$ and $\widetilde{f}(\Delta_{k,p_{k+1}^*-1}')=\Delta_{k+1,p_{k+1}^*}$. Hence, $f$ and (\ref{eqn linked replace condition (a)}) show that $\mathfrak r\langle k \rangle$ and $\mathfrak r\langle k+1 \rangle-\Delta_{k+1,p_{k+1}^*-1}+\Delta_{k+1,p_{k+1}^*}$ are linked by mapping. This contradicts the longest choices in $\mathfrak r\langle k \rangle$. This shows (a).\\

\noindent
{\it Prove latter part of condition (b):} By Lemma \ref{lem back to minimal link}, $\left\{ \Delta_{k,1}', \ldots, \Delta_{k,p_k^*-1}'\right\}$ and $\left\{ \Delta_{k+1,1}, \ldots, \Delta_{k+1,p_{k}^*-1}\right\}$ are linked by mapping. On the other hand, by Lemma \ref{lem structure of segments from MW}, $\left\{ \Delta_{k,1}, \ldots, \Delta_{k, p_{k-1}^*-1}\right\}$ is minimally linked to $\left\{ \Delta_{k+1,1},\ldots, \Delta_{k+1,p_{k}^*-1}\right\}$. Now the first part of (b) forces the second assertion holds. \\

\noindent{\it Prove condition (c):} The proof is slightly long, so we separate it into the next section. We only need (a) and (b) (but not (d)) to prove (c). \\

\noindent
{\it Prove condition (d):} $\Delta_{c,p^*_c}$ is the shortest segment in $\mathfrak m\langle c \rangle -\mathfrak r\langle c \rangle$ follows from assertion (b). By (a) and (c), we then have that $\left\{ \Delta_{k,1}, \ldots,\Delta_{k,p_k^*-1} \right\} \cap (MW_{k}(\mathfrak m)-\Delta_{k,p_k^*})$, and $\left\{ \Delta_{k+1,1}, \ldots,\Delta_{k+1,p_k^*} \right\} \cap (MW_{k+1}(\mathfrak m)-\Delta_{k+1,p_{k+1}^*}),$ are also minimally linked, where the former (resp. latter) set is simply the first $p_k^*-1$ shortest segments of $MW_{k}(\mathfrak m)-\Delta_{k,p_k^*}$ (resp. $MW_{k+1}(\mathfrak m)-\Delta_{k+1,p_k^*}$). Similarly, we have
\[ \left\{ \Delta_{k,1}, \ldots,\Delta_{k,p_k^*} \right\} \cap MW_k(\mathfrak m), \mbox{and} \left\{ \Delta_{k+1,1}, \ldots,\Delta_{k+1,p_k^*} \right\} \cap MW_{k+1}(\mathfrak m)
\]
is minimally linked. The uniqueness of minimal linkedness then implies (d). \qed

\subsubsection{Proof of Condition (c) in Lemma \ref{lem structure of w minus max}} \label{ss prove **}

Recall that we are assuming $b\leq k< c$. Let 
\[  \mathfrak m'=\mathfrak m-\Delta_{c,p_c^*}-\ldots -\Delta_{b,p_b^*} .
\]  

{\it Step 1: Show $MW_{k}(\mathfrak m)-\Delta_{k,p_k^*}$ and $MW_{k+1}(\mathfrak m)-\Delta_{k+1,p_{k+1}^*}$ are linked by mapping.} Define an injective map
 \[  f: MW_{k}(\mathfrak m)-\Delta_{k,p_k^*} \longrightarrow MW_{k+1}(\mathfrak m)-\Delta_{k+1, p_{k+1}^*}
 \]
 as follows:
  \begin{itemize}
  \item for $1\leq j \leq p_{k+1}^*-1$ and $p_{k}^*+1\leq j \leq r+1$, define $f(\Delta_{k,j})=\Delta_{k+1,j}$. It follows from the minimal linkedness between $MW_k(\mathfrak m)$ and $MW_{k+1}(\mathfrak m)$, we also have $\Delta_{k,j} \prec f(\Delta_{k,j})$.
  \item for $p_{k+1}^*\leq j \leq p_k^*-1$, define $f(\Delta_{k,j})=\Delta_{k+1,j+1}$. By condition (b) in Lemma \ref{lem structure of w minus max}, $\Delta_{k,1}, \ldots, \Delta_{k, p^*_k-1}$ are in $\mathfrak r$. As $\mathfrak r\subset \mathfrak m'$ and the induction assumption gives that $MW_x(\mathfrak m)-\Delta_{x,p_x^*}$ and $MW_{x+1}(\mathfrak m)-\Delta_{x+1,p_{x+1}^*}$ are minimally linked in $\mathfrak m'$ ($x=k+1,\ldots, c-1$), one applies a similar argument of the proof of Lemma \ref{lem back to minimal link} to show that 
  \[ \left\{ \Delta_{k,1}, \ldots, \Delta_{k,p^*_k-1}\right\} \quad \mbox{and} \quad \left\{\Delta_{k+1,1}, \ldots, \Delta_{k+1, p^*_k}\right\}-\Delta_{k+1, p^*_{k+1}} 
  \]
  are linked by mapping. This verifies that $\Delta_{k,j}\prec f(\Delta_{k,j})$.
  %This guarantees that $\left\{\Delta_{k,1}, \ldots, \Delta_{k,p^*_k-1}\right\}$ and some subset $\mathfrak n_{k+1}$ of $\mathfrak m\langle k+1\rangle-\Delta_{k+1,p_{k+1}^*}$ are linked by mapping, and also there exist subsets $\mathfrak n_{x} \subset \mathfrak m\langle x \rangle-\Delta_{x,p_x^*}$ ($x=k+2, \ldots, c$) such that $\mathfrak n_x$ and $\mathfrak n_{x+1}$ are linked by mapping. However, now inductively (using $MW'_{x}(\mathfrak m)$ and $MW'_{x+1}(\mathfrak m)$ are minimally linked) we can replace $\mathfrak n_{x}$, for $x$ starting from $c$ and then downward to $k+1$, by $\widetilde{\mathfrak n}_x:=\left\{ \Delta_{x,1}, \ldots, \Delta_{x,p^*_k-1} \right\}-\Delta_{x,p_x^*}$ so that 
  %\begin{itemize}
  %\item  for $x=k+1, \ldots, c-1$, $\widetilde{\mathfrak n}_x$ and $\widetilde{\mathfrak n}_{x+1}$ is linked by mapping; and
  %\item $\left\{ \Delta_{k,1}, \ldots, \Delta_{k,p_k^*-1} \right\}$ and $\widetilde{\mathfrak n}_{k+1}$ is also linked by mapping. 
%\end{itemize}
%  Thus, the linkedness by mapping verifies that $\Delta_{k,j}\prec f(\Delta_{k,j})$. 
\end{itemize}

Therefore, the map $f$ shows that $MW_k(\mathfrak m)-\Delta_{k,p_k^*}$ and $MW_{k+1}(\mathfrak m)-\Delta_{k+1,p_{k+1}^*}$ are linked by mapping. 

{\it Step 2: Check minimal linkedness.}
Suppose $MW_k(\mathfrak m)-\Delta_{k,p_k^*}$ and $MW_{k+1}(\mathfrak m)-\Delta_{k+1, p_{k+1}^*}$ are not minimally linked in $\mathfrak m'$. By the condition (a) in Lemma \ref{lem structure of w minus max} (which has been shown before), we have $\Delta_{k, p_k^*} \prec \Delta_{k+1, p_{k+1}^*}$ again, and hence,  $(MW_k(\mathfrak m)-\Delta_{k,p_k^*})+\Delta_{k, p_k^*}$ and $(MW_{k+1}(\mathfrak m)-\Delta_{k+1,p_{k+1}^*})+\Delta_{k+1, p_{k+1}^*}$ are linked by mapping but not minimally linked in $\mathfrak m$. This contradicts that $MW_k(\mathfrak m)$ and $MW_{k+1}(\mathfrak m)$ are minimally linked in $\mathfrak m$.

\subsubsection{Minimal linkedness between $MW_{b-1}(\mathfrak m)$ and $MW_b(\mathfrak m)-\Delta_{b,p_b^*}$}

\begin{lemma}
We use the notations in Lemma \ref{lem structure of w minus max}, and similarly, for $a \leq k \leq b-1$, we define 
\[MW_{k}(\mathfrak m)=\left\{ \Delta_{k,1},\ldots, \Delta_{k,r}\right\}.\]
Then $MW_{b-1}(\mathfrak m)$ and $MW_b(\mathfrak m)-\Delta_{b,p_b^*}$ are minimally linked in $\mathfrak m-\Delta_{c,p_c^*}-\ldots-\Delta_{b,p_b^*}$.
\end{lemma}

\begin{proof}
One can define an injective map from $MW_{b-1}(\mathfrak m)$ to $MW_b(\mathfrak m)-\Delta_{b,p_b^*}$ by similar arguments as in Section \ref{ss prove **}. We first argue that $\Delta_{b-1,1},\ldots, \Delta_{b-1,p_b^*-1}$ are in $\mathfrak{r}\langle{b-1}\rangle$ by using a similar argument in proving condition (a) of Lemma \ref{lem structure of w minus max} in Section \ref{sss overlap mw and removal}. By the fact that $\Delta_{b,p_b^*}$ is not in $\mathfrak{r}\langle{b}\rangle$, one can observe  
\[   \Delta_{b-1,p_b^*} \prec \Delta_{b, p_b^*+1},  \ldots ,\quad \Delta_{b-1, r} \prec \Delta_{b, r+1}  .
\]
 Now the minimal linkedness between $MW_{b-1}(\mathfrak m)$ and $MW_b(\mathfrak m)-\Delta_{b,p_b^*}$ follows from the minimal linkedness between $MW_{b-1}(\mathfrak m)$ and $MW_{b}(\mathfrak m)$ in $\mathfrak m$.
\end{proof}

\subsubsection{Segments participating in the MW algorithms for  $\mathcal D^{\mathrm{Zel}}_{[b,c]_{\rho}}(\mathfrak m)$}

\begin{lemma}\label{lem zel under const mw}
We use the notations in Lemma \ref{lem structure of w minus max}. Then, we have 
\begin{align} \label{eqn zel under minimal MW}
      \mathcal D^{\mathrm{Zel}}_{[b,c]_{\rho}}(\mathfrak m)=\mathfrak m-\sum_{k=b}^c\Delta_{k,p_k^*}+\sum_{k=b}^c\Delta_{k,p_k^*}^- .
\end{align}
In particular, $\mathcal D^{\mathrm{Zel}}_{[b,c]_{\rho}}(\mathfrak m)\neq \infty$.
\end{lemma}

\begin{proof}
This follows from Condition (d) of Lemma \ref{lem structure of w minus max}. 
\end{proof}

We are now going to find segments participating in the MW-algorithms for $(\mathcal D^{\mathrm{MW}})^{i-1}(\mathcal D^{\mathrm{Zel}}_{[b,c]_{\rho}}(\mathfrak m))$. In view of the formula (\ref{eqn zel under minimal MW}), the answer is almost given by Lemmas \ref{lem structure of segments from MW} and \ref{lem structure of w minus max}(c), but we still have to take care the possible contributions of those terms $\Delta_{k,p_k^*}^-$ in (\ref{eqn zel under minimal MW}). This will be done in the following lemma:

\begin{lemma} \label{lem mw after zel}
We use the notations in Lemma \ref{lem structure of w minus max}.  For $1\leq i \leq r$, the segments participating in the MW algorithm for $(\mathcal D^{\mathrm{MW}})^{i-1}(\mathcal D^{\mathrm{Zel}}_{[b,c]_{\rho}}(\mathfrak m))$ lie in either $MW_k(\mathfrak m)-\Delta_{k, p_k^*}$ for some $b \leq k\leq c$ or $MW_k(\mathfrak m)$ for some $a \leq k \leq b-1$.
\end{lemma}

\begin{proof}

Note that $\mathcal D^{\mathrm{Zel}}_{[b,c]_{\rho}}(\mathfrak m)$ is described in Lemma \ref{lem zel under const mw}. We pick the segments participating in the MW algorithm for $\mathcal D^{\mathrm{Zel}}_{[b,c]_{\rho}}(\mathfrak m)$ as follows. The first segment is the shortest segment in $MW_c(\mathfrak m)-\Delta_{c,p_c^*}$, that is $\Delta_{c,1}$ if $p_c^*\neq 1$ and $\Delta_{c,2}$ if $p_c^*=1$. 

Let $k^*$ be the largest integer such that $p_{k^*}^*\neq 1$. If such an integer does not exist, set $k^*=b-1$. In general, the segments participating in the MW algorithm are:
\[   \Delta_{c,2}, \ldots, \Delta_{k^*+1,2}, \Delta_{k^*,1}, \ldots,  \Delta_{a,1} ,
\]
in which each consecutive segments are linked by Lemma \ref{lem structure of w minus max}(c). By condition (a) of Lemma \ref{lem structure of w minus max}, the above choices are well-defined. We now justify that the above choices are the shortest ones:

\begin{enumerate}
\item Case 1. $k^*+1\leq k < c$: Note that $\Delta_{k+1,1}^-$ is not linked to $\Delta_{k+1,2}$. Hence, we can only find the shortest one in $\mathcal D^{\mathrm{Zel}}_{[b,c]_{\rho}}(\mathfrak m)\langle k\rangle-\Delta_{k+1,1}^-$. By Lemma \ref{lem structure of w minus max} Condition (c) and Lemma \ref{lem structure of segments from MW}(ii)(c), $\Delta_{k,2}$ is the shortest choice.
\item Case 2. $k =k^*$: Similar reasoning as above, $\Delta_{k^*+1,1}^-$ cannot be a choice, and $\Delta_{k^*,1}$ is the shortest choice. 
\item Case 3. $k \leq k^*-1$: If $\Delta_{k+1,p_{k+1}^*}^-$ is a choice, then $s(\Delta_{k+1, p_{k+1}^*})> s(\Delta_{k,1})$ and so $\Delta_{k,1} \prec \Delta_{k+1, p_{k+1}^*}$.

Now, one considers $\mathfrak r'=\mathfrak r-\Delta_{k+1,1}+\Delta_{k+1, p_{k+1}^*}$. By Lemma \ref{lem structure of w minus max}(b), $ \Delta_{k,1}$ are in $\mathfrak r$. Now, by using Condition (b) in Lemma \ref{lem structure of w minus max}, we have that $\Delta_{k+1,1}$ (resp. $\Delta_{k,1}$) is the first segment in the increasing order of $\mathfrak r\langle k+1\rangle$ (resp. $\mathfrak{r}\langle k\rangle$). Now one can define an injective map $f$ from $\mathfrak{r}\langle k\rangle$ to $\mathfrak{r}\langle k+1\rangle$ satisfying $f(\Delta_{k,1})=\Delta_{k+1,1}$ and $\Delta\prec f(\Delta)$ for all $\Delta \in \mathfrak r\langle k \rangle$.

Now one defines $\widetilde{f}:\mathfrak r'\langle k \rangle \rightarrow \mathfrak r'\langle k+1\rangle$ by $\widetilde{f}(\Delta_{k,1})=\Delta_{k+1,p_{k+1}^*}$ and $\widetilde{f}(\Delta)=f(\Delta)$ for $\Delta\neq \Delta_{k,1}$, which also determines that $\mathfrak r'\langle k\rangle$ and $\mathfrak r'\langle k+1\rangle$ are linked by mapping. This contradicts the maximal choice of the removal upward sequences of maximal linked segments in neighbors on $\mathfrak m$ ranging from $b-1$ to $c$. Hence, one cannot choose $\Delta_{k^*+1,1}^-$, and so $\Delta_{k^*,1}$ is the shortest choice again by Lemma \ref{lem structure of w minus max}(c).

\end{enumerate}

One can now proceed to find segments participating in the MW algorithm for $\mathcal (\mathcal D^{\mathrm{MW}})^i(\mathcal D_{[b,c]_{\rho}}^{\mathrm{Zel}}(\mathfrak m))$ in a similar manner for $i \geq 1$, and see they lie in $MW_k(\mathfrak m)-\Delta_{k,p_k^*}$ (for $b\leq k\leq c$) or $MW_k(\mathfrak m)$ (for $a\leq k \leq b-1)$. We omit the details.
\end{proof}

\begin{proposition} \label{prop fomrula for multiple mw}
We use the notations in Lemma \ref{lem structure of w minus max}. Recall that $r=\varepsilon_{[a,c]_{\rho}}^{\mathrm{MW}}(\mathfrak m)+\ldots +\varepsilon_{[b-1,c]_{\rho}}^{\mathrm{MW}}(\mathfrak m)$. Suppose $[b,c]_{\rho}$ is the first segment produced in the MW algorithm for $(\mathcal D^{\mathrm{MW}})^r(\mathfrak m)$. Then
\[  \left(\mathcal D^{\mathrm{MW}}\right)^{r+1}(\mathfrak m) = \left(\mathcal D^{\mathrm{MW}}\right)^r \left(\mathcal D^{\mathrm{Zel}}_{[b,c]_{\rho}}(\mathfrak m)\right) \neq \infty.
\]
\end{proposition}

\begin{proof}
It follows from Lemmas \ref{lem structure of segments from MW}, \ref{lem zel under const mw}, and \ref{lem mw after zel}.
\end{proof}

\begin{lemma} \label{lem zel non infty to mw}
We use the notations in Lemma \ref{lem structure of w minus max}. If  $\mathcal D^{\mathrm{Zel}}_{[b,c]_{\rho}}(\mathfrak m)\neq \infty$, then $[b,c]_{\rho}$ is the first segment produced in the MW algorithm for $(\mathcal D^{\mathrm{MW}})^r(\mathfrak m)$.
\end{lemma}

\begin{proof}
  By Lemma \ref{lem mw number=no max link zel}, one obtains $r$ removal upward sequences of maximal linked segments in neighbors on $\mathfrak m$ ranging from $b-1$ to $c$, and then one downward sequence of minimal linked segments in neighbors from $c$ to $b$. From here, one can do a similar construction in the proof of Lemma \ref{lem mw number=no max link zel} to obtain segments participating in the MW algorithm. Since the construction and details are again elementary and are similar to the proof of Lemma \ref{lem mw number=no max link zel}, we omit further details.
\end{proof}

\begin{corollary}\label{prop:derivative}
    Let $\Delta_0 \in \mathrm{Seg}_\rho$ and $\mathfrak{m} \in \mathrm{Mult}_\rho$ such that $e(\Delta_0)$ is the maximal cuspidal support in $\mathfrak{m}$. If $\mathcal{D}^\mathrm{Zel}_{\Delta_0}(\mathfrak{m}) \neq \infty$, then
$\mathrm{D}^\mathrm{R}_{\Delta_0}(Z(\mathfrak{m})) \cong Z\left(\mathcal{D}^\mathrm{Zel}_{\Delta_0}(\mathfrak{m})\right)$.   
\end{corollary}

\begin{proof}
Let $e(\Delta_0)=\nu^c\rho$ for some $c$. Let $\Delta_0=[b,c]_{\rho}$. Let $\Delta(\mathfrak{m})=[a,c]_{\rho}$ be the first segment produced in the MW algorithm for $\mathfrak{m}$. Let
\[  r=\varepsilon^{\mathrm{MW}}_{[a,c]_{\rho}}(\mathfrak m)+\ldots +\varepsilon^{\mathrm{MW}}_{[b-1,c]_{\rho}}(\mathfrak m) .
\]
As $\mathcal D^{\mathrm{Zel}}_{\Delta_0}(\mathfrak m)\neq \infty$, one observes the proof of Lemma \ref{lem mw number=no max link zel} also gives $a \leq b$. By Lemma \ref{lem mw after zel},
\[  (\mathcal D^{\mathrm{MW}})^{r+1}(\mathfrak m) =(\mathcal D^{\mathrm{MW}})^r(\mathcal D^{\mathrm{Zel}}_{[b,c]_{\rho}}(\mathfrak m))) .
\]
By Lemma \ref{lem zel non infty to mw}, the first segment produced in the MW algorithm for $(\mathcal D^{\mathrm{MW}})^r(\mathfrak m)$ is $[b,c]_{\rho}$. Then, by applying Proposition \ref{prop:MW_der} multiple times on above equation, we have:
\begin{align*} 
 & \mathrm D^{\mathrm R}_{[b,c]_{\rho}}\circ (\mathrm D^{\mathrm R}_{[b-1,c]_{\rho}})^{\varepsilon^{\mathrm{MW}}_{[b-1,c]}(\mathfrak m)}\circ \ldots \circ (\mathrm D^{\mathrm R}_{[a,c]_{\rho}})^{\varepsilon^{\mathrm{MW}}_{[a,c]}(\mathfrak m)}(Z(\mathfrak m)) \\
\cong & (\mathrm{D}^{\mathrm R}_{[b-1,c]_{\rho}})^{\varepsilon^{\mathrm{MW}}_{[b-1,c]_{\rho}}(\mathfrak m)}\circ \ldots \circ (\mathrm{D}^\mathrm{R}_{[a,c]_{\rho}})^{\varepsilon^{\mathrm{MW}}_{[a,c]}(\mathfrak m)} \circ  \mathrm D^{\mathrm R}_{[b,c]_{\rho}} (Z(\mathcal D^{\mathrm{Zel}}_{[b,c]_{\rho}}(\mathfrak m))) .
\end{align*}
By the commutativity of derivatives for unlinked segments (see e.g. \cite[Lemma 4.4]{{Cha_csq}}), we have 
\begin{align*} 
 & (\mathrm D^{\mathrm R}_{[b-1,c]_{\rho}})^{\varepsilon^{\mathrm{MW}}_{[b-1,c]}(\mathfrak m)}\circ \ldots \circ (\mathrm D^{\mathrm R}_{[a,c]_{\rho}})^{\varepsilon^{\mathrm{MW}}_{[a,c]}(\mathfrak m)} \circ  \mathrm D^{\mathrm R}_{[b,c]_{\rho}}(Z(\mathfrak m)) \\
\cong & (\mathrm{D}^{\mathrm R}_{[b-1,c]_{\rho}})^{\varepsilon^{\mathrm{MW}}_{[b-1,c]_{\rho}}(\mathfrak m)}\circ \ldots \circ (\mathrm{D}^{\mathrm{R}}_{[a,c]_{\rho}})^{\varepsilon^{\mathrm{MW}}_{[a,c]}(\mathfrak m)} (Z(\mathcal D^{\mathrm{Zel}}_{[b,c]_{\rho}}(\mathfrak m))) .
\end{align*}
Now, applying suitable integrals multiple $r$-times, we can cancel the first $r$ derivatives on both sides, we then have
$\mathrm{D}^\mathrm{R}_{[b,c]_{\rho}}(Z(\mathfrak m)) \cong Z(\mathcal D^{\mathrm{Zel}}_{[b,c]_{\rho}}(\mathfrak m))$.
\end{proof}

\subsection{Reduction to maximal cuspidal support case} \label{ss reduction to max cus}

For $\mathfrak{m} \in \mathrm{Mult}_\rho$ and $x \in \mathbb{Z} $, we denote $\mathfrak{m}^{\leq x}=\{[a^\prime,b^\prime]_\rho \in \mathfrak{m} \mid b^\prime \leq x\}$ and $\mathfrak m^{>x}=\left\{[a^{\prime},b^{\prime}]_{\rho} \in \mathfrak m\mid b^{\prime}>x \right\}$.

%Then the $\rho$-derivative in the Zelevinsky classification gives the following Lemma \ref{lem:Zel_der=0_A} and \ref{lem:Zel_der=0_B}.

%%%%%%%%%%%%%%%%
%\begin{lemma}\label{lem:Zel_der=0_A}
%	Fix $\mathfrak{m} \in \mathrm{Mult}_\rho$. For any $c^\prime \geq c$, we have \[\mathrm{D}_{[c]_\rho}\left(Z\left(\mathfrak{m}^{\leq c^\prime}\right)\right)=0 \Longleftrightarrow \mathrm{D}_{[c]_\rho}(Z(\mathfrak{m}))=0.\]   
%\end{lemma}

%\begin{lemma}\label{lem:Zel_der=0_B}
%	Let $\mathfrak{m}, \mathfrak{n} \in \mathrm{Mult}_\rho$ such that $Z(\mathfrak{n})=\mathrm{D}_{[c]_\rho}(Z(\mathfrak{m})) \neq 0$. Then, for any $c^\prime \geq c$, \[\mathrm{D}_{[c]_\rho}\left(Z\left(\mathfrak{m}^{\leq c^\prime}\right)\right)= Z\left({\mathfrak{n}}^{\leq c^\prime}\right).\]   
%\end{lemma}
%%%%%%%%%%
\begin{proposition}\label{prop:Zel_der=0}
	Let $[a,c]_\rho \in \mathrm{Seg}_\rho$ and $\mathfrak{m} \in \mathrm{Mult}_\rho$. Then, 
    \begin{enumerate}
        \item[(i)] For any $c^\prime \geq c$, $\mathrm{D}^\mathrm{R}_{[a,c]_\rho}\left(Z\left(\mathfrak{m}^{\leq c^\prime}\right)\right) = 0$ if and only if $\mathrm{D}^\mathrm{R}_{[a,c]_\rho}(Z(\mathfrak{m})) = 0$.    
    \item[(ii)] Suppose $\mathrm D^\mathrm{R}_{[a,c]_{\rho}}(Z(\mathfrak m))\neq 0$. Let $\mathfrak p\in \mathrm{Mult}_{\rho}$ such that $Z(\mathfrak p)\cong \mathrm D^\mathrm{R}_{[a,c]_{\rho}}(Z(\mathfrak m^{\leq c}))$. Then, 
    \[   \mathrm{D}^\mathrm{R}_{[a,c]_{\rho}}(Z(\mathfrak m)) \cong Z(\mathfrak m^{>c}+\mathfrak p) .
    \]
    \end{enumerate}
\end{proposition}

\begin{proof}
Let $\Delta=[a,c]_{\rho}$. Suppose $\mathrm D_{[a,c]_{\rho}}^{\mathrm R}(Z(\mathfrak m^{\leq c'}))\neq 0$. Then $Z(\mathfrak m^{\leq c'})\hookrightarrow \mathrm{D}^{\mathrm R}_{[a,c]_{\rho}}(Z(\mathfrak m^{\leq c'}))\times \mathrm{St}([a,c']_{\rho})$, and so
\[  Z(\mathfrak m)\hookrightarrow Z(\mathfrak m^{>c'})\times Z(\mathfrak m^{\leq c'}) \hookrightarrow Z(\mathfrak m^{>c'})\times \mathrm{D}^{\mathrm R}_{[a,c]_{\rho}}(Z(\mathfrak m^{\leq c'}))\times \mathrm{St}([a,c]_{\rho}).
\]
Thus, $Z(\mathfrak m) \hookrightarrow \tau' \times \mathrm{St}([a,c']_{\rho})$ for some irreducible composition factor $\tau'$ in $Z(\mathfrak m^{>c'})\times \mathrm{D}^{\mathrm R}_{[a,c]_{\rho}}(Z(\mathfrak m^{\leq c'}))$. By Frobenius reciprocity, we have $\mathrm D^{\mathrm R}_{[a,c]_{\rho}}(Z(\mathfrak m))\neq 0$.

Suppose $\mathrm{D}^\mathrm{R}_{[a,c]_{\rho}}(Z(\mathfrak m))\neq 0$ and we shall show that $\mathrm{D}^\mathrm{R}_{[a,c]_{\rho}}(Z(\mathfrak m^{\leq c'}))\neq 0$. Let $\mathfrak n \in \mathrm{Mult}_{\rho}$ such that $Z(\mathfrak n)=\mathrm D^\mathrm{R}_{[a,c]_{\rho}}(Z(\mathfrak m))$. Then we have embeddings:
\[  Z(\mathfrak m) \hookrightarrow Z(\mathfrak n) \times \mathrm{St}(\Delta)  \hookrightarrow Z(\mathfrak n^{> c'})\times Z(\mathfrak n^{\leq c'}) \times \mathrm{St}(\Delta) ,
\]
where the first one follows from Frobenius reciprocity and the second one follows from e.g. \cite[Proposition 3.6]{LM16}. Now, by standard arguments, see e.g. \cite[Lemma 4.13]{LM16}, one has that $\mathfrak m^{> c'}=\mathfrak n^{> c'}$ and
\[  Z(\mathfrak m^{\leq c'}) \hookrightarrow Z(\mathfrak n^{\leq c'})\times \mathrm{St}(\Delta) .
\]
Now, applying Frobenius reciprocity, one has 
\[ \mathrm D^{\mathrm R}_{[a,c]_{\rho}}(Z(\mathfrak m^{\leq c'}))\neq 0 \quad \mbox{and} \quad Z(\mathfrak n^{\leq c'})\cong \mathrm D^{\mathrm R}_{[a,c]_{\rho}}(Z(\mathfrak m^{\leq c'})).\]
This gives (ii) and the only if direction of (i).
\end{proof}

\subsection{Main result}

\begin{theorem} \label{thm alg der zel}
Let $\mathfrak m \in \mathrm{Mult}_{\rho}$ and let $\Delta \in \mathrm{Seg}_{\rho}$. Then
\begin{enumerate}
\item[(i)] $\mathrm{D}^{\mathrm R}_{\Delta}(Z(\mathfrak m))\neq 0$ if and only if $\mathcal D^{\mathrm{Zel}}_{\Delta}(\mathfrak m)\neq \infty$.
\item[(ii)] If $\mathcal D^{\mathrm{Zel}}_{\Delta}(\mathfrak m) \neq \infty$, we have $\mathrm{D}^{\mathrm R}_{\Delta}(Z(\mathfrak m)) \cong Z\left(\mathcal D^{\mathrm{Zel}}_{\Delta}(\mathfrak m)\right)$.
\end{enumerate}
\end{theorem}

\begin{proof}
Suppose $\mathrm D^{\mathrm{R}}_{\Delta}(Z(\mathfrak m))\neq 0$. Write $\Delta=[b,c]_{\rho}$. By Proposition \ref{prop:Zel_der=0}, we have 
\begin{align} \label{eqn non zero derivative red}
\mathrm D^{\mathrm R}_{[b,c]_{\rho}}(Z(\mathfrak m^{\leq c})) \neq 0 .
\end{align}
Let $[a,c]_\rho$ be the first segment produced in the MW algorithm for $\mathcal D^{\mathrm{MW}}(\mathfrak m^{\leq c})$.

Let $r=\varepsilon^{\mathrm{MW}}_{[a,c]_{\rho}}(\mathfrak m^{\leq c})+\ldots+\varepsilon^{\mathrm{MW}}_{[b-1,c]_{\rho}}(\mathfrak m^{\leq c})$. Then, by Proposition \ref{prop:MW_der}, 
\[  (\mathrm{D}^{\mathrm{R}}_{[b-1,c]_{\rho}})^{\varepsilon^{\mathrm{MW}}_{[b-1,c]_{\rho}}(\mathfrak m^{\leq c})}\circ \ldots \circ  (\mathrm D^{\mathrm R}_{[a,c]_{\rho}})^{\varepsilon^{\mathrm{MW}}_{[a,c]_{\rho}}(\mathfrak m^{\leq c})}(Z(\mathfrak m^{\leq c})) \neq 0
\]

By the third bullet of \cite[Proposition 9.3(2)]{Cha_csq} and (\ref{eqn non zero derivative red}), one has:
\begin{align} \label{eqn epsilon mw 2}
 \varepsilon^{\mathrm R}_{[b,c]_{\rho}}((\mathrm D^{\mathrm R}_{[b-1,c]_{\rho}})^{\varepsilon^{\mathrm{MW}}_{[b-1,c]_\rho}(\mathfrak m^{\leq c})}\circ \ldots \circ (\mathrm D^{\mathrm R}_{[a,c]_{\rho}})^{\varepsilon^{\mathrm{MW}}_{[a,c]_\rho}(\mathfrak m^{\leq c})} (Z(\mathfrak m^{\leq c})) =\varepsilon^{\mathrm R}_{[b,c]_{\rho}}(Z(\mathfrak m^{\leq c})) \neq 0
\end{align}
and, for $ a'\leq b-1$,
\begin{align} \label{eqn epsilon mw 3}
 \varepsilon^{\mathrm R}_{[a',c]_{\rho}}(\mathrm D^{\mathrm R}_{[b-1,c]_{\rho}})^{\varepsilon^{\mathrm{MW}}_{[a',c]_\rho}(\mathfrak m^{\leq c})}\circ \ldots \circ (\mathrm D^{\mathrm R}_{[a,c]_{\rho}})^{\varepsilon^{\mathrm{MW}}_{[a,c]_\rho}(\mathfrak m^{\leq c})} (Z(\mathfrak m^{\leq c})) = 0.
\end{align}

Hence, by Proposition \ref{prop:MW_der} and (\ref{eqn epsilon mw 2}) and (\ref{eqn epsilon mw 3}), the first segment produced the MW algorithm for $(\mathcal D^{\mathrm{MW}})^{r}(\mathfrak m^{\leq c})$ is $[b,c]_{\rho}$. Now Proposition \ref{prop fomrula for multiple mw} implies $\mathcal D^{\mathrm{Zel}}_{\Delta}(\mathfrak m^{\leq c})\neq \infty$. It is clear from Algorithm \ref{alg:der_Zel} that we then have $\mathcal D^{\mathrm{Zel}}_{\Delta}(\mathfrak m)\neq \infty$. This proves the only if direction of (i).

Suppose $\mathcal D^{\mathrm{Zel}}_{\Delta}(\mathfrak m)\neq \infty$. Write $\Delta=[b,c]_{\rho}$. Then, $\nu^c\rho\in \mathrm{supp}(\mathfrak m)$ and so $\mathcal D^{\mathrm{Zel}}_{\Delta}(\mathfrak m^{\leq c})\neq \infty$. By Proposition \ref{prop fomrula for multiple mw},
\[  (\mathcal D^{\mathrm{MW}})^{r+1}(\mathfrak m^{\leq c})\neq \infty,
\]
where $r$ is defined as above. By Lemma \ref{lem zel non infty to mw}, the first segment produced for $(\mathcal D^{\mathrm{MW}})^r(\mathfrak m)$ is $[b,c]_{\rho}$. Thus, now by Lemma \ref{lem structure of segments from MW}(i), we have
\[ \mathrm D_{\Delta}^{\mathrm R}\circ \left(\mathrm D^{\mathrm R}_{[b-1,c]_{\rho}}\right)^{\varepsilon^\mathrm{MW}_{[b-1,c]_{\rho}}(\mathfrak m^{\leq c})}\circ \ldots \circ  \left(\mathrm D^{\mathrm R}_{[a,c]_{\rho}}\right)^{\varepsilon^\mathrm{MW}_{[a,c]_{\rho}}(\mathfrak m^{\leq c})}(Z(\mathfrak m^{\leq c}))\neq 0 .
\]
By the commutativity of derivatives for unlinked segments, we then have:
\[  \mathrm D^{\mathrm R}_{\Delta}(Z(\mathfrak m^{\leq c}))\neq 0 .
\]
Now, Proposition \ref{prop:Zel_der=0}(i) implies the if direction of (i) of this theorem. The assertion (ii) now follows from Proposition \ref{prop:Zel_der=0} and Corollary \ref{prop:derivative}.
\end{proof}

\subsection{Left derivative algorithm}

For $\mathfrak m \in \mathrm{Mult}_{\rho}$ and $[a,b]_{\rho} \in \mathrm{Seg}_\rho$, define 
\[ \mathcal D^{\mathrm{Zel}, \mathrm L}_{[a,b]_{\rho}}(\mathfrak m)=\Theta\left(\mathcal D^{\mathrm{Zel}}_{[-b,-a]_{\rho^{\vee}}}(\Theta(\mathfrak m))\right) .
\]
Now, with Theorem \ref{thm alg der zel} and discussions in Section \ref{sss gk invol}, we have:

\begin{theorem}
Let $\mathfrak{m} \in \mathrm{Mult}_\rho$ and $\Delta \in \mathrm{Seg}_\rho$. Then, the following hold:
\begin{enumerate}
\item $\mathcal{D}_\Delta^{\mathrm{Zel}, \mathrm{L}}(\mathfrak{m}) \neq \infty$ if and only if $\mathrm D^{\mathrm L}_{\Delta}(Z(\mathfrak m))\neq 0$; and
\item if $\mathcal{D}_\Delta^{\mathrm{Zel}, \mathrm{L}}(\mathfrak{m}) \neq \infty$, we have $\mathrm{D}^\mathrm{L}_\Delta \left(L(\mathfrak{m}) \right) \cong Z \left( \mathcal{D}_\Delta^{\mathrm{Zel}, \mathrm{L}}(\mathfrak{m})  \right)$.
\end{enumerate}
\end{theorem}

%%%%%%%%%%%%%%%%%%%%%%% Algorithm for Integrals %%%%%%%%%%%%%%%%%%%%%%%%%%%%%%%%%%%%%%%%%%%%%%%%%%%%%%
%
\section{Integrals in Langlands classification}\label{sec:int}
%Let $\pi$ be an irreducible smooth representation of $\mathrm{GL}_n(F)$ and $\sigma$ be an essentially square-integrable representation of $\mathrm{GL}_n(F)$. Each $\pi$ and $\sigma$ has combinatorial description in the form of $\pi=L(\mathfrak{m})$ for some multisegment $\mathfrak{m}\in \mathrm{Mult}$ and for some segment $\Delta \in \mathrm{Seg}$.

In this section, for $\mathfrak m \in \mathrm{Mult}_{\rho}$ and $\Delta \in \mathrm{Seg}_{\rho}$, we give an algorithm to compute the integral $\mathrm{I}^\mathrm{R}_{\Delta}(L(\mathfrak m))$. The basic strategy is to reduce to $\rho$-integrals, but we shall compare with Algorithm \ref{alg:der:Lang} and transfer some properties in Lemma \ref{lem:Lang_comm_I[ab]_I[a]}.

\subsection{Algorithm for $\rho$-integral} We are now going to present an algorithm for calculating the $\rho$-integrals of irreducible representations in the Langlands classification. 

\subsubsection*{$\mathfrak{tus}$-process:} This process involves the removal of two linked segments from a multisegment $\mathfrak{n} \in \mathrm{Mult}_\rho$ for a fixed integer $c$. The steps are as follows:
 \begin{itemize}
     \item[(i)] First, pick the shortest segment $\Delta^\prime \in \mathfrak{n}[c]$.
     \item[(ii)] Choose the shortest segment $\Delta^{\prime \prime} \in \mathfrak{n}[c+1]$ such that $\Delta^{\prime} \prec \Delta^{\prime  \prime}$.
     \item[(iii)] If both $\Delta^{\prime}$ and $\Delta^{\prime \prime}$  exist, remove them to define a new multisegment as \[\mathfrak{tus}(\mathfrak{n}, c)= \mathfrak{n}-\Delta^{\prime} - \Delta^{\prime \prime}.\]
 \end{itemize}

\begin{algorithm}\label{alg:rho_int:Lang}
Let $\mathfrak{n} \in \mathrm{Mult}_{\rho}$ and $c \in \mathbb{Z}$. We now define a new multisegment $\mathcal{I}^\mathrm{Lan}_{[c]_{\rho}} (\mathfrak{n})$ by the following algorithm:

Step 1. Set $\mathfrak{n}_0=\mathfrak{n}$, and recursively for an integer $i>0$, define $\mathfrak{n}_{i}= \mathfrak{tus}(\mathfrak{n}_{i-1}, c)$ until the process terminates. Suppose the $\mathfrak{tus}(-, c)$ process on $\mathfrak{n}$ terminates after $\ell$ times and the final remaining multisegment is $\mathfrak{n}_\ell$.

Step 2. Choose the longest segment (if it exists) ${\Delta}_* \in \mathfrak{n}_\ell[c+1]$ and define the multisegment \[ \mathcal{I}^\mathrm{Lan}_{[c]_{\rho}} (\mathfrak{n}) :=\mathfrak{n}- {\Delta}_* + {^+}{\Delta}_*.\]
If such segment ${\Delta}_*$ does not exist, we write $\mathcal{I}^\mathrm{Lan}_{[c]_{\rho}} (\mathfrak{n}) :=\mathfrak{n} + [c]_\rho.$
\end{algorithm}

\begin{example}
(i) Let $\mathfrak{m}=\left\{ [0,4]_\rho, [0,2]_\rho, [1,5]_\rho, [1,4]_\rho, [1,3]_\rho, [1]_\rho \right\}$ and $c=0$. Then, $\mathfrak{m}_1=\mathfrak{tus}(\mathfrak{m}, c)=\mathfrak{m}-[0,2]_\rho-[1,3]_\rho$ and $\mathfrak{m}_2=\mathfrak{tus}(\mathfrak{m}_1, c)=\mathfrak{m}_1-[0,4]_\rho-[1,5]_\rho$. The $\mathfrak{tus}(-,c)$-process terminates on $\mathfrak{m}_2=\left\{[1,4]_\rho, [1]_\rho \right\}$ since $\mathfrak{m}_2[0]=\emptyset$. As $[1,4]_\rho$ is the longest in $\mathfrak{m}_2[1]$, we have
\begin{align*}
	\mathcal{I}^\mathrm{Lan}_{[0]_{\rho}} (\mathfrak{m})=\mathfrak{m} - [1,4]_\rho + {^+}[1,4]_\rho 
	= \left\{ [0,4]_\rho,[0,2]_\rho, [1,5]_\rho, [0,4]_\rho, [1,3]_\rho, [1]_\rho\right\}.
\end{align*}

(ii) Let $\mathfrak{m}=\left\{ [0,2]_\rho, [1,3]_\rho, [1]_\rho, [2,3]_\rho \right\}$ and $c=1$. Then, $\mathfrak{m}_1=\mathfrak{tus}(\mathfrak{m}, 1)=\mathfrak{m}-[1]_\rho-[2,3]_\rho$ and the $\mathfrak{tus}(-,1)$ process terminates on $\mathfrak{m}_1$ as $\mathfrak{m}_1[2]=\emptyset$. Therefore,
\[\mathcal{I}^\mathrm{Lan}_{[1]_\rho} (\mathfrak{m})=\mathfrak{m}+[1]_\rho =\left\{ [0,2]_\rho, [1,3]_\rho, [1]_\rho, [1]_\rho, [2,3]_\rho \right\}.\]
\end{example} 
Like $\rho$-derivative, the $\rho$-integral seems to be better understood in the literature, for example, see \cite[Proposition 5.1 and 5.11]{LM16} and references therein. 

%For interested readers, we provide details of the proof of the following result by showing $\mathcal{D}_{[a]_\rho}^\mathrm{Lan}\left( \mathcal{I}_{[a]_\rho}^\mathrm{Lan}(\mathfrak{n})\right)=\mathfrak{n}$ (cf. Lemma \ref{int:lem:Lang:length=1}) in the Appendix.
\begin{proposition}[Jantzen, M\'inguez, and Lapid-M\'inguez] (see \cite[Theorem 5.11]{LM16})\label{int:Lang:length=1}
 For $\mathfrak{n} \in \mathrm{Mult}_\rho$ and $a \in \mathbb{Z}$, we have
\[\mathrm{I}^\mathrm{R}_{[a]_\rho}(L(\mathfrak{n})) \cong  L \left( \mathcal{I}_{[a]_\rho}^\mathrm{Lan}(\mathfrak{n})\right).\] 
\end{proposition}

\subsection{Algorithm for $\mathrm{St}$-integral}Let $\mathfrak{m}$ be a multisegment and $\Delta$ be a segment. We want to define a new multisegment $\mathcal{I}^\mathrm{Lang}_{\Delta}(\mathfrak{m})$ by the following algorithm so that the right integral of $L(\mathfrak{m})$ under $\Delta$ is given by $L \left(\mathcal{I}^\mathrm{Lang}_{\Delta}(\mathfrak{m})\right)$.

\subsubsection*{Downward sequence $\underline{\mathfrak{Ds}}$:} Let $\mathfrak{n}$ be a non-void multisegment in $\mathrm{Mult}_\rho$. We define the downward sequence of minimal linked segments with the largest starting as follows: find the largest number $a_{1}$ such that $\mathfrak{n}[a_{1}] \neq \emptyset$. Pick a shortest segment $\Delta_{1}=[a_{1},b_{1}]_\rho$ in $\mathfrak{n}[a_{1}]$. For $q \geq 2$, one recursively find largest number $a_{q}$ (if it exists) such that $a_{q} < a_{q-1}$ and there exists a segment in $\mathfrak{m}[a_{q}]$ which precedes $[a_{q-1},b_{q-1}]_\rho$. Then, we pick a shortest segment $\Delta_{q}=[a_{q},b_{q}]_\rho$ in $\mathfrak{n}[a_{q}]$. This process terminates after some finite steps, say $r$, and let $\Delta_{1}, \Delta_{2}, \cdots,\Delta_{r}$ be all the segments found in this process. We define 
	\[\underline{\mathfrak{Ds}}(\mathfrak{n})=\left\{\Delta_{1} , \Delta_{2}, \cdots ,\Delta_{r}\right\}.\]

\begin{algorithm}\label{alg:int:Lang}
Let $\mathfrak{m} \in \mathrm{Mult}_\rho$ and $\Delta=[a,b]_\rho \in \mathrm{Seg}_\rho$. Define the following multisegment \[\mathfrak{m}_1=\mathfrak{m}_{[a,b]}=\{[a^\prime, b^\prime]_\rho \in \mathfrak{m} \mid a \leq a^\prime \leq b+1 \leq b^\prime+1\}.\] 

Step 1. (Arrange downward sequences): Let $\underline{\mathfrak{Ds}}(\mathfrak{m}_1)=\left\{\Delta_{1,1} , \Delta_{1,2}, \cdots ,\Delta_{1,r_1}\right\}$ with $\Delta_{1,q} \prec \Delta_{1,q-1}$. Recursively for $2 \leq p \leq k$, we define, $\mathfrak{m}_{p}= \mathfrak{m}_{p-1} -\underline{\mathfrak{Ds}}(\mathfrak{m}_{p-1})$ and the corresponding downward sequence \[\underline{\mathfrak{Ds}}(\mathfrak{m}_p)= \left\{\Delta_{p,1}, \Delta_{p,2},...,\Delta_{p,r_p}\right\}, \text{ where } \Delta_{p,r_p} \prec \cdots \prec \Delta_{p,2} \prec \Delta_{p,1},\]	
such that $k$ is the smallest integer for which $\mathfrak{m}_{k+1}=\emptyset$. 

Step 2. (Addable free points): Set $\Delta_{p,q}=[a_{p,q},b_{p,q}]_\rho$. We define the `addable free points' set for the segment $\Delta_{p,q}$ for each $1 \leq p \leq k$ by:
	\[\mathfrak{af}\left(\Delta_{p,q} \right)= \begin{cases}
		\left\{\left[a_{p,q+1}+1 \right]_\rho,...,\left[a_{p,q}-1 \right]_\rho\right\} &\mbox{ if } q < r_p \text{ and } a_{p,q+1} \leq a_{p,q}-2,\\
		\left\{[a]_\rho,[a+1]_\rho,...,\left[a_{p,q}-1\right]_\rho\right\} &\mbox{ if } q=r_p \text{ and } a < a_{p,q},
	\end{cases}\]
otherwise, we write $\mathfrak{af}\left(\Delta_{p,q}\right)=\emptyset$. 

Step 3. (Selection): We now perform the following algorithm by picking the addable free points: find the largest index $p_1$ such that $[a]_\rho \in \mathfrak{af}\left(\Delta_{p_1,q_1} \right)$ for some $1 \leq q_1 \leq r_{p_1}$. Recursively for $t \geq 2$, we find the largest index $p_t < p_{t-1}$ such that $[a_{p_{t-1},q_{t-1}}]_\rho \in \mathfrak{af}\left(\Delta_{p_t,q_t} \right)$ for some $1 \leq q_t \leq r_{p_t}$. This process terminates after finite times, say $\ell$ times. 

Step 4. (Expand and replace): We define new extended segments as follows:
	\begin{align*}
		\Delta_{p_1, q_1}^\mathrm{ex} &= \left[a,  b_{p_{1}, q_{1}}  \right]_\rho;\\
		\Delta_{p_t, q_t}^\mathrm{ex} &= \left[a_{p_{t-1}, q_{t-1}}, ~ b_{p_{t}, q_{t}} \right]_\rho~ \text{ for } 2 \leq t \leq \ell;\\ 
		\Delta_{p_{\ell +1}, q_{\ell +1}}^\mathrm{ex} &= \left[a_{p_{\ell}, q_{\ell}}, b \right]_\rho    
	\end{align*}

As convention, $[c, c-1]_\rho = \emptyset$. Finally, we define the right integral multisegment by
 \begin{equation}\label{eq:int_Lang_[a,b]}
 \mathcal{I}^\mathrm{Lang}_{[a,b]_\rho}(\mathfrak{m}) := \mathfrak{m} - \sum\limits_{t=1}^{\ell} \Delta_{p_t,q_t} + \sum\limits_{t=1}^{\ell+1} \Delta^\mathrm{ex}_{p_t,q_t}.    
 \end{equation}   
 We shall say that $\Delta_{p_1,q_1}, \ldots, \Delta_{p_{\ell},q_{\ell}}$ participate in the extension process for $\mathcal I^{\mathrm{Lang}}_{[a,b]_{\rho}}(\mathfrak m)$.
\end{algorithm}
\begin{remark}
    It can be easily observed that $\mathcal{I}^\mathrm{Lang}_{[a,b]_\rho}(\mathfrak{m})=\mathcal{I}^\mathrm{Lan}_{[a]_\rho}(\mathfrak{m})$ when $b=a$. For the rest of the article, we use this fact without mentioning it further.
\end{remark}

\begin{remark} \label{rmk lang alg matching}
One may view that finding downward sequences of maximally linked segments in Algorithm \ref{alg:int:Lang} above is to look for the matching in the sense of Lapid-M\'inguez \cite[Secton 5.13]{LM16} i.e. for given $\mathfrak m\in \mathrm{Mult}_{\rho}$ and $[a,b]_{\rho} \in \mathrm{Seg}_{\rho}$, look for an injective function 
\[ f:\left\{ [a',b']_{\rho} \in \mathfrak m: a<a'\leq b+1 \leq b'+1 \right\}\rightarrow \left\{ [a',b']_{\rho} \in \mathfrak m: a\leq a'<b+1<b'+1 \right\} \] 
such that $f(\Delta) \prec \Delta$.

A new input of our algorithm is the notion of addable free points to tell precisely which segments have to be expanded in (\ref{eq:int_Lang_[a,b]}) for the general case in computing $\mathrm{I}^{\mathrm R}_{[a,b]_{\rho}}(\mathfrak m)$. If $\nu^a\rho$ is not an addable free point for any segment $\Delta_{p,q}$ (notations in Algorithm \ref{alg:int:Lang}), then such matching function exists.

\end{remark}

\begin{example}
	Let $\mathfrak{m}=\left\{ [1]_\rho, [1,2]_\rho, [2,4]_\rho, [4,6]_\rho \right\}$. We have the following $\mathcal{I}^\mathrm{Lang}_{\Delta}(\mathfrak{m})$:
	\begin{itemize}
		\item[(i)] Let $\Delta=[1,2]_\rho$. Then, $\mathfrak{m}_1=\left\{[1,2]_\rho, [2,4]_\rho\right\}$ and there is no segment in $\mathfrak{m}_1$ contributing the free point $[1]_\rho$. Therefore, $\mathcal{I}^\mathrm{Lang}_{[1,2]_\rho}(\mathfrak{m}) = \mathfrak{m} + [1,2]_\rho.$
		\item[(ii)] Let $\Delta=[1,3]_\rho$. Then, $\mathfrak{m}_1=\left\{[2,4]_\rho, [4,6]_\rho\right\}$. In $\mathfrak{m}_1$, the segment contributing the free point $[1]_\rho$ is $[2,4]_\rho$, and there is no segment in $\mathfrak{m}_1$ contributing the free point $[2]_\rho$. Therefore, $\mathcal{I}^\mathrm{Lang}_{[1,3]_\rho}(\mathfrak{m}) = \mathfrak{m} - [2,4]_\rho + [1,4]_\rho + [2,3]_\rho.$ \qed
	\end{itemize} 
\end{example}

We first have two useful properties of segments from the above algorithm:

\begin{lemma} \label{lem refined relation in alg}
We use the notations in Algorithm \ref{alg:int:Lang}.  Let $\Delta \in \underline{\mathfrak{Ds}}(\mathfrak m_p)$ and let $\Delta' \in \underline{\mathfrak{Ds}}(\mathfrak m_{p'})$. If $p<p'$, it cannot happen that $\Delta'\subsetneq \Delta$.
\end{lemma}

\begin{proof}
Recall that $\Delta_{p,1}, \ldots,   \Delta_{p,r_p}$ be the segments in $\underline{\mathfrak{Ds}}(\mathfrak m_p)$. Then $\Delta=\Delta_{p,k}$ for some $k=1,\ldots, r_p$. Suppose $\Delta' \subsetneq \Delta$. Now, since $e(\Delta_{p,y})>e(\Delta_{p,k})\geq e(\Delta')$ for $y<k$, we can only have that $\Delta' \prec \Delta_{p,y}$ or $ \Delta' \subset \Delta_{p,y}$. However, the former case is not possible from the choices of the algorithm. Thus, we must also have that $\Delta' \subset \Delta_{p,1}$. However, we then should choose $\Delta'$ first in the algorithm before picking $\Delta_{p,1}$, and hence we arrive at a contradiction.
\end{proof}

\begin{lemma} \label{lem segment stucture}
We use the notations in Algorithm \ref{alg:int:Lang}. Let $a\leq c \leq b$ and let $p, p' \in \left\{ 1, \ldots, k \right\}$. Suppose $\underline{\mathfrak{Ds}}(\mathfrak m_p)[c]\neq \emptyset$ and $\underline{\mathfrak{Ds}}(\mathfrak m_p)[c]\neq \emptyset$. Let $\Delta$ and $\Delta'$ be the unique segments in $\underline{\mathfrak{Ds}}(\mathfrak m_p)[c]$ and $\underline{\mathfrak{Ds}}(\mathfrak m_{p'})[c]$ respectively. If $p<p'$, then $\Delta \subset \Delta'$.
\end{lemma}

\begin{proof}
This is a reformulation of Lemma \ref{lem refined relation in alg}.
%This follows from the shortest choices of segments in downward sequences, and Lemma \ref{lem either prec or overlap}.
\end{proof}

\subsection{Reduction to $\mathfrak m_{[a,b]}$} For $\mathfrak m \in \mathrm{Mult}_{\rho}$,  and $[a,b]_{\rho}\in \mathrm{Seg}_{\rho}$, we recall \[\mathfrak m_{[a,b]}=\left\{ [a',b']_{\rho} \in \mathfrak m\ |\ a \leq a' \leq b+1\leq b'+1 \right\}.\]

\begin{lemma} \label{lem reduction to m1}
Let $\mathfrak m \in \mathrm{Mult}_{\rho}$ and $[a,b]_{\rho}\in \mathrm{Seg}_{\rho}$. Suppose $L(\mathcal I^{\mathrm{Lang}}_{[a,b]_{\rho}}(\mathfrak m_{[a,b]}))=\mathrm I^{\mathrm R}_{[a,b]_{\rho}}(L(\mathfrak m_{[a,b]}))$. Then \[L(\mathcal I^{\mathrm{Lang}}_{[a,b]_{\rho}}(\mathfrak m))= \mathrm I^{\mathrm R}_{[a,b]_{\rho}}(L(\mathfrak m)).\]
\end{lemma}
\begin{proof}
Set $\mathfrak m_1=\mathfrak m_{[a,b]}$. By Theorem \ref{thm:der:Lang}, it suffices to show that $\mathcal D^{\mathrm{Lang}}_{[a,b]_{\rho}}\circ \mathcal I^{\mathrm{Lang}}_{[a,b]_{\rho}}(\mathfrak m)=\mathfrak m$. On the other hand, the assumption $L(\mathcal I^{\mathrm{Lang}}_{[a,b]_{\rho}}(\mathfrak m_1))=\mathrm I^{\mathrm R}_{[a,b]_{\rho}}(L(\mathfrak m_1))$ implies that $\mathrm D^{\mathrm R}_{[a,b]_{\rho}}(L(\mathcal I^{\mathrm{Lang}}_{[a,b]_{\rho}}(\mathfrak m_1)))=L(\mathfrak m_1)$, and so by Theorem \ref{thm:der:Lang}, $\mathcal D^{\mathrm{Lang}}_{[a,b]_{\rho}}\circ \mathcal I^{\mathrm{Lang}}_{[a,b]_{\rho}}(\mathfrak m_1)=\mathfrak m_1$.

Let $\mathfrak m'=\mathfrak m-\mathfrak m_1$. It follows from Algorithm \ref{alg:int:Lang} that 
$\mathcal I^{\mathrm{Lang}}_{[a,b]_{\rho}}(\mathfrak m) = \mathcal I^{\mathrm{Lang}}_{[a,b]_{\rho}}(\mathfrak m_1)+\mathfrak m' $.
But it follows from Algorithm \ref{alg:der:Lang}, $\mathfrak m'$ also plays no role in that algorithm. Thus, 
\[  \mathcal D^{\mathrm{Lang}}_{[a,b]_{\rho}}( \mathcal I^{\mathrm{Lang}}_{[a,b]_{\rho}}(\mathfrak m_1)+\mathfrak m')=\mathcal D^{\mathrm{Lang}}_{[a,b]_{\rho}} \circ \mathcal I^{\mathrm{Lang}}_{[a,b]_{\rho}}(\mathfrak m_1)+\mathfrak m' =\mathfrak m_1+\mathfrak m'.
\]
Now the lemma follows from the above discussions.
%Note that for $\Delta \in \mathfrak m$ and $\Delta' \in \mathfrak m'$, if $\Delta \prec \Delta'$. This implies that for an upward sequence found by Algorithm \ref{alg:der:Lang} for $\mathcal I^{\mathrm{Lang}}_{[a,b]_{\rho}}(\mathfrak m)$, with segments lablled by
%\[  \Delta_1 \prec \ldots \prec \Delta_k ,
%\]
%\begin{enumerate}
%\item  if $\Delta_i \in \mathfrak n$, then for any $j<i$, $\Delta_j \in \mathfrak n$; and
%\item  if $\Delta_i \in \mathfrak m_1$, then for any $i<j$, $\Delta_j \in \mathfrak m_1$.
%\end{enumerate}
\end{proof}

\subsection{Transfer between integrals and derivatives by exotic duality} \label{ss transfer exotic duality}

Let $[a,b]_{\rho} \in \mathrm{Seg}_{\rho}$. A multisegment $\mathfrak m \in \mathrm{Mult}_{\rho}$ is said to be in good range for $[a,b]_{\rho}$ if $\mathfrak m=\mathfrak m_{[a,b]}$ that means for any $\Delta \in \mathfrak m$,
\[      a \leq s(\Delta) \leq b+1 \leq e(\Delta)+1 .
\]
For any $\mathfrak m \in \mathrm{Mult}_{\rho}$ in good range for $[a,b]_{\rho}$ and $r \in \mathbb Z_{>0}$, we define
\[ \mathbb D_r(\mathfrak m)= \left\{ [-r+b'+1, a'-1]_{\rho}\ |\  [a',b']_{\rho} \in \mathfrak m \right\}, \quad \mathbb D_{r}^{[a,b]_{\rho}}(\mathfrak m) = \mathbb D_r(\mathfrak m) + [b-r+1, b]_{\rho} .
\]

\begin{example}
\begin{enumerate}
\item Let $\mathfrak m=\left\{ [2,4]_{\rho}, [1,7]_{\rho} \right\}$. Then, $\mathbb D_{10}(\mathfrak m)=\left\{ [-5,1]_{\rho}, [-2,0]_{\rho}\right\}$ and 
\[ \mathbb D^{[0,1]_{\rho}}_{10}(\mathfrak m)=\left\{ [-5,1]_{\rho}, [-2,0]_{\rho}\right\}+[-8,1]_{\rho} .
\]
\item Let $\mathfrak m=\left\{ [2,6]_{\rho}, [1,5]_{\rho}\right\}$. Then, $\mathbb D_{15}(\mathfrak m)=\left\{[-8,1]_{\rho}, [-9,0]_{\rho} \right\}$ and 
\[  \mathbb D^{[1,4]_\rho}_{15}(\mathfrak m)=\left\{[-8,1]_{\rho}, [-9,0]_{\rho}, [-10,4]_{\rho} \right\} .
\]
\end{enumerate}
\end{example}

\begin{proposition} \label{prop dual derivative integral}
Let $[a,b]_{\rho}\in \mathrm{Seg}_{\rho}$. Let $\mathfrak m \in \mathrm{Mult}_{\rho}$ be in good range for $[a,b]_{\rho}$. Then, for sufficiently large $r \in \mathbb Z_{>0}$,
\begin{enumerate}
\item[(i)] if $|\mathcal I^{\mathrm{Lang}}_{[a,b]_{\rho}}(\mathfrak m)|=|\mathfrak m|$, then $\infty \neq \mathcal D^{\mathrm{Lang}, \mathrm{L}}_{[a,b]_{\rho}}(\mathbb D_r(\mathfrak m)) =\mathbb D_r(\mathcal I^{\mathrm{Lang}}_{[a,b]_{\rho}}(\mathfrak m))$.
\item[(ii)] if $|\mathcal I^{\mathrm{Lang}}_{[a,b]_{\rho}}(\mathfrak m)|>|\mathfrak m|$, then $\infty \neq \mathcal D^{\mathrm{Lang}, \mathrm{L}}_{[a,b]_{\rho}}(\mathbb D^{[a,b]_{\rho}}_r(\mathfrak m)) =\mathbb D_r(\mathcal I^{\mathrm{Lang}}_{[a,b]_{\rho}}(\mathfrak m))$.
\item[(iii)] $|\mathcal I^{\mathrm{Lang}}_{[a,b]_{\rho}}(\mathfrak m)|=|\mathfrak m|$ if and only if $\mathcal D^{\mathrm{Lang}, \mathrm{L}}_{[a,b]_{\rho}}(\mathbb D_r(\mathfrak m))\neq \infty$.
\end{enumerate}
\end{proposition}

It is quite straightforward to prove Proposition \ref{prop dual derivative integral}. The key is to translate objects between the algorithms under $\mathbb D_r$ (see Table \ref{Tab:duality alg} for a summary). We refer the interested reader to Appendix \ref{appendix prop duality} for a detailed check of Proposition \ref{prop dual derivative integral}.

\begin{table}
\begin{tabular}{|c|c|}
\hline
  Algorithm $\ref{alg:der:Lang}$ & Algorithm \ref{alg:int:Lang} \\
\hline
Upward sequences & Downward sequences \\
 \hline 
 Removable free points & Addable free points \\
 \hline 
 Truncations & Extensions  \\
 \hline 
\end{tabular}
\caption{Correspondences under exotic duality $\mathbb D_r$}
\label{Tab:duality alg}
\end{table}

When we later write $\mathbb D_r$, we shall assume $r$ is any sufficiently large integer.

\begin{lemma} \label{lem zero and integral cuspidal}
Let $[a,b]_{\rho}\in \mathrm{Seg}_{\rho}$ with $b>a$ and let $\mathfrak m \in \mathrm{Mult}_{\rho}$ be in good range for $[a,b]_{\rho}$. Then the following statements are equivalent:
\begin{enumerate}
    \item[(i)] $|\mathcal I^{\mathrm{Lang}}_{[a]_{\rho}}(\mathfrak m)|=|\mathfrak m|$;
    \item[(ii)] $\mathcal D^{\mathrm{Lang,L}}_{[a]_{\rho}}(\mathbb D_r(\mathfrak m))\neq \infty$; 
    \item[(iii)] $\mathcal D^{\mathrm{Lang,L}}_{[a]_{\rho}}(\mathbb D^{[a,b]_{\rho}}_r(\mathfrak m))\neq \infty$.
\end{enumerate}
\end{lemma}
%{\color{blue} (i)$\Leftrightarrow$(ii) follows from Proposition \ref{prop dual derivative integral}(iii) and  (ii)$\Leftrightarrow$(iii) since $-a > -b$ and $[-b,r-b-1]_{\rho^\vee}$ can not participate in the $\mathfrak{tds}(-,-a)$ process on $\Theta \left(\mathbb D_r(\mathfrak m) + [b-r+1,b]_\rho \right)$.}
\begin{proof}
Note that (i)$\Leftrightarrow$(ii) follows from a similar argument in Proposition \ref{prop dual derivative integral}(iii) and is simpler. (ii)$\Leftrightarrow$(iii) since $-a > -b$ and $[-b,r-b-1]_{\rho^\vee}$ can not participate in the $\mathfrak{tds}(-,-a)$ process on $\Theta \left(\mathbb D_r(\mathfrak m) + [b-r+1,b]_\rho \right)$.

%(iii)$\Rightarrow$(i): Suppose $|\mathcal I^{\mathrm{Lang}}_{[a]_{\rho}}(\mathfrak m)|\neq |\mathfrak m|$. Then, by Proposition \ref{int:Lang:length=1}, there exists an injective map  $f: \mathfrak m[a+1]\rightarrow \mathfrak m[a]$ such that, for any $\Delta \in \mathfrak m[a+1]$, $f(\Delta) \prec \Delta$.

%Translating under the map $\Psi: \mathfrak m \rightarrow \mathbb D_r(\mathfrak m)$ given by $\Psi([a',b']_{\rho})=[b'-r+1, a^\prime -1]_{\rho}$, we obtain an injective map $\widetilde{f}$ from $\mathbb D^{[a,b]_{\rho}}_r(\mathfrak m)\langle a\rangle$ to $\mathbb D^{[a,b]_{\rho}}_r(\mathfrak m)\langle a-1\rangle$ such that $\widetilde{f}(\Delta)  \prec \Delta$. This implies $\mathcal D^{\mathrm{Lang,L}}_{[a]_{\rho}}(\mathfrak m) = \infty$ by the left version of Theorem \ref{thm:der:Lang:length=1}. This proves (iii)$\Rightarrow$(i).

%Note that the segment $[b-r+1 ,b]_{\rho}$ plays no role in defining above injective map $\widetilde{f}$, and one can reverse the above argument to show the (i)$\Rightarrow$(iii). Same argument also shows (i)$\Leftrightarrow$(ii).
\end{proof}

\begin{lemma} \label{lem duality under cuspidal derivative}
Let $[a,b]_{\rho}\in \mathrm{Seg}_{\rho}$ with $b>a$ and let $\mathfrak m \in \mathrm{Mult}_{\rho}$ be in good range.  If $|\mathcal I^{\mathrm{Lang}}_{[a]_{\rho}}(\mathfrak m)|=|\mathfrak m|$ or $\mathcal D^{\mathrm{Lang}, \mathrm{L}}_{[a]_{\rho}}(\mathbb D_r(\mathfrak m))\neq \infty$, then 
\[ \mathcal D_{[a]_{\rho}}^{\mathrm{Lang},\mathrm{L}}(\mathbb D_r(\mathfrak m))=\mathbb D_r(\mathcal I^{\mathrm{Lang}}_{[a]_{\rho}}(\mathfrak m)), \quad \mathcal D^{\mathrm{Lang}, \mathrm{L}}_{[a]_{\rho}}(\mathbb D^{[a,b]_{\rho}}_r(\mathfrak m)) =\mathbb D^{[a,b]_{\rho}}_r(\mathcal I^{\mathrm{Lang}}_{[a]_{\rho}}(\mathfrak m)) .
\]
\end{lemma}

\begin{proof}
A proof is similar to the one of Proposition \ref{prop dual derivative integral} (see Appendix \ref{appendix prop duality}) and is much simpler. We omit the details.
\end{proof}

\subsection{More on commutation relation of derivatives}

\begin{lemma} \label{lem derivative commut revise}
Let $\mathfrak m \in \mathrm{Mult}_{\rho}$. Let $[a,b]_{\rho} \in \mathrm{Seg}_{\rho}$. Let $a<c \leq b$. Then 
\begin{enumerate}
    \item[(i)] Suppose $\mathcal D^{\mathrm{Lang}}_{[a,b]_{\rho}}(\mathfrak m)\neq \infty$ and $\mathcal D^{\mathrm{Lang}}_{[c]_{\rho}}(\mathfrak m)\neq \infty$. Then 
    \[ \mathcal D^{\mathrm{Lang}}_{[c]_{\rho}}\circ \mathcal D^{\mathrm{Lang}}_{[a,b]_{\rho}}(\mathfrak m)=\mathcal D^{\mathrm{Lang}}_{[a,b]_{\rho}}\circ \mathcal D^{\mathrm{Lang}}_{[c]_{\rho}}(\mathfrak m)\neq \infty.\]
    \item[(ii)] If $\mathcal D^{\mathrm{Lang}}_{[a,b]_{\rho}}\circ \mathcal D^{\mathrm{Lang}}_{[c]_{\rho}}(\mathfrak m)\neq \infty$, then $\mathcal D^{\mathrm{Lang}}_{[a,b]_{\rho}}(\mathfrak m)\neq \infty$ and $\mathcal D^{\mathrm{Lang}}_{[c]_{\rho}}(\mathfrak m)\neq \infty$.
\end{enumerate}
\end{lemma}

\begin{proof}
For (i), by Theorem \ref{thm:zero_der_Lang}, $\mathrm{D}^{\mathrm R}_{[a,b]_{\rho}}(L(\mathfrak m))\neq 0$ and $\mathrm{D}^{\mathrm R}_{[c]_{\rho}}(L(\mathfrak m))\neq 0$. This follows from Lemma \ref{lem derivative removal} that $\mathrm D^{\mathrm R}_{[c]_{\rho}}\circ \mathrm D^{\mathrm R}_{[a,b]_{\rho}}(\mathfrak m)\neq 0$ and so we have the commutativity (see e.g. \cite[Lemma 4.4]{Cha_csq}). Now, one applies Theorem \ref{thm:zero_der_Lang} to obtain (i). 

For (ii), Theorem \ref{thm:zero_der_Lang} implies $\mathrm{D}^{\mathrm{R}}_{[a,b]_{\rho}}\circ \mathrm{D}^{\mathrm R}_{[c]_{\rho}}(L(\mathfrak m))\neq 0$. Hence, $\mathrm D^{\mathrm{R}}_{[c]_{\rho}}(L(\mathfrak m))\neq 0$, and by \cite[Proposition 9.3(2)]{Cha_csq}, $\mathrm{D}^{\mathrm{R}}_{[a,b]_{\rho}}(L(\mathfrak m))\neq 0$. Now, one applies  Theorem \ref{thm:zero_der_Lang} to obtain statements for $\mathcal D^{\mathrm{Lang},L}_{[a,b]_{\rho}}$.
\end{proof}

We shall need the following left version of Lemma \ref{lem derivative commut revise}:

\begin{corollary} \label{cor derivative revise left}
Let $\mathfrak m \in \mathrm{Mult}_{\rho}$. Let $[a,b]_{\rho} \in \mathrm{Seg}_{\rho}$. Let $a\leq c < b$. Then 
\begin{enumerate}
    \item[(i)] Suppose $\mathcal D^{\mathrm{Lang}, \mathrm L}_{[a,b]_{\rho}}(\mathfrak m)\neq \infty$ and $\mathcal D^{\mathrm{Lang}, \mathrm L}_{[c]_{\rho}}(\mathfrak m)\neq \infty$. Then 
    \[ \mathcal D^{\mathrm{Lang}, \mathrm L}_{[c]_{\rho}}\circ \mathcal D^{\mathrm{Lang,L}}_{[a,b]_{\rho}}(\mathfrak m)=\mathcal D^{\mathrm{Lang}, \mathrm L}_{[a,b]_{\rho}}\circ \mathcal D^{\mathrm{Lang}, \mathrm L}_{[c]_{\rho}}(\mathfrak m)\neq \infty.\]
    \item[(ii)] If $\mathcal D^{\mathrm{Lang}, \mathrm L}_{[a,b]_{\rho}}\circ \mathcal D^{\mathrm{Lang}, \mathrm L}_{[c]_{\rho}}(\mathfrak m)\neq \infty$, then $\mathcal D^{\mathrm{Lang}, \mathrm L}_{[a,b]_{\rho}}(\mathfrak m)\neq \infty$ and $\mathcal D^{\mathrm{Lang}, \mathrm{L}}_{[c]_{\rho}}(\mathfrak m)\neq \infty$..
\end{enumerate}
\end{corollary}

\subsection{Commutation of $\mathcal I_{[a]_{\rho}}^{\mathrm{Lang}}$ and $\mathcal I_{[a,b]_{\rho}}^{\mathrm{Lang}}$}

\begin{lemma}\label{lem:Lang_comm_I[ab]_I[a]}
	Let $[a,b]_\rho \in \mathrm{Seg}_\rho$ with $b>a$ and $\mathfrak{m} \in \mathrm{Mult}_\rho$ be in good range for $[a,b]_{\rho}$. Then,  we have:
	\[\mathcal{I}^\mathrm{Lang}_{[a,b]_\rho}\circ \mathcal{I}^\mathrm{Lang}_{[a]_\rho}(\mathfrak{m})=\mathcal{I}^\mathrm{Lang}_{[a]_\rho}\circ \mathcal{I}^\mathrm{Lang}_{[a,b]_\rho}(\mathfrak{m}).\]
\end{lemma}

\begin{proof}
By Algorithm \ref{alg:int:Lang}, we must have $|\mathcal I^{\mathrm{Lang}}_{[a]_{\rho}}(\mathfrak m)|\geq|\mathfrak m|$. We shall divide it into two cases:
\begin{enumerate}
\item Case 1: $|\mathcal I^{\mathrm{Lang}}_{[a]_{\rho}}(\mathfrak m)|=|\mathfrak m|$. We only show the steps for $|\mathcal I^{\mathrm{Lang}}_{[a,b]_{\rho}}\circ \mathcal I^{\mathrm{Lang}}_{[a]_{\rho}}(\mathfrak m)|>|\mathcal I^{\mathrm{Lang}}_{[a]_{\rho}}(\mathfrak m)|$, and the other case is similar.
\begin{align*}
   \mathbb D_r( \mathcal{I}^\mathrm{Lang}_{[a,b]_\rho}\circ \mathcal{I}^\mathrm{Lang}_{[a]_\rho}(\mathfrak{m}))&=    \mathcal{D}^\mathrm{Lang, \mathrm{L}}_{[a,b]_\rho}( \mathbb D^{[a,b]_{\rho}}_r(\mathcal{I}^\mathrm{Lang}_{[a]_\rho}(\mathfrak{m})))\quad \mbox{(by Proposition \ref{prop dual derivative integral})} \\
     &=\mathcal D^{\mathrm{Lang}, \mathrm{L}}_{[a,b]_{\rho}}\circ \mathcal D^{\mathrm{Lang}, \mathrm{L}}_{[a]_{\rho}}(\mathbb D^{[a,b]_{\rho}}_r(\mathfrak m)) \quad \mbox{(by Lemma \ref{lem duality under cuspidal derivative})} \\
     &= \mathcal D^{\mathrm{Lang}, \mathrm L}_{[a]_{\rho}}\circ \mathcal D^{\mathrm{Lang}, \mathrm{L}}_{[a,b]_{\rho}}(\mathbb D^{[a,b]_{\rho}}_r(\mathfrak m)) \quad \mbox{(by Corollary \ref{cor derivative revise left}(ii) and (i))}\\
     &= \mathcal D^{\mathrm{Lang}, \mathrm L}_{[a]_{\rho}}(\mathbb D_r(\mathcal I_{[a,b]_{\rho}}^{\mathrm{Lang}}(\mathfrak m)) \quad \mbox{(by Proposition \ref{prop dual derivative integral}(ii))} \\
     &= \mathbb D_r(\mathcal I^{\mathrm{Lang}}_{[a]_{\rho}}\circ \mathcal I_{[a,b]_{\rho}}^{\mathrm{Lang}}(\mathfrak m)) \quad \mbox{(by Lemma \ref{lem duality under cuspidal derivative})}
\end{align*}
%{\color{blue} As $|\mathcal I^{\mathrm{Lang}}_{[a]_{\rho}}(\mathfrak m)|=|\mathfrak m|$ and $|\mathcal I^{\mathrm{Lang}}_{[a,b]_{\rho}}\circ \mathcal I^{\mathrm{Lang}}_{[a]_{\rho}}(\mathfrak m)|>|\mathcal I^{\mathrm{Lang}}_{[a]_{\rho}}(\mathfrak m)|$, one can easily observe that $|\mathcal I^{\mathrm{Lang}}_{[a,b]_{\rho}}(\mathfrak m)|>|\mathfrak m|$. For third equality, the required conditions to apply Corollary \ref{cor derivative revise left} are $\mathcal D^{\mathrm{Lang}, \mathrm L}_{[a]_{\rho}}\left(\mathbb D^{[a,b]_{\rho}}_r(\mathfrak m) \right) \neq \infty$ and $\mathcal D^{\mathrm{Lang}, \mathrm L}_{[a,b]_{\rho}}\left(\mathbb D^{[a,b]_{\rho}}_r(\mathfrak m) \right) \neq \infty$, those follow from Proposition \ref{lem zero and integral cuspidal} and \ref{prop dual derivative integral}(ii) by applying $|\mathcal I^{\mathrm{Lang}}_{[a]_{\rho}}(\mathfrak m)|=|\mathfrak m|$ and $|\mathcal I^{\mathrm{Lang}}_{[a,b]_{\rho}}(\mathfrak m)|>|\mathfrak m|$ respectively.}

%({\color{green}To give a proof of $|\mathcal I^{\mathrm{Lang}}_{[a,b]_{\rho}}(\mathfrak m)|>|\mathfrak m|$, we can write the following:} 

 For the fourth equality, we can show the condition $|\mathcal I^{\mathrm{Lang}}_{[a,b]_{\rho}}(\mathfrak m)|>|\mathfrak m|$ as follows: One sees from the previous expressions that the segment $[-r+b+1,b]_{\rho}$ has to be truncated for $\mathcal D^{\mathrm{Lang},\mathrm{L}}_{[a,b]_{\rho}}(\mathbb D^{[a,b]_{\rho}}_r(\mathfrak m))$, and this implies $\mathcal D^{\mathrm{Lang}, \mathrm{L}}_{[a,b]_{\rho}}(\mathbb D_r(\mathfrak m))=\infty$, which implies  $|\mathcal I^{\mathrm{Lang}}_{[a,b]_{\rho}}(\mathfrak m)|>|\mathfrak m|$ by Proposition \ref{prop dual derivative integral}(iii).

\item Case 2: $|\mathcal I^{\mathrm{Lang}}_{[a]_{\rho}}(\mathfrak m)|>|\mathfrak m|$. We consider $|\mathcal I^{\mathrm{Lang}}_{[a,b]_{\rho}}(\mathfrak m)|=|\mathfrak m|$ and the other case only needs some notation changes. Then, $ \mathcal I^{\mathrm{Lang}}_{[a]_{\rho}}(\mathfrak m)=\mathfrak m+[a]_{\rho}$. Note that $[a]_{\rho}$ has no role in running Algorithm \ref{alg:int:Lang} for $\mathcal I^{\mathrm{Lang}}_{[a,b]_{\rho}}( \mathcal I^{\mathrm{Lang}}_{[a]_{\rho}}(\mathfrak m))$. Hence, 
\begin{align*}
 \mathcal I^{\mathrm{Lang}}_{[a,b]_{\rho}}\circ \mathcal I^{\mathrm{Lang}}_{[a]_{\rho}}(\mathfrak m) & =\mathcal I^{\mathrm{Lang}}_{[a,b]_{\rho}}(\mathfrak m)+[a]_{\rho} .
\end{align*}
Now, it suffices to show $|\mathcal I^{\mathrm{Lang}}_{[a]_{\rho}}\circ \mathcal I^{\mathrm{Lang}}_{[a,b]_{\rho}}(\mathfrak m)|>|\mathcal I^{\mathrm{Lang}}_{[a,b]_{\rho}}(\mathfrak m)|$. Otherwise, $|\mathcal I^{\mathrm{Lang}}_{[a]_{\rho}}\circ \mathcal I^{\mathrm{Lang}}_{[a,b]_{\rho}}(\mathfrak m)|=|\mathcal I^{\mathrm{Lang}}_{[a,b]_{\rho}}(\mathfrak m)|$ and so by Lemma \ref{lem duality under cuspidal derivative},
\[  \mathbb D_r(\mathcal I^{\mathrm{Lang}}_{[a]_{\rho}}\circ \mathcal I^{\mathrm{Lang}}_{[a,b]_{\rho}}(\mathfrak m))=\mathcal D^{\mathrm{Lang}, \mathrm{L}}_{[a]_{\rho}} \left( \mathbb D_r\left(\mathcal I^{\mathrm{Lang}}_{[a,b]_{\rho}}(\mathfrak m) \right) \right) \neq \infty.
\]
Therefore, by Proposition \ref{prop dual derivative integral} and the equality $|\mathcal I^{\mathrm{Lang}}_{[a,b]_{\rho}}(\mathfrak m)|=|\mathfrak m|$, we have:
\[   \infty\neq  \mathcal D^{\mathrm{Lang}, \mathrm{L}}_{[a]_{\rho}} \left( \mathbb D_r\left(\mathcal I^{\mathrm{Lang}}_{[a,b]_{\rho}}(\mathfrak m) \right) \right) =\mathcal D^{\mathrm{Lang}, \mathrm{L}}_{[a]_{\rho}}\circ \mathcal D^{\mathrm{Lang}, \mathrm L}_{[a,b]_{\rho}}(\mathbb D_r(\mathfrak m)) .
\]
So $\mathcal D^{\mathrm{Lang},\mathrm{L}}_{[a]_{\rho}}\left(\mathbb D_r(\mathfrak m)\right) \neq \infty$ by Corollary \ref{cor derivative revise left}(ii). However, the given condition is $|\mathcal I^{\mathrm{Lang}}_{[a]_{\rho}}(\mathfrak m)|>|\mathfrak m|$ and so it contradicts to Lemma \ref{lem zero and integral cuspidal}. 
\end{enumerate}
\end{proof}

\begin{lemma}\label{lem:comm_I[ab]_I[a]_pi}
	Let $\pi \in \mathrm{Irr}_{\rho} $ and $[a,b]_\rho \in \mathrm{Seg}_\rho$. Then, $\mathrm{I}^\mathrm{R}_{[a,b]_\rho}\circ \mathrm{I}^\mathrm{R}_{[a]_\rho}(\pi) \cong \mathrm{I}^\mathrm{R}_{[a]_\rho}\circ \mathrm{I}^\mathrm{R}_{[a,b]_\rho}(\pi).$
\end{lemma}
\begin{proof}
	The proof follows from $\mathrm{St}\left([a]_\rho\right) \times \mathrm{St}\left([a,b]_\rho\right) \cong  \mathrm{St}\left([a,b]_\rho\right) \times \mathrm{St}\left([a]_\rho\right)$ and also \cite[Corollary 6.11]{LM16}. %as  $[a]_\rho,$ and $[a,b]_\rho$ are unlinked. 
\end{proof}

\subsection{Composition of integrals $\mathcal I^{\mathrm{Lang}}_{[a]_{\rho}}$ and $\mathcal I^{\mathrm{Lang}}_{[a+1,b]_{\rho}}$}

\begin{lemma}\label{lem:Lang_I[ab]=I[a]_I[a+1,b]}
	Let $\mathfrak{m} \in \mathrm{Mult}_\rho$ and $[a,b]_\rho \in \mathrm{Seg}_\rho$. If  $\mathcal{D}^\mathrm{Lang}_{[a]_\rho}(\mathfrak{m}) = \infty$, we then have
	\[\mathcal{I}^\mathrm{Lang}_{[a,b]_\rho}(\mathfrak{m})=\mathcal{I}^\mathrm{Lang}_{[a]_\rho}\circ \mathcal{I}^\mathrm{Lang}_{[a+1,b]_\rho}(\mathfrak{m})\]
\end{lemma}

\begin{remark}
Let $\mathfrak m=\left\{ [1,2]_{\rho}, [-1,0]_{\rho} \right\}$. In this case $\mathcal D^{\mathrm{Lan}}_{[1]_{\rho}}(\mathfrak m)\neq \emptyset$ and 
\[ \mathcal I_{[-1,0]_{\rho}}^{\mathrm{Lang}}(\mathfrak m)= \left\{ [-1,0]_{\rho}, [-1,0]_{\rho}, [1,2]_{\rho} \right\} \neq \mathcal I^{\mathrm{Lang}}_{[-1]_{\rho}}\circ I_{[0]_{\rho}}^{\mathrm{Lang}}(\mathfrak m) =\left\{ [-1]_{\rho}, [-1,0]_{\rho}, [0,2]_{\rho} \right\} . 
\]
This shows the infinity condition in Lemma \ref{lem:Lang_I[ab]=I[a]_I[a+1,b]} cannot be dropped.
\end{remark}

\begin{remark}
Note that one may also want to use the exotic duality to prove the following Lemma \ref{lem:Lang_I[ab]=I[a]_I[a+1,b]}. In order to do so, one still needs to translate the condition $\mathcal D^{\mathrm{Lang}}_{[a]_{\rho}}(\mathfrak m)=\infty$ under the duality $\mathbb D_r$ (the translations in Lemma \ref{lem zero and integral cuspidal} are not so useful here). We briefly explain such translation. We shall also use $\Theta$ in Section \ref{sss gk invol} to translate from left derivatives to right derivatives. Let $\mathfrak m \in \mathrm{Mult}_{\rho}$ and let $a \in \mathbb Z$. Let $r$ be a sufficiently large integer and let $b'=-a$. Let $\mathfrak m'=\Theta(\mathbb D_r(\mathfrak m)))$. Then $\mathcal D^{\mathrm{Lang}}_{[a]_{\rho}}(\mathfrak m)\neq \infty$ if and only if for any $a'< b'$ with $\mathcal D^{\mathrm{Lang}}_{[a',b'-1]_{\rho^{\vee}}}(\mathfrak m')\neq \infty$,
\[ \varepsilon^\mathrm{R}_{[b']_{\rho^{\vee}}}(\mathcal D^{\mathrm{Lang}}_{[a',b'-1]_{\rho^{\vee}}}(\mathfrak m'))=\varepsilon^\mathrm{R}_{[b']_{\rho^{\vee}}}(\mathfrak m')+1 .
\]
This translation combined with the machinery of highest derivative multisegments in Section \ref{ss hd removal} could also give a proof of Lemma \ref{lem:Lang_I[ab]=I[a]_I[a+1,b]}. However, proving such translation also takes some work, and we opt to use a more direct approach to show Lemma \ref{lem:Lang_I[ab]=I[a]_I[a+1,b]} below.
\end{remark}

%The proof does not make much simpler via the duality and so we shall simply give a more direct proof for the following lemma, which is given in Appendix \ref{s appendix lemma commute} for the interested reader. 

We now prove Lemma \ref{lem:Lang_I[ab]=I[a]_I[a+1,b]}. We use all the notations in Algorithm \ref{alg:int:Lang} for $\mathcal I^{\mathrm{Lang}}_{[a,b]_{\rho}}(\mathfrak m)$. In particular, $\underline{\mathfrak{Ds}}(\mathfrak m_p)$ is a downward sequence for the integral algorihtm for computing $\mathcal I^{\mathrm{Lang}}_{[a,b]_{\rho}}(\mathfrak m)$.

\begin{lemma} 
Each downward sequences for the integral algorithm for $\mathcal I^{\mathrm{Lang}}_{[a+1,b]_{\rho}}(\mathfrak m)$  is either
\begin{enumerate}
\item $\underline{\mathfrak{Ds}}(\mathfrak m_p)$ if $\underline{\mathfrak{Ds}}(\mathfrak m_p)[a]=\emptyset $ (i.e. $a_{p,r_p}>a$).
\item $\underline{\mathfrak{Ds}}(\mathfrak m_p)-\Delta_{p, r_p}$ if $\underline{\mathfrak{Ds}}(\mathfrak m_p)[a]\neq \emptyset$ (i.e. $a_{p,r_p}=a$).
\end{enumerate}
\end{lemma}

\begin{proof}
We have 
\[ \mathfrak m_{[a,b]}=\mathfrak m_{[a+1,b]}+\sum_{b'\geq b, [a,b']_{\rho}\in \mathfrak m} [a,b']_{\rho}  .
\]
Note that to obtain a downward sequence for $\mathcal I^{\mathrm{Lang}}_{[a,b]_{\rho}}(\mathfrak m)$, the segments $[a,b']_{\rho}$ ($b' \geq b$) are added after the last segments of downward sequences for $\mathcal I^{\mathrm{Lang}}_{[a+1,b]_{\rho}}(\mathfrak m)$, whenever possible. 
\end{proof}

Suppose $\mathcal D^{\mathrm{Lang}}_{[a]_{\rho}}(\mathfrak m)=\infty$ for the remaining of the proof.

\begin{lemma} \label{lem later add free}
 Suppose $\underline{\mathfrak{Ds}}(\mathfrak m_p)[a]\neq \emptyset$ and $\underline{\mathfrak{Ds}}(\mathfrak m_p)[a+1]=\emptyset$ i.e. $a_{p,r_p}=a$ and $a_{p,r_p-1}\neq a+1$. Then there exists $p'>p$ such that $a_{p',r_{p'}}=a+1$. In particular, the segment $\Delta_{p',r_{p'}}$ has $\nu^a\rho$ as an addable free point. 
\end{lemma}

\begin{proof}
We briefly explain this. Suppose it does not lead to a contradiction. We first observe the following two facts:
\begin{itemize}
\item For any $p''>p$ such that $\underline{\mathfrak{Ds}}(\mathfrak m_{p''})[a+1]\neq \emptyset$, we then have $\underline{\mathfrak{Ds}}(\mathfrak m_{p''})[a]\neq \emptyset$ i.e. $a_{p'', r_{p''}}=a$. This follows from what we are assuming.
\item For any $p''<p$, if $a_{p'',r_{p''}}=a+1$, then the segment $[a+1, b_{p'', r_{p''}}]_{\rho}$ is not linked to $[a, b_{p, r_p}]_{\rho}$. This follows from choices in a downward sequence.
\end{itemize}

Now one carries out Step 1 of Algorithm \ref{alg:rho_der:Lang}. However, by the above two bullets, one uses Lemma \ref{lem segment stucture} to deduce that there is a segment left in $\mathfrak m[a]$ after a sequence of removal steps of $\mathfrak{tds}(.,a)$-process. This contradicts that $\mathcal D^{\mathrm{Lang}}_{[a]_{\rho}}(\mathfrak m)=\infty$ by Algorithm \ref{alg:rho_der:Lang}. 
\end{proof}

With the above lemmas, one sees that 
\[ \mathcal I^{\mathrm{Lang}}_{[a+1,b]_{\rho}}(\mathfrak m)=\mathcal I^{\mathrm{Lang}}_{[a,b]_{\rho}}(\mathfrak m)-[a,b_{p_1,q_1}]+[a+1,b_{p_1,q_1}] .
\]
Note that in the case that $a_{p_1, q_1}=a+1$, one has $p_1=y_s$ and Lemma \ref{lem later add free} is useful in such case.

It remains to apply $\mathcal I^{\mathrm{Lang}}_{[a]_{\rho}}$ on $\mathcal I^{\mathrm{Lang}}_{[a+1,b]_{\rho}}(\mathfrak m)$. \ref{lem later add free}. We now need to understand how the segments in $\mathfrak m[a]$ and $\mathfrak m[a+1]$ distribute in the downward sequences $\underline{\mathfrak{Ds}}(\mathfrak m_p)$. We need to investigate $\mathfrak{tus}(\mathcal I^{\mathrm{Lang}}_{[a+1,b]_{\rho}}(\mathfrak m), a)$-process in the following lemma:

\begin{lemma} \label{lem exists l}
Let $x_1<x_2<\ldots <x_r$ be all the indices such that $\underline{\mathfrak{Ds}}(\mathfrak m_{x_k})[a]\neq \emptyset$ (whose segment is denoted by $\Delta_{x_k}'=\Delta_{x_k, r_{x_k}}$) and $\underline{\mathfrak{Ds}}(\mathfrak m_{x_k})[a+1]=\emptyset$. Let $y_1<y_2<\ldots<y_s$ be all the indices such that $\underline{\mathfrak{Ds}}(\mathfrak m_{y_k})[a]=\emptyset$ and $\underline{\mathfrak{Ds}}(\mathfrak m_{y_k})[a+1]\neq \emptyset$ (whose segment is denoted by $\Delta_{y_k}''=\Delta_{y_k, r_{y_k}}$). Then, 
\[   \Delta_{x_1}'\subset \Delta_{x_2}'\subset \ldots \subset \Delta_{x_r}',  \text{ and } \Delta_{y_1}''\subset \Delta_{y_2}''\subset \ldots \subset \Delta_{y_s}'' .
\]
Moreover, there exists an injective map 
\[ f: \left\{ x_1, \ldots, x_r \right\}  \rightarrow  \left\{  y_1, \ldots,y_s \right\}  .
\]
such that $f(x_1)<\ldots <f(x_r)$ and $\Delta_{x_k}'\prec \Delta_{f(x_k)}''$ for all $x_k$. 
\end{lemma}

\begin{proof}
The first assertion follows from Lemma \ref{lem segment stucture}. The second assertion follows from a straightforward check from the condition that $\mathcal D^{\mathrm{Lang}}_{[a]_{\rho}}(\mathfrak m)=\infty$.
\end{proof}

\begin{lemma} \label{lem higher index refomrulate}
 Use the notations in Lemma \ref{lem exists l}. If $a_{p_1, q_1}=a+1$, then $p_1=y_s>x_r$.
\end{lemma}

\begin{proof}
This follows from Lemma \ref{lem later add free}.
\end{proof}

\begin{lemma} \label{lem two segment in largest x}
 Use the notations in Lemma \ref{lem exists l}. There exists a segment in $\underline{\mathfrak{Ds}}(\mathfrak m_{x_r})$, but not in $\mathfrak m[a]$ (i.e. not with the starting point $\nu^a\rho$). In particular, there exist at least two segments in $\underline{\mathfrak{Ds}}(\mathfrak m_{x_r})$.
\end{lemma}

\begin{proof}
We consider the following set:
\[  \mathfrak n=\left\{ \Delta \in \mathfrak m[a+1] :  \Delta_{x_r}' \prec  \Delta \right\} .
\]
Suppose $\mathfrak{Ds}(\mathfrak m_{x_r})$ has only one segment. Then all those segments in $\mathfrak n$ must appear in $\underline{\mathfrak{Ds}}(\mathfrak m_p)$ for some $p\leq x_r$. However, one then uses Lemma \ref{lem segment stucture} to deduce that it is impossible to have $\mathcal D^{\mathrm{Lang}}_{[a]_{\rho}}(\mathfrak m)=\infty$ from Algorithm \ref{alg:rho_der:Lang}.
\end{proof}

\begin{lemma} \label{lem a+1 linked case} 
 Use the notations in Lemma \ref{lem exists l}. If $a_{p_{1}, q_{1}}=a+1$, then $[a+1, b_{p_2,q_2}]_{\rho}$ is linked to $\Delta'_{x_r}$. 
\end{lemma}

\begin{proof}
By Lemma \ref{lem two segment in largest x}, we may and shall consider the second last segment $\widetilde{\Delta}=\Delta_{x_r,r_{x_r}-1}$ picked in the upward sequence. By Lemma \ref{lem higher index refomrulate}, the segment $\widetilde{\Delta}$ gives a possible choice on the extension process. If it is not chosen, one has to choose a segment $[a_{p_2,q_2}, b_{p_2,q_2}]_{\rho}$ such that $p_2 >y_s$. Now Lemmas \ref{lem either prec or overlap} and \ref{lem refined relation in alg} boil down to three possibilities: (1) $\widetilde{\Delta} \subset [a_{p_2,q_2} , b_{p_2, q_2}]_{\rho}$; (2) $\widetilde{\Delta} \prec [a_{p_2,q_2}, b_{p_2,q_2}]_{\rho}$; or (3) $[a_{p_2,q_2}, b_{p_2,q_2}]_{\rho}\prec \widetilde{\Delta}$. However, the last one (3) is not possible from the choices of segments in $\underline{\mathfrak{Ds}}(\mathfrak m_{x_r})$. In the former two case, we must then have $[a+1, b_{p_2, q_2}]_{\rho}$ is linked to $\Delta'_{x_r}$ since $\widetilde{\Delta}$ is linked to $\Delta'_{x_r}$.
\end{proof}

\begin{lemma} \label{lem latter part matching}
Let $p >y_s$. If $\underline{\mathfrak{Ds}}(\mathfrak m_p)[a]\neq \emptyset$ (i.e. $a_{p,r_p}=a$), then $\underline{\mathfrak{Ds}}(\mathfrak m_p)[a+1] \neq \emptyset$ (i.e. $a_{p,r_p-1}=a+1$).  Moreover, the unique segment $[a_{p,r_p-1}, b_{p, r_{p}-1}]_{\rho}$ in $\underline{\mathfrak{Ds}}(\mathfrak m_p)[a+1]$ is not linked to $\Delta''_{y_s}$.  
\end{lemma}

\begin{proof}
The first assertion is a reformulation of Lemma \ref{lem later add free}, and the second assertion follows from the choices of segments in the algorithm.
\end{proof}

We shall consider Lemma \ref{lem latter part matching 2} in the case that $a_{p_1,r_{p_1}}= a+1$.

\begin{lemma} \label{lem latter part matching 2}
We use the same notations in Lemma \ref{lem latter part matching}. Let $p<p'<y_s$. Then the segment $[a_{p,r_p-1}, b_{p,r_p-1}]_{\rho}$ in $\underline{\mathfrak{Ds}}(\mathfrak m_p)[a+1]$ is not linked to $\Delta_{p',r_{p'}}$. 
\end{lemma}
We shall consider Lemma \ref{lem latter part matching 2} in the case that $a_{p_1,r_{p_1}}\neq a+1$ and $p'=p_1$ later.

\begin{lemma} \label{lem replace choice}
\begin{enumerate}
\item  If $a_{p_{1}, q_{1}}=a+1$, then $[a+1, b_{p_2, q_2}]_{\rho} \subset \Delta''_{y_s}$.  
\item If $a_{p_{1}, q_{1}}>a+1$, then $\Delta''_{y_S}\subset [a+1, b_{p_1, q_1}]_{\rho}$.
\end{enumerate}
\end{lemma}

\begin{proof}
We briefly explain this. For the first bullet, we must have $p_{2}<q_s$. Then from the algorithm, one sees that $\Delta_{q_s}''$ cannot be linked to $[a_{p_2,q_2}, b_{p_2, q_2}]_{\rho}$ and so we must have $[a_{i_2,j_2}, b_{p_2, q_2}]_{\rho}$ by Lemma \ref{lem either prec or overlap}.. For the second bullet, we must have $p_{1}> y_s$. Then form the algorithm, one sees that $[a_{p_1, q_1}, b_{p_1, q_1}]_{\rho}$ cannot be a subset of $\Delta_{y_s}''$ and so we must have $\Delta_{y_s}''\prec [a_{p_1, q_1}, b_{p_1, q_1}]_{\rho}$. 
\end{proof}

Now, one applies Lemmas \ref{lem a+1 linked case}, \ref{lem latter part matching}, and \ref{lem replace choice} to find the segments in the removal steps for the $\mathfrak{tus}(\mathcal I^{\mathrm{Lang}}_{[a+1,b]_{\rho}}, a)$. One sees that if $a_{p_1, q_1}=a+1$ (resp. $a_{p_1, q_1}>a+1$), the remaining segments left in $\mathcal I^{\mathrm{Lang}}_{[a+1,b]_{\rho}}(\mathfrak m)$ after the removal steps for  $\mathfrak{tus}(\mathcal I^{\mathrm{Lang}}_{[a+1,b]_{\rho}}, a)$ contain $\Delta''_{y_s}$ (resp.$[a+1, b_{p_1, q_1}]_{\rho}$). Now one applies Lemma \ref{lem replace choice} to obtain that
\[  \mathcal I^{\mathrm{Lang}}_{[a]_{\rho}}\circ \mathcal I^{\mathrm{Lang}}_{[a+1,b]_{\rho}}(\mathfrak m)=\mathcal I^{\mathrm{Lang}}_{[a,b]_{\rho}}(\mathfrak m).
\]
This completes the proof of Lemma \ref{lem:Lang_I[ab]=I[a]_I[a+1,b]}.

\begin{example}
\begin{enumerate}
\item Let $\mathfrak m=\left\{ [1,5]_{\rho}, [4,7]_{\rho}, [2,9]_{\rho}, [3,8]_{\rho}\right\}$ with $a=1$ and $b=3$. This is a case of $a_{i_1,j_1}=a+1=2$ in above discussions. For $\mathcal I^{\mathrm{Lang}}_{[a,b]_{\rho}}(\mathfrak m)$, 
\[ \underline{\mathfrak{Ds}}(\mathfrak m_1)=\left\{ [4,7]_{\rho}, [1,5]_{\rho} \right\}, \quad \underline{\mathfrak{Ds}}(\mathfrak m_2)=\left\{ [3,8]_{\rho}\right\}, \quad \underline{\mathfrak{Ds}}(\mathfrak m_3)=\left\{ [2,9]_{\rho}\right\} .
\]
Note that $\mathcal I^{\mathrm{Lang}}_{[2,3]_{\rho}}(\mathfrak m)=\left\{ [1,5]_{\rho}, [3,7]_{\rho}, [2,9]_{\rho}, [2,8]_{\rho} \right\}$. The removal step in the $\mathfrak{tus}(\mathcal I^{\mathrm{Lang}}_{[2,3]_{\rho}}(\mathfrak m), 1)$-process takes away $[1,5]_{\rho}$ and so 
\[ \mathcal I^{\mathrm{Lang}}_{[1]_{\rho}}\circ \mathcal I^{\mathrm{Lang}}_{[2,3]_{\rho}}(\mathfrak m)=\left\{[1,5]_{\rho}, [3,7]_{\rho}, [1,9]_{\rho}, [2,8]_{\rho}\right\} .\]
\item Let $\mathfrak m=\left\{ [1,5]_{\rho}, [1,14]_{\rho}, [2,9]_{\rho}, [2,15]_{\rho}, [3,13]_{\rho}, [4,7]_{\rho}, [4,12]_{\rho}  \right\}$ with $a=1$ and $b=3$. This is a case $a_{i_1,j_1}>a+1=2$. For $\mathcal I^{\mathrm{Lang}}_{[a,b]_{\rho}}(\mathfrak m)$, 
\[ \underline{\mathfrak{Ds}}(\mathfrak m_1)=\left\{ [4,7]_{\rho}, [1,5]_{\rho}\right\}, \quad \underline{\mathfrak{Ds}}(\mathfrak m_2)=\left\{ [4,12]_{\rho} , [2,9]_{\rho} \right\}, \]
\[ \underline{\mathfrak{Ds}}(\mathfrak m_3)=\left\{ [3,13]_{\rho}\right\}, \quad  \underline{\mathfrak{Ds}}(\mathfrak m_4)=\left\{ [1,14]_{\rho}, [2,15]_{\rho}\right\} .
\]
Note that $\mathcal I^{\mathrm{Lang}}_{[2,3]_{\rho}}(\mathfrak m)=\left\{[1,5]_{\rho}, [1,14]_{\rho}, [2,9]_{\rho}, [2,15]_{\rho}, [2,13]_{\rho}, [4,7]_{\rho}, [3,12]_{\rho}  \right\}$. The removal step in the $\mathfrak{tus}(\mathcal I^{\mathrm{Lang}}_{[2,3]_{\rho}}(\mathfrak m), 1)$-process takes away the segments $[1,5]_{\rho}, [2,9]_{\rho}, [1,14]_{\rho}, [2,15]_{\rho}$, and so
\[  \mathcal I^{\mathrm{Lang}}_{[1]_{\rho}}\circ \mathcal I^{\mathrm{Lang}}_{[2,3]_{\rho}}(\mathfrak m)=  \left\{[1,5]_{\rho}, [1,14]_{\rho}, [2,9]_{\rho}, [2,15]_{\rho}, [1,13]_{\rho}, [4,7]_{\rho}, [3,12]_{\rho}  \right\}.
\]
\end{enumerate}
\end{example}

\subsection{Composition of $\mathrm I^{\mathrm{R}}_{[a+1,b]_{\rho}}$ and $\mathrm I^{\mathrm{R}}_{[a]_{\rho}}$}

\begin{lemma}\label{lem:I[ab]=I[a]_I[a+1,b]_pi}
	Let $\pi \in \mathrm{Irr}_\rho$ and $[a,b]_\rho \in \mathrm{Seg}_\rho$. If  $\varepsilon^\mathrm{R}_{[a]_\rho}(\pi) = 0$, we then have
	\[\mathrm{I}^\mathrm{R}_{[a,b]_\rho}(\pi) \cong \mathrm{I}^\mathrm{R}_{[a]_\rho}\circ \mathrm{I}^\mathrm{R}_{[a+1,b]_\rho}(\pi).\]
\end{lemma}
\begin{proof}
	Take $\tau=\mathrm{I}^\mathrm{R}_{[a,b]_\rho}(\pi)$. Then, $\mathrm{D}^\mathrm{R}_{[a,b]_\rho}(\tau)\cong \pi\neq 0$. As $\varepsilon^\mathrm{R}_{[a]_\rho}(\pi) = 0$, we have $\varepsilon^\mathrm{R}_{[a]_\rho}(\tau)=1$. By Lemma \ref{lem:D[ab]=D[a+1,b]_D[a]_pi}, we get $\mathrm{D}^\mathrm{R}_{[a+1,b]_\rho} \circ \mathrm{D}_{[a]_\rho}^\mathrm{R}(\tau)  \cong   \mathrm{D}_{[a,b]_\rho}^\mathrm{R}(\tau).$ Hence the result follows.
    \end{proof}

\subsection{Main result}

\begin{theorem}\label{thm:integral_Lang}
	Let $\Delta \in \mathrm{Seg}_\rho$ and $\mathfrak{m} \in \mathrm{Mult}_\rho$. Then, 
	$\mathrm{I}^\mathrm{R}_\Delta(L(\mathfrak{m})) \cong  L \left(\mathcal{I}^\mathrm{Lang}_\Delta(\mathfrak{m})\right).$
\end{theorem}
\begin{proof}
We use induction argument on the length $\ell_{rel}(\Delta)$ of $\Delta=[a,b]_\rho$, and the length $\ell_{rel}(\mathfrak{m})$ of $\mathfrak{m}$ to give a proof of the theorem. By Lemma \ref{lem reduction to m1}, we may assume $\mathfrak m$ is in good range for $\Delta$. By Proposition \ref{int:Lang:length=1}, for $\ell_{rel}(\Delta)=1$ and for any $\mathfrak{m}^\prime \in \mathrm{Mult}_\rho$, we have 
	\begin{equation}\label{Eq:int_Lang_A}
		\mathrm{I}^\mathrm{R}_{[a]_\rho}(L(\mathfrak{m}^\prime)) \cong  L \left(\mathcal{I}^\mathrm{Lang}_{[a]_\rho}(\mathfrak{m}^\prime)\right).   
	\end{equation} 
    This also serves as a basic case.
	
	Case 1. Let $\mathcal{D}^\mathrm{Lang}_{[a]_\rho}(\mathfrak{m}) \neq \infty$. As an inductive step, we assume that 
	\begin{equation}\label{Eq:int_Lang_B}
		\mathrm{I}^\mathrm{R}_{[a,b]_\rho}(L(\mathfrak{n})) \cong  L \left(\mathcal{I}^\mathrm{Lang}_{[a,b]_\rho}(\mathfrak{n})\right).  
	\end{equation}
	for any $\mathfrak{n} \in \mathrm{Mult}_\rho$ with $\ell_{rel}(\mathfrak{n}) < \ell_{rel}(\mathfrak{m})$. Then,
	\begin{align*}
		\mathcal{I}^\mathrm{Lang}_{[a,b]_\rho}(\mathfrak{m})   &= \mathcal{I}^\mathrm{Lang}_{[a,b]_\rho}\circ \mathcal{I}^\mathrm{Lang}_{[a]_\rho} \circ \mathcal{D}^\mathrm{Lang}_{[a]_\rho}(\mathfrak{m})\\
		&= \mathcal{I}^\mathrm{Lang}_{[a]_\rho}\circ \mathcal{I}^\mathrm{Lang}_{[a,b]_\rho} \circ \mathcal{D}^\mathrm{Lang}_{[a]_\rho}(\mathfrak{m}) \quad (\text{by Lemma }\ref{lem:Lang_comm_I[ab]_I[a]}).
	\end{align*}
	Therefore, we conclude that
	\begin{align*}
		L\left(\mathcal{I}^\mathrm{Lang}_{[a,b]_\rho}(\mathfrak{m}) \right) &= L\left(\mathcal{I}^\mathrm{Lang}_{[a]_\rho}\circ \mathcal{I}^\mathrm{Lang}_{[a,b]_\rho} \circ \mathcal{D}^\mathrm{Lang}_{[a]_\rho}(\mathfrak{m})\right) \\
		& \cong  \mathrm{I}^\mathrm{R}_{[a]_\rho} \left( L\left( \mathcal{I}^\mathrm{Lang}_{[a,b]_\rho} \circ \mathcal{D}^\mathrm{Lang}_{[a]_\rho}(\mathfrak{m})\right)\right) \quad (\text{by }\eqref{Eq:int_Lang_A})\\
		&\cong  \mathrm{I}^\mathrm{R}_{[a]_\rho} \circ \mathrm{I}^\mathrm{R}_{[a,b]_\rho} \left( L\left(  \mathcal{D}^\mathrm{Lang}_{[a]_\rho}(\mathfrak{m})\right)\right) \quad (\text{by }\eqref{Eq:int_Lang_B})\\
		&\cong  \mathrm{I}^\mathrm{R}_{[a]_\rho} \circ \mathrm{I}^\mathrm{R}_{[a,b]_\rho} \circ  \mathrm{D}^\mathrm{R}_{[a]_\rho}\left( L\left(\mathfrak{m}\right)\right) \quad (\text{by $\rho$-derivative})\\
		&\cong  \mathrm{I}^\mathrm{R}_{[a,b]_\rho} \circ \mathrm{I}^\mathrm{R}_{[a]_\rho} \circ  \mathrm{D}^\mathrm{R}_{[a]_\rho}\left( L\left(\mathfrak{m}\right)\right) \quad (\text{by Lemma } \ref{lem:comm_I[ab]_I[a]_pi})\\
		&\cong  \mathrm{I}^\mathrm{R}_{[a,b]_\rho}\left( L\left(\mathfrak{m}\right)\right).
	\end{align*}
	
	Case 2. Let $\mathcal{D}^\mathrm{Lang}_{[a]_\rho}(\mathfrak{m}) = \infty$. As an inductive step, we assume that 
	\begin{equation}\label{Eq:int_Lang_C}
		\mathrm{I}^\mathrm{R}_{[a+1,b]_\rho}(L(\mathfrak{m})) \cong  L \left(\mathcal{I}^\mathrm{Lang}_{[a+1,b]_\rho}(\mathfrak{m})\right). 
	\end{equation}
	Therefore, using Lemma \ref{lem:Lang_I[ab]=I[a]_I[a+1,b]}, we have
	\begin{align*}
		L\left(\mathcal{I}^\mathrm{Lang}_{[a,b]_\rho}(\mathfrak{m}) \right) & \cong  L\left(\mathcal{I}^\mathrm{Lang}_{[a]_\rho}\circ \mathcal{I}^\mathrm{Lang}_{[a+1,b]_\rho} (\mathfrak{m})\right) \\
		& \cong  \mathrm{I}^\mathrm{R}_{[a]_\rho} \left( L\left( \mathcal{I}^\mathrm{Lang}_{[a+1,b]_\rho} (\mathfrak{m})\right)\right) \quad (\text{by }\eqref{Eq:int_Lang_A})\\
		& \cong  \mathrm{I}^\mathrm{R}_{[a]_\rho} \circ \mathrm{I}^\mathrm{R}_{[a+1,b]_\rho} \left( L\left( \mathfrak{m}\right)\right) \quad (\text{by }\eqref{Eq:int_Lang_C})\\
		& \cong \mathrm{I}^\mathrm{R}_{[a,b]_\rho} \left( L\left(\mathfrak{m}\right)\right) \quad (\text{by Lemma } \ref{lem:I[ab]=I[a]_I[a+1,b]_pi}).
	\end{align*}
\end{proof}

\subsection{Left integral algorithm}

The following theorem follows from Theorem \ref{thm:integral_Lang} and Section \ref{sss gk invol}:

\begin{theorem}
For $\mathfrak m \in \mathrm{Mult}_{\rho}$ and $[a,b]_\rho \in \mathrm{Seg}_\rho$, we define \[\mathcal I^{\mathrm{Lang}, \mathrm L}_{[a,b]_{\rho}}(\mathfrak m)=\Theta\left(\mathcal I^{\mathrm{Lang}}_{[-b,-a]_{\rho^{\vee}}}(\Theta(\mathfrak m))\right).\] Then,
 \[\mathrm{I}^\mathrm{L}_{[a,b]_{\rho}} \left(L(\mathfrak{m}) \right) \cong L \left( \mathcal{I}_{[a,b]_{\rho}}^{\mathrm{Lang}, \mathrm{L}}(\mathfrak{m})  \right).\]
\end{theorem}

\section{Integral in Zelevinsky classification} \label{sec:int_Zel}

%Let $\mathfrak{m} \in \mathrm{Mult}_\rho$ and $\Delta \in \mathrm{Seg}_\rho$. We define a multisegment $\mathcal{I}^\mathrm{Zel}_{\Delta}(\mathfrak{m})$ by the following algorithm (cf. Algorithm \ref{alg:int:Zel}) such that the right integral $\mathrm{I}^\mathrm{R}_\Delta(Z(\mathfrak{m}))$ is given by $Z \left(\mathcal{I}^\mathrm{Zel}_{\Delta}(\mathfrak{m}) \right)$ (cf. Theorem \ref{thm:integral_Zel}). 

In this section, we present an algorithm for computing $\mathrm I^{\mathrm R}_{\Delta}(Z(\mathfrak m))$. We shall continue the approach of using the MW algorithm from Section \ref{sec:der_Zel}. An alternate approach is to use reduction similar to Section \ref{sec:der_Lang}, while the details require some lengthy routine checking, and so we shall not provide details.

\subsection{Algorithm for integrals}
\begin{algorithm}\label{alg:int:Zel}
Let $\mathfrak{m} \in \mathrm{Mult}_\rho$ and $[a,b]_\rho \in \mathrm{Seg}_\rho$. Set $\mathfrak{m}_0=\mathfrak{m}$ to apply the following steps: 

{Step 1}. (Choose a downward sequence of minimal linked segments): Define the downward sequence of minimal linked segments in neighbors on $\mathfrak{m}_0$ ranging from $b$ to $a-1$ as follows: start with the shortest segment $\Delta^{b}_1$ (if it exists) in $\mathfrak{m}_0\left\langle b \right \rangle$. Recursively for $b-1 \geq i \geq a-1$,  we choose the shortest segment $\Delta^{i}_1$ (if it exists) in $\mathfrak{m}_0\left\langle i \right \rangle$ such that $\Delta^{i}_1 \prec \Delta^{i+1}_1$, and set $\Delta^{i}_1=\emptyset$ if it does not exist. Then the sequence $\Delta^{a-1}_1 \prec \cdots \prec \Delta^{b}_1$ defines a downward sequence of minimal linked segments in neighbors on $\mathfrak{m}_0$ ranging from $b$ to $a-1$. 

Step 2. (Remove and replace): We replace $\mathfrak{m}_0$ by $\mathfrak{m}_1$ defined by
	\[\mathfrak{m}_1:= \mathfrak{m}_0-\sum\limits_{i=a-1}^b {\Delta}^{i}_1 .\]

Step 3. (Repeat Step 1 and 2): Again find (if it exists say $\Delta^{a-1}_2 \prec \cdots \prec \Delta^{b}_2$) the downward sequence of minimal linked segments in neighbors on $\mathfrak{m}_1$ ranging from $b$ to $a-1$ and replace $\mathfrak{m}_1$ by 
	\[\mathfrak{m}_2:= \mathfrak{m}_1-\sum\limits_{i=a-1}^b {\Delta}^{i}_2 .\]
Repeat this removal process until it terminates after a finite number of times, say $k$ times, and there does not exist any downward sequence of minimal linked segments in neighbors on $\mathfrak{m}_k$ ranging from $b$ to $a-1$.
	
Step 4. (Upward sequence of maximal linked segments): If $\mathfrak{m}_k\left\langle a-1 \right \rangle \neq \emptyset$, we choose the maximal length segment $\widetilde{\Delta}_{a-1} \in \mathfrak{m}_k\left\langle a-1 \right \rangle$. Otherwise, we set $\widetilde{\Delta}_{a-1} = \emptyset$, the void segment. Recursively for $a \leq i \leq b-1$,  we choose the maximal segment $\widetilde{\Delta}_{i} \in \mathfrak{m}_k\left\langle i \right \rangle$ (if it exists) such that $\widetilde{\Delta}_{i-1} \prec \widetilde{\Delta}_{i}$. Otherwise, we set $\widetilde{\Delta}_{i}=\emptyset$. 

Step 5. (Extension): Finally, we define the right integral multisegment by
    \begin{equation}\label{eq:int_Zel}
   \mathcal{I}^\mathrm{Zel}_{[a,b]_\rho}(\mathfrak{m}) := \mathfrak{m} - \sum\limits_{i=a-1}^{b-1} \widetilde{\Delta}_{i} + \sum\limits_{i=a-1}^{b-1} \left(\widetilde{\Delta}_{i}\right)^+,     
    \end{equation}
	where, we set $\left(\widetilde{\Delta}_{i}\right)^+=[i+1,i+1]_\rho= \{\nu^{i+1}\rho\}$ if $\widetilde{\Delta}_{i}=\emptyset$. 
\end{algorithm}

\begin{example}
Let $\mathfrak{m}=\left\{ [0,2]_\rho, [0,1]_\rho, [0,1]_\rho, [1,2]_\rho, [1,1]_\rho, [2,3]_\rho \right\}$ and $\Delta=[2,3]_\rho$. Then, $\mathfrak{m}_1=\mathfrak{m}-[2,3]_\rho-[1,2]_\rho-[0,1]_\rho$. Since $\mathfrak{m}_1\left\langle 3 \right \rangle=\emptyset$, there is no removable downward sequence of minimal linked segments in neighbors on $\mathfrak{m}_1$ ranging from $3$ to $1$. We have $\widetilde{\Delta}_1=[0,1]_\rho$ and $\widetilde{\Delta}_2=\emptyset$. Therefore, $\mathcal{I}^\mathrm{Zel}_\Delta(\mathfrak{m}) = \mathfrak{m}-[0,1]_\rho + [0,2]_\rho+[3,3]_\rho$.
\end{example}

\begin{remark} \label{rmk zel alg mw}
One may consider the algorithm here as an effective version of \cite[Proposition 5.1]{LM16}, which does not involve the direct use of the M\oe glin-Waldspurger algorithm.
\end{remark}

\subsection{MW algorithm and Integral algorithm}
\begin{lemma} \label{lem redu integral }
Let $\mathfrak m \in \mathrm{Mult}_{\rho}$. Let $\nu^c\rho$ be the maximal cuspidal support for $\mathfrak m$. Fix an integer $b\leq c$. Suppose that there is no downward sequence of minimally linked segments in neighbors on $\mathfrak m$ ranging from $c$ to $b-1$. Let $\widetilde{\Delta}_{b-1}, \ldots, \widetilde{\Delta}_{c-1}$ be the (possibly void) segments participating in the extension process for $\mathcal I^{\mathrm{Zel}}_{[b,c]_{\rho}}(\mathfrak m)$ i.e., $\mathcal I^{\mathrm{Zel}}_{[b,c]_{\rho}}(\mathfrak m)=\mathfrak m-\sum\limits_{i=b-1}^{c-1}\widetilde{\Delta}_{i} + \sum\limits_{i=b-1}^{c-1}\widetilde{\Delta}_{i}^{+}$. Then the segments participating in the MW algorithm for $\mathcal I^{\mathrm{Zel}}_{[b,c]_{\rho}}(\mathfrak m)$ are  $\widetilde{\Delta}_{b-1}^+,\ldots, \widetilde{\Delta}_{c-1}^+$.
\end{lemma}

\begin{proof}
Note that the condition on downward sequences guarantee that $\widetilde{\Delta}_{b-1}$ is the longest segment in $\mathfrak m\langle b-1 \rangle$, and for each $k=b,\ldots, c-1$, $\widetilde{\Delta}_k$ is the longest segment in $\mathfrak m\langle k \rangle$ linked to $\widetilde{\Delta}_{k-1}$. If a segment $\Delta_c \in \mathfrak m\langle c \rangle$ is shorter than $\widetilde{\Delta}_{c-1}^+$, the sequence $\Delta_c, \widetilde{\Delta}_{c-1},\ldots , \widetilde{\Delta}_{b-1}$ produces a downward sequence of linked segments in neighbors on $\mathfrak{m}$ ranging from $c$ to $b-1$, that contradicts the given condition. Also, for $c-1 \geq k \geq b$, we cannot find a segment $\Delta$ in $\mathfrak m\langle k \rangle$ such that $\Delta \subset \widetilde{\Delta}_{k-1}^+$, and $\Delta \prec \widetilde{\Delta}_k^+$. Then, one can use this to show the lemma. 
\end{proof}

\begin{lemma} \label{lem MW on integral}
Let $\mathfrak m \in \mathrm{Mult}_{\rho}$. Let $\Delta(\mathfrak{m})=[a,c]_\rho$ be the first segment produced by MW algorithm on $\mathfrak m$. Let $r=\varepsilon^{\mathrm{MW}}_{[a,c]_{\rho}}(\mathfrak m)+\ldots +\varepsilon^{\mathrm{MW}}_{[b-1,c]_{\rho}}(\mathfrak m)$ for some $a \leq b \leq c$, where $r=0$ if $a=b$. Then, 
\[  \left(\mathcal{D}^{\mathrm{MW}}\right)^{r+1} \left(\mathcal I^{\mathrm{Zel}}_{[b,c]_{\rho}}(\mathfrak m)\right)=\left(\mathcal{D}^{\mathrm{MW}}\right)^{r}(\mathfrak m).
\]
Moreover, if $a'\leq b-1$, $\varepsilon^{\mathrm{MW}}_{[a',c]_{\rho}}(\mathcal I^{\mathrm{Zel}}_{[b,c]_{\rho}}(\mathfrak m))=\varepsilon_{[a',c]_{\rho}}^{\mathrm{MW}}(\mathfrak m)$; and  the first segment produced by the MW-algorithm for $(\mathcal D^{\mathrm{MW}})^r\circ \mathcal I^{\mathrm{Zel}}_{[b,c]_{\rho}}(\mathfrak m) $ is $[b,c]_{\rho}$.
\end{lemma}

\begin{proof}
We use the notations in Algorithm \ref{alg:int:Zel} applied to $\mathfrak m$. Comparing with the MW algorithm, we obtain $r$ downward sequences of minimally linked segments:
\[   \Delta_1^c,\ldots, \Delta_1^{b-1}; \ldots; \Delta_r^c, \ldots, \Delta_r^{b-1} .
\]
Using the notation in Algorithm \ref{alg:int:Zel}, we also have the segments $\widetilde{\Delta}_{c-1}, \ldots, \widetilde{\Delta}_{b-1}$ (possibly empty) in $\mathfrak m-\sum\limits_{i=1}^r\sum\limits_{k=b-1}^c\Delta_i^k$ participating in the extension process for $\mathcal I^{\mathrm{Zel}}_{[b,c]_{\rho}}(\mathfrak m)$. 

Now, for $b-1 \leq k \leq c$, let $MW_k(\mathfrak m)=\Delta_1^k+\ldots+\Delta_r^k$. For $k=b-1,\ldots, c-1$, Lemma  \ref{lem structure of segments from MW} implies that $MW_k(\mathfrak m)$ and $MW_{k+1}(\mathfrak m)$ are minimally linked in $\mathfrak m$, and so, they are also minimally linked in $\mathfrak m-\sum\limits_{k=b-1}^{c-1}\widetilde{\Delta}_k$. On the other hand, by Lemma \ref{lem redu integral }, for $k=b, \ldots, c-1$, $\widetilde{\Delta}_k^+$ and $\widetilde{\Delta}_{k+1}^+$ are minimally linked in $\mathcal I^{\mathrm{Zel}}_{[b,c]_{\rho}}(\mathfrak m)-\sum_{k=b-1}^cMW_k(\mathfrak m)$ and no segment in $(\mathcal I^{\mathrm{Zel}}_{[b,c]_{\rho}}(\mathfrak m)-\sum\limits_{k=b-1}^c MW_k(\mathfrak m))\langle b-1 \rangle$ is linked to $\widetilde{\Delta}_{b-1}^+$. Now by Lemma \ref{lem combine minimal}, for $b \leq k \leq c-1$, $MW_k(\mathfrak m)+\widetilde{\Delta}_{k-1}^+$ and $MW_{k+1}(\mathfrak m)+\widetilde{\Delta}_k^+$ are minimally linked in 
\[ \mathcal I^{\mathrm{Zel}}_{[b,c]_{\rho}}(\mathfrak m)=\left(\mathfrak m-\sum_{k=b-1}^{c-1}\widetilde{\Delta}_k-\sum_{k=b-1}^c WM_k(\mathfrak m)\right)+\sum_{k=b-1}^c WM_k(\mathfrak m)+\sum_k \widetilde{\Delta}_{k-1}^+ ,
\]
and $MW_{b-1}(\mathfrak m)$ and $MW_b(\mathfrak m)+\widetilde{\Delta}_{b-1}^+$ are also minimally linked in $\mathcal I^{\mathrm{Zel}}_{[b,c]_{\rho}}(\mathfrak m)$.

In order to find all the segments participating in the MW algorithm for \[\mathfrak m, \mathcal D^{\mathrm{MW}}(\mathfrak m), \ldots, (\mathcal D^{\mathrm{MW}})^r(\mathfrak m), \]
 we, for $k=b-2, \ldots, a$, recursively find the submultisegments $WM_k(\mathfrak m) \subseteq \mathfrak m \langle k\rangle$ such that $WM_k(\mathfrak m)$ and $WM_{k+1}(\mathfrak m)$ are minimally linked in $\mathfrak m$. For $k=b-2,\ldots, a$, as $\mathcal I^{\mathrm{Zel}}_{[b,c]_{\rho}}(\mathfrak m)\langle k \rangle =\mathfrak m \langle k\rangle$, we have $WM_k(\mathfrak m)$ and $WM_{k+1}(\mathfrak m)$  are also minimally linked in $\mathcal I^{\mathrm{Zel}}_{[b,c]_{\rho}}(\mathfrak m)$.\\ 
 Now, by using Lemma \ref{lem structure of segments from MW}(i) and above discussions, one has: $\left(\mathcal D^{\mathrm{MW}}\right)^{r+1}\left(\mathcal I^{\mathrm{Zel}}_{[b,c]_{\rho}}(\mathfrak m)\right)=$
\begin{align*} 
& \mathcal I_{[b,c]_{\rho}}^{\mathrm{Zel}}(\mathfrak m)-\sum_{k=a}^{b-1} MW_k(\mathfrak m)-\sum_{k=b}^c (MW_k(\mathfrak m)+\widetilde{\Delta}_{k-1}^+)+\sum_{k=a}^{b-1} MW_k(\mathfrak m)^- +\sum_{k=b}^c(MW_k(\mathfrak m)+\widetilde{\Delta}_{k-1}^+)^-  \\
&= \mathcal I_{[b,c]_{\rho}}^{\mathrm{Zel}}(\mathfrak m)-\sum_{k=a}^c MW_k(\mathfrak m)-\sum_{k=b}^c\widetilde{\Delta}_{k-1}^++\sum_{k=a}^cMW_k(\mathfrak m)^-+\sum_{k=b}^c \left(\widetilde{\Delta}_{k-1}^+ \right)^- \\
&=\mathfrak m-\sum_{k=a}^c MW_k(\mathfrak m)+\sum_{k=a}^c MW_k(\mathfrak m)^-=\left(\mathcal D^{\mathrm{MW}}\right)^r(\mathfrak m).
\end{align*}
The assertion for $\varepsilon^{\mathrm{MW}}_{[a^\prime,c]_{\rho}}$ follows from its definition, the segments participating in the MW-algorithms above, and Lemma \ref{lem structure of segments from MW}(i).
\end{proof}

\subsection{Main result}

\begin{theorem}\label{thm:integral_Zel}
Let $\Delta \in \mathrm{Seg}_\rho$ and $\mathfrak{m} \in \mathrm{Mult}_\rho$. Then, $\mathrm{I}^\mathrm{R}_\Delta(Z(\mathfrak{m})) \cong Z \left(\mathcal{I}^\mathrm{Zel}_\Delta(\mathfrak{m})\right).$
\end{theorem}
\begin{proof}
Let $\Delta=[b,c]_{\rho}$. Let $\mathfrak n=\mathfrak m^{\leq c}$. Let $\Delta(\mathfrak{n})=[a,c]_{\rho}$ be the first segment produced by the MW-algorithm for $\mathfrak n$. Let $k_x=\varepsilon^{\mathrm{MW}}_{[x,c]_{\rho}}(\mathfrak n)$ for $a \leq x \leq b-1$, and set $r= k_{a}+\ldots +k_{b-1}$. Then,
\begin{align*}
 (\mathrm D^{\mathrm R}_{[b-1,c]_{\rho}})^{k_{b-1}}\circ \ldots \circ  (\mathrm D^{\mathrm R}_{[a,c]_{\rho}})^{k_{a}}(Z(\mathfrak n)) 
=& Z((\mathcal D^{\mathrm{MW}})^r(\mathfrak n)) \\
=& Z((\mathcal D^{\mathrm{MW}})^{r+1}\circ \mathcal I^{\mathrm{Zel}}_{[b,c]_{\rho}}(\mathfrak n)) \\
=& \mathrm D^{\mathrm R}_{\Delta}\circ(\mathrm D^{\mathrm R}_{[b-1,c]_{\rho}})^{k_{b-1}}\circ \ldots \circ  (\mathrm D^{\mathrm R}_{[a,c]_{\rho}})^{k_{a}}(Z(\mathcal I^{\mathrm{Zel}}_{\Delta}(\mathfrak n)))
\end{align*}
where the first equality follows from Proposition \ref{prop:MW_der}, the second equality follows from the first assertion of Lemma \ref{lem MW on integral}, and the third one follows from the second assertion of Lemma \ref{lem MW on integral}.

Now, applying integrals on both sides and using the commutativity of derivatives, we can cancel to obtain:
\[  \mathrm D^\mathrm{R}_{[b,c]_{\rho}}\left(Z\left(\mathcal I^{\mathrm{Zel}}_{[b,c]_{\rho}}(\mathfrak n)\right)\right)= Z(\mathfrak n) .
\]
With Theorem \ref{thm alg der zel}, we have $Z(\mathcal D^{\mathrm{Zel}}_{[b,c]_{\rho}}\circ \mathcal I^{\mathrm{Zel}}_{[b,c]_{\rho}}(\mathfrak n))=Z(\mathfrak n)$ and so $\mathcal D^{\mathrm{Zel}}_{[b,c]_{\rho}}\circ \mathcal I^{\mathrm{Zel}}_{[b,c]_{\rho}}(\mathfrak n)=\mathfrak n$. Now, note that
\begin{align*}
\mathcal D^{\mathrm{Zel}}_{[b,c]_{\rho}}\circ \mathcal I^{\mathrm{Zel}}_{[b,c]_{\rho}}(\mathfrak m) =&\mathcal D^{\mathrm{Zel}}_{[b,c]_{\rho}}\circ \mathcal I^{\mathrm{Zel}}_{[b,c]_{\rho}}(\mathfrak n+\mathfrak m^{>c})=\mathcal D^{\mathrm{Zel}}_{[b,c]_{\rho}}\left(I^{\mathrm{Zel}}_{[b,c]_{\rho}}(\mathfrak n)+\mathfrak m^{>c}\right) \\
= &\mathcal D^{\mathrm{Zel}}_{[b,c]_{\rho}}\left(\mathcal I^{\mathrm{Zel}}_{[b,c]_{\rho}}(\mathfrak n)\right)+\mathfrak m^{>c}=\mathfrak{n}+\mathfrak m^{>c}=\mathfrak m .
\end{align*}
Thus, by Theorem \ref{thm alg der zel} again,
$Z(\mathfrak m)=\mathrm D^{\mathrm R}_{[b,c]_{\rho}}(Z(\mathcal I^{\mathrm{Zel}}_{[b,c]_{\rho}})) 
$. Now, the theorem follows by applying $\mathrm I^{\mathrm R}_{[b,c]_{\rho}}$ on both sides.
\end{proof}

\subsection{Algorithm for left integral}

We again have the left version:

\begin{theorem}
For $\mathfrak m \in \mathrm{Mult}_{\rho}$ and $[a,b]_\rho \in \mathrm{Seg}_\rho$, define $\mathcal I^{\mathrm{Zel}, \mathrm L}_{[a,b]_{\rho}}(\mathfrak m)=\Theta\left(\mathcal I^{\mathrm{Zel}}_{[-b,-a]_{\rho^{\vee}}}(\Theta(\mathfrak m))\right)$. Then,
 \[\mathrm{I}^\mathrm{L}_{[a,b]_{\rho}} \left(Z(\mathfrak{m}) \right) \cong Z \left( \mathcal{I}_{[a,b]_{\rho}}^{\mathrm{Lang}, \mathrm{L}}(\mathfrak{m})  \right).\]
\end{theorem}

\subsection{Exotic duality}

For completeness, we shall also establish an exotic duality analogous to Proposition \ref{prop dual derivative integral}. We use $\mathbb D_r$ defined in Section \ref{ss transfer exotic duality}.

\begin{proposition}
Let $\mathfrak m \in \mathrm{Mult}_{\rho}$ and $[a,b]_\rho \in \mathrm{Seg}_\rho$. Then, for sufficiently large $r$,
\begin{enumerate}
  \item[(i)] if $|\mathcal I^{\mathrm{Zel},\mathrm L}_{[a,b]_{\rho}}(\mathfrak m)|=|\mathfrak m|$, we have $\mathbb D_r(\mathcal I^{\mathrm{Zel}, \mathrm L}_{[a,b]_{\rho}}(\mathfrak m))=\mathcal D^{\mathrm{Zel}}_{[a,b]_{\rho}}(\mathbb D_r(\mathfrak m))$;
  \item[(ii)] $|\mathcal I^{\mathrm{Zel},\mathrm L}_{[a,b]_{\rho}}(\mathfrak m)|=|\mathfrak m|$ if and only if $\mathcal D^{\mathrm{Zel}}_{[a,b]_{\rho}}(\mathbb D_r(\mathfrak m)) \neq \infty$.
\end{enumerate}
\end{proposition}

A proof is similar to Proposition \ref{prop dual derivative integral}, and we omit the details.

\section{Applications} \label{s alg hd der}
\subsection{Algorithm for highest BZ derivative $\pi^-$ in Langlands classification}
To define an algorithm for computing the highest BZ derivative in the Langlands classification, we recall the notion of the removable free section $\mathfrak{rf}(\Delta)$ for a segment $\Delta$ in an upward sequence $\underline{\mathfrak{Us}}(\mathfrak{n})$ of maximally linked segments in $\mathfrak{n}\in \mathrm{Mult}_\rho$ as introduced in section \ref{subsec:der:lang}.

\begin{algorithm}\label{alg:hd:Lang}
 Let $\mathfrak m \in \mathrm{Mult}_{\rho}$. Set $\mathfrak{m}_1=\mathfrak m$ and apply the following steps:

Step 1. (Arrange the upward sequences) Let $\underline{\mathfrak{Us}}(\mathfrak{m}_1)=\Delta_{1,1} + \ldots + \Delta_{1,r_1}$. Recursively for $i \geq 2$, set 
\begin{equation}\label{eq:Us_m}
\mathfrak{m}_i=\mathfrak{m}_{i-1} - \underline{\mathfrak{Us}}(\mathfrak{m}_{i-1}) \text{ and }  \underline{\mathfrak{Us}}(\mathfrak{m}_i)=\Delta_{i,1} + \ldots + \Delta_{i,r_i},    
\end{equation}
where $\Delta_{i,j}=[x_{i,j}, y_{i,j}]_\rho$. This process terminates after a finite number,
 say $k$ times if $\underline{\mathfrak{Us}}(\mathfrak{m}_k)=\mathfrak{m}_k$. 

Step 2. (Remove the free section) Considering $\Delta_{i,j} \in \underline{\mathfrak{Us}}(\mathfrak{m}_i)$, we define the non-free section of $\Delta_{i,j}$ by
\[\mathfrak{nf}(\Delta_{i,j})=\begin{cases}
    \left[x_{i,j+1}-1,y_{i,j} \right]_\rho &\mbox{ if } j<r_i\\
    \emptyset &\mbox{ if } j=r_i.
\end{cases}\]
In other words, the disjoint union $\mathfrak{rf}(\Delta_{i,j}) \sqcup \mathfrak{nf}(\Delta_{i,j})=\Delta_{i,j}$.

Step 3. (Collect all non-free parts) Finally, we define, \[\mathfrak{d}({\mathfrak{m}})= \sum\limits_{i=1}^k \sum\limits_{j=1}^{r_i} \mathfrak{nf}(\Delta_{i,j}).\]
\end{algorithm}

We shall show that $\mathfrak{d}(\mathfrak m)$ gives the highest BZ derivative in Theorem \ref{thm:hd_Lang}. The key idea is to use a description of the highest BZ derivative in terms of derivatives of cuspidal representations, and then show the corresponding combinatorial counterpart in the following Lemma \ref{lem:hd_1}.

\begin{lemma}\label{lem:hd_1}
   Let $\mathfrak{m} \in \mathrm{Mult}_\rho$ and let $a$ (resp. $b$) be the smallest (resp.
 largest) integer such that $\nu^a\rho$ (resp. $\nu^b\rho$) lies in $\mathrm{csupp}(\mathfrak{m})$. Let $\varepsilon_a= \varepsilon^\mathrm{R}_{[a]_\rho}(L(\mathfrak{m}))$ and recursively, for each $t=a+1,\ldots,b$, let 
 \[\varepsilon_t=\varepsilon^\mathrm{R}_{[t]_\rho} \left( \left(\mathrm{D}^\mathrm{R}_{[t-1]_\rho}\right)^{\varepsilon_{t-1}} \circ \ldots \circ \left(\mathrm{D}^\mathrm{R}_{[a]_\rho}\right)^{\varepsilon_{a}} (L(\mathfrak{m}))\right).\]
 Then,
 \[\left(\mathcal{D}^\mathrm{Lang}_{[b]_\rho}\right)^{\varepsilon_{b}} \circ \ldots \circ \left(\mathcal{D}^\mathrm{Lang}_{[a]_\rho}\right)^{\varepsilon_{a}} (\mathfrak{m})=\mathfrak{d}(\mathfrak{m}).\]
\end{lemma}
\begin{proof}
Let's assume all the notations as mentioned in Algorithm \ref{alg:hd:Lang}. Then, by Algorithm \ref{alg:der:Lang}, we have
 \begin{equation}\label{eq:hd_a}
 \mathfrak{n}=\left(\mathcal{D}^\mathrm{Lang}_{[a]_\rho}\right)^{\varepsilon_{a}} (\mathfrak{m})=\mathfrak{m}-\sum\limits_{\substack{\Delta \in \mathfrak{m}[a]\\ [a]_\rho \in \mathfrak{rf}\left(\Delta\right)}} \Delta + \sum\limits_{\substack{\Delta \in \mathfrak{m}[a]\\ [a]_\rho \in \mathfrak{rf}\left(\Delta\right)}} {^-}\Delta,    
 \end{equation}
where $\mathfrak{rf}(\Delta)$ is defined in the system \eqref{eq:Us_m} of upward sequences $\underline{\mathfrak{Us}}(\mathfrak{m}_i)$ and the right hand side of \eqref{eq:hd_a} is obtained by deleting all removable free points $[a]_\rho$ from the segments of $\mathfrak{m}$ in that system.

Note that 
\[ \mathfrak n[a] =\{\mathfrak{nf}(\Delta) :\Delta \in \mathfrak{m}[a] \text{ and } \mathfrak{rf}(\Delta)=\emptyset\}\]
i.e. all segments in $\mathfrak m[a]$ with the whole segment to be the non-free section. It follows from Algorithm \ref{alg:der:Lang} that
\[ \left(\mathcal{D}^\mathrm{Lang}_{[b]_\rho}\right)^{\varepsilon_{b}} \circ \ldots \circ \left(\mathcal{D}^\mathrm{Lang}_{[a]_\rho}\right)^{\varepsilon_{a}} (\mathfrak{m})[a]=\mathfrak n[a]=\mathfrak{d}(\mathfrak{m})[a] .
\]

Now, one considers $\mathfrak n'=\mathfrak n-\mathfrak n[a]$. Again,
As in the discussion in \eqref{eq:Us_a+1_b}, the upward sequences of $\mathfrak n'$ are
\begin{align}\label{eq:Us_after_D_a}
 \underline{\mathfrak{Us}}(\mathfrak{n}'_i)=\begin{cases}
     \Delta_{i,2} + \ldots + \Delta_{i,r_i} &\mbox{ if } \Delta_{i,1} \in \mathfrak{m}_i[a] \text{ and } [a]_\rho \notin \mathfrak{rf}(\Delta_{i,1})\\
     {^-}\Delta_{i,1} + \Delta_{i,2} + \ldots + \Delta_{i,r_i} &\mbox{ if } [a]_\rho \in \mathfrak{rf}(\Delta_{i,1})\\
     \underline{\mathfrak{Us}}(\mathfrak{m}_i) &\mbox{ if } \mathfrak{m}_i[a]=\emptyset .
 \end{cases}   
\end{align} 
From (\ref{eq:Us_after_D_a}), note that the non-free section of each segment in $\mathfrak n'$ is the same as the non-free section of the corresponding one in $\mathfrak m$. Thus, we have, for $a'>a$
\begin{align*}
      \left(\mathcal{D}^\mathrm{Lang}_{[b]_\rho}\right)^{\varepsilon_{b}} \circ \ldots \circ \left(\mathcal{D}^\mathrm{Lang}_{[a+1]_\rho}\right)^{\varepsilon_{a}} (\mathfrak{n})[a'] &= \left(\mathcal{D}^\mathrm{Lang}_{[b]_\rho}\right)^{\varepsilon_{b}} \circ \ldots \circ \left(\mathcal{D}^\mathrm{Lang}_{[a+1]_\rho}\right)^{\varepsilon_{a}}(\mathfrak n')[a']  \\
      &=\mathfrak{d}(\mathfrak n')[a'] \\
      &=\mathfrak{d}(\mathfrak{m})[a'] ,
\end{align*}
where the first equality follows from that the segments $\mathfrak n[a]$ does not have a role in the derivatives $\mathcal D^{\mathrm{Lang}}_{[a'']_{\rho}}$ for $a''>a$; the second equality follows from induction; the third inequality follows from above discussion. This shows the lemma.

\end{proof}

A precise definition of the highest Bernstein-Zelevinsky derivative can be found in \cite[Section 4.2]{Zel}.

\begin{theorem}\label{thm:hd_Lang}
Let $\mathfrak{m} \in \mathrm{Mult}_\rho$ and  $\pi=L(\mathfrak{m})$. Then, we have $\pi^-\cong L\left( \mathfrak{d}(\mathfrak{m}) \right)$.
\end{theorem}
\begin{proof}
Let $a$ (resp. $b$) be the smallest (resp.
 largest) integer such that $\nu^a\rho$ (resp. $\nu^b\rho$) lies in $\mathrm{csupp}(\pi)$. Let $\varepsilon_a= \varepsilon^\mathrm{R}_{[a]_\rho}(\pi)$ and recursively, for each $t=a+1,\ldots,b$, let 
 \[\varepsilon_t=\varepsilon^\mathrm{R}_{[t]_\rho} \left( \left(\mathrm{D}^\mathrm{R}_{[t-1]_\rho}\right)^{\varepsilon_{t-1}} \circ \ldots \circ \left(\mathrm{D}^\mathrm{R}_{[a]_\rho}\right)^{\varepsilon_{a}} (\pi)\right).\]
Similar arugment in \cite[Proposition 3.6]{Cha_csq} deduces that
\[\left(\mathrm{D}^\mathrm{R}_{[b]_\rho}\right)^{\varepsilon_{b}} \circ \ldots \circ \left(\mathrm{D}^\mathrm{R}_{[a]_\rho}\right)^{\varepsilon_{a}} (\pi) \cong \pi^-.\]
By applying Theorem \ref{thm:der:Lang} to Lemma \ref{lem:hd_1}, we can conclude that $\pi^- \cong L(\mathfrak{d}(\mathfrak{m}))$.
 
% Using M\oe glin-Waldspurger algorithm, let $\mathfrak{n}=\mathfrak{m}^\#$ and so $\pi\cong Z(\mathfrak{n})$.  Let $c$ (resp. $d$) be the smallest (resp. largest) integer such that $\nu^c \rho \cong e(\Delta)$ (resp. $\nu^d \rho \cong e(\Delta)$) for some $\Delta \in \mathfrak{n}$. Then, $a \leq c$ and $d=b$. For $a \leq t \leq b$, define
 % \[k_t=\left| \left\{\Delta \in \mathfrak{n}: e(\Delta)\cong \nu^t \rho\right\}\right|.\]
%By Algorithm \ref{alg:der_Zel}, recursively, we have $k_t=\varepsilon_t$ for $a \leq t \leq b$. Thus, by \cite[Proposition 3.6]{Cha_csq}, we deduce that
%\[\left(\mathrm{D}^\mathrm{R}_{[b]_\rho}\right)^{\varepsilon_{b}} \circ \ldots \circ \left(\mathrm{D}^\mathrm{R}_{[a]_\rho}\right)^{\varepsilon_{a}} (\pi) \cong \pi^-.\]
%By applying Theorem \ref{thm:der:Lang} to Lemma \ref{lem:hd_1}, we can conclude that $\pi^- \cong L(\mathfrak{d}(\mathfrak{m}))$.
\end{proof}
\begin{example}
 Let $\pi=L(\mathfrak{m})$ with $\mathfrak{m}=[1,5]+[1,5]+[2,5]+[3,4]+[3,6]+[4,6]+[5,6]+[6,7]$. Then,  \begin{align*}
  \mathfrak{m}_1&=\mathfrak{m}   && \text{with }~\underline{\mathfrak{Us}}(\mathfrak{m}_1)=[1,5] + [3,6]+[6,7];\\
  \mathfrak{m}_2&=\mathfrak{m}_1-\underline{\mathfrak{Us}}(\mathfrak{m}_1)=[1,5]+[2,5]+[3,4]+[4,6]+[5,6]   & &\text{with }~\underline{\mathfrak{Us}}(\mathfrak{m}_2)=[1,5]+[4,6]; \\
  \mathfrak{m}_3&=\mathfrak{m}_2-\underline{\mathfrak{Us}}(\mathfrak{m}_2)=[2,5]+[3,4]+[5,6]   && \text{with }~\underline{\mathfrak{Us}}(\mathfrak{m}_3)=[2,5]+[5,6]\text{ and}\\
  \mathfrak{m}_4&=\mathfrak{m}_3-\underline{\mathfrak{Us}}(\mathfrak{m}_3)=[3,4]   && \text{with }~\underline{\mathfrak{Us}}(\mathfrak{m}_4)=[3,4].
 \end{align*} 
 We have the following non-free sections:
\begin{align*}
 \mathfrak{nf}([1,5])=[2,5], ~\mathfrak{nf}([3,6])=[5,6], ~\mathfrak{nf}([6,7])&=\emptyset, \\
 \mathfrak{nf}([1,5])=[3,5], ~\mathfrak{nf}([4,6])&=\emptyset,\\
 \mathfrak{nf}([2,5])=[4,5], ~\mathfrak{nf}([5,6])&=\emptyset,\\
 \mathfrak{nf}([3,4])&=\emptyset.
\end{align*} 
 Therefore, $\mathfrak{d}(\mathfrak{m})=[2,5]+[3,5]+[4,5]+[5,6],$ and $\pi^-=L\left([2,5]+[3,5]+[4,5]+[5,6] \right)$.
 \end{example}

\begin{example}
Let $\pi=L(\mathfrak{m})$ be a ladder representation, where $\mathfrak{m}=\{[1,4]_\rho, [3,6]_\rho,[7,9]_\rho\}$. Then, we have only one upward sequence of maximally linked segments with the smallest starting for $\mathfrak{m}$ given by $\underline{\mathfrak{Us}}(\mathfrak{m})=[1,4]_\rho + [3,6]_\rho + [7,9]_\rho$. Thus,
$\mathfrak{d}(\mathfrak{m})= \mathfrak{nf}([1,4]_\rho) + \mathfrak{nf}([3,6]_\rho) + \mathfrak{nf}([7,9]_\rho)=[2,4]_\rho + [6]_\rho+ \emptyset$, and $\pi^-=L([2,4]_\rho + [6]_\rho)$.
\end{example}

\begin{example}
 Let $\pi=L(\mathfrak{m})$ be a generic representation. Then, each segment $\Delta \in \mathfrak{m}$ lies in the distinct upward sequence $\underline{\mathfrak{Us}}(\mathfrak{m}_i)$, which is a singleton set. Therefore, $\mathfrak{rf}(\Delta)=\Delta$ for each $\Delta \in \mathfrak{m}$ and  $\mathfrak{d}(\mathfrak{m})= \emptyset$. Thus, $\pi^-=L(\emptyset)$, the trivial representation of $G_0$. 
\end{example}

\subsection{$\mathfrak{hd}(\pi)$ in Zelevinsky classification}
\begin{algorithm} \label{alg hd der}
Let $\mathfrak m \in \mathrm{Mult}_{\rho}$. Set $\mathfrak m_1=\mathfrak m$ and apply the following step:

Step 1. (Choose upward sequences) Let $a_1$ be the smallest integer such that $\mathfrak m_1\langle a_1 \rangle\neq \emptyset$. Let $\Delta_{1,a_1}$ be the longest segment in $\mathfrak m_1\langle a_1\rangle$. For $j \geq a_1+1$,  we recursively find the longest segment $\Delta_{1,j}$ in $\mathfrak m_1\langle j\rangle$ such that $\Delta_{1,j}$ is linked to $\Delta_{1,j-1}$. This process of choosing segments terminates when no further such segment $\Delta_{1,j}$ can be found. Set the last such segment to be $\Delta_{1,b_1}$ and define
\[ \mathfrak m_2=\mathfrak m_1-\Delta_{1,a_1}-\ldots -\Delta_{1,b_1} .
\]

Step 2. (Repeat Step 1) For $i\geq 2$, we repeat Step 1 for $\mathfrak m_i$, and obtain segments $\Delta_{i,a_i}, \ldots, \Delta_{i,b_i}$. We recursively define:
\[  \mathfrak m_{i+1}=\mathfrak m_i-\Delta_{i,a_i}-\ldots-\Delta_{i,b_i} .
\]
This removal process terminates after say $\ell$ times when $\mathfrak m_{\ell+1}=\emptyset$. 

Step 3.  Finally, we define
\[ \mathcal{H}^\mathrm{Zel}(\mathfrak m)=\sum_{i=1}^{\ell} [a_i,b_i]_{\rho} .
\]
\end{algorithm}

\begin{theorem}
For $\mathfrak m\in \mathrm{Mult}_{\rho}$, we have $ \mathfrak{hd}(Z(\mathfrak m))= \mathcal{H}^\mathrm{Zel}(\mathfrak m)$.
\end{theorem}

\begin{proof}
Most of the combinatorial arguments have been discussed before, and so we only sketch the main steps in this proof. We use the notations in Algorithm \ref{alg hd der}. Let $[a',b']_{\rho}\in \mathrm{Seg}_{\rho}$. Let $i^*$ be the largest integer such that $a_{i^*} \leq a'-1$. For each $i\leq i^*$, if $b_i\geq b'$, we set segments 
\[  \mathfrak n_i=\Delta_{i,a'-1}+\ldots +\Delta_{i,b'} \subset \mathfrak m_i
\]
and otherwise, set $\mathfrak n_i=\emptyset$. 

Now, one shows that the multisegment $\mathfrak n_1+\ldots+\mathfrak n_{i^*}$ coincides with the sum of all the removal upward sequences of maximal linked segments in neighbors on $\mathfrak m$ ranging from $a'-1$ to $b'$. This can be proved by a version of gluing suitable maximally linked segments of Lemma \ref{lem combine minimal}. 

Let 
\[ r= r_{[a',b']_{\rho}} :=|\left\{ [a', \widetilde{b}]_{\rho} \in \mathcal{H}^\mathrm{Zel}(\mathfrak m)[a'] : \widetilde{b}\geq b' \right\} |.\] 
Then, one can use the linked relation of the segments \[ \Delta_{i^*+1,a'}, \ldots, \Delta_{i^*+1,b'}, \ldots, \Delta_{i^*+r,a'}, \ldots, \Delta_{i^*+r,b'} \]
in $\mathfrak m_{i^*+1}$ to find a collection of minimally linked segments $MW_{b'}, \ldots, MW_{a'}$ such that
\begin{enumerate}
\item  for all $j=b',\ldots,a'$, $MW_j$ is a submultisegment of $\mathfrak m_{i^*+1}\langle j\rangle$, and $|MW_j|=r$;
\item $MW_{b'}$ consists of the first $r$ shortest segments in $\mathfrak m_{i^*+1}\langle b' \rangle$, and for $a'\leq j\leq b'-1$, $MW_{j}$ is minimally linked to $MW_{j+1}$ in $\mathfrak m_{i^*+1}$. 
\end{enumerate}
Showing above is similar to the proof of Lemma \ref{lem mw number=no max link zel}.

Now, this shows that $(\mathcal D_{[a',b']_{\rho}}^{\mathrm{Zel}})^r(\mathfrak m) \neq \infty$, and $\left(\mathcal D_{[a',b']_{\rho}}^{\mathrm{Zel}}\right)^{r+1}(\mathfrak m) = \infty$ can be proved by similar arguments. By Theorem \ref{thm alg der zel}, we have 
\[ \varepsilon^{\mathrm R}_{[a',b']_{\rho}}(Z(\mathfrak m))=r_{[a',b']_{\rho}}.
\]
By \cite[Propostion 5.2]{Cha_csq}, the multiplicity of the segment $[a',b']_{\rho}$ in $\mathfrak{hd}(Z(\mathfrak m))$ is precisely 
\[  \varepsilon^{\mathrm{R}}_{[a',b']_{\rho}}(Z(\mathfrak m))- \varepsilon^{\mathrm{R}}_{[a',b'+1]_{\rho}}(Z(\mathfrak m)) 
\]
and so is equal to $r_{[a',b']_{\rho}}-r_{[a',b'+1]_{\rho}}$, that is the number of segments $[a',b']_{\rho}$ in $\mathcal{H}^\mathrm{Zel}(\mathfrak m)$. Thus, one now sees that $\mathfrak{hd}(Z(\mathfrak m))=\mathcal{H}^\mathrm{Zel}(\mathfrak m)$.
\end{proof}
\begin{example}
\begin{enumerate}
\item Let $\mathfrak m=\left\{ [1,4]_{\rho}, [2,5]_{\rho}, [3,4]_{\rho},[2,6]_{\rho} \right\}$. Then, we get
\[  \mathfrak{hd}(Z(\mathfrak m))=\left\{ [4,5]_{\rho}, [4]_{\rho}, [6]_{\rho} \right\} .
\]
\item If all the segments in $\mathfrak m$ are mutually unlinked, then
\[  \mathfrak{hd}(Z(\mathfrak m))=\left\{ [e(\Delta)]_{\rho}: \Delta \in \mathfrak m \right\} .
\]
\end{enumerate}
\end{example}

%%%%%%%%%%%%%%%%%%%%%%%%%%%%%%%%%%%%%%%%%%%%%%%%%%%%%%%%%%%%%%%%%%%%%%%%%%%%%%%%%%%

\appendix

\section{More examples on $\mathfrak{tds}$-process } \label{s example tds process}

\begin{example} \label{ex tds a commute aa+1 Delta in m1}
 We consider $\mathfrak m=\left\{ [1,5]_{\rho}, [2,4]_{\rho}, [2]_{\rho}, [3,6]_{\rho} \right\}$ with $a=1$ and $b=3$. In such case, $\mathcal D_{[1,3]_{\rho}}^{\mathrm{Lang}}(\mathfrak m)=\left\{ [2,5]_{\rho}, [4]_{\rho}, [2]_{\rho}, [3,6]_{\rho} \right\}$ and $\mathfrak m_{[1,3]}=\left\{[1,5]_{\rho}, [2,4]_{\rho}, [3,6]_{\rho}\right\}$. Thus, the segments $[2,5]_{\rho}$ and $[3,6]_{\rho}$ in $\mathcal D_{[1,3]_{\rho}}^{\mathrm{Lang}}(\mathfrak m)$ participate in the $\mathfrak{tds}(\mathcal D_{[1,3]_{\rho}}^{\mathrm{Lang}}(\mathfrak m), 2)$-process. On the other hand, the segments $[2,4]_{\rho}$ and $[3,6]_{\rho}$ participate in the $\mathfrak{tds}(\mathfrak m, 2)$-process. So the role of $[2,4]_{\rho}$ in $\mathfrak m$ for $\mathfrak{tds}(\mathfrak m, 2)$-process is now replaced by the truncated $[2,5]_{\rho}$ for $\mathfrak{tds}(\mathcal D_{[1,3]_{\rho}}^{\mathrm{Lang}}(\mathfrak m), 2)$-process.
\end{example}

\begin{example}
 We consider $\mathfrak m=\left\{ [1,5]_{\rho}, [2,4]_{\rho}, [2]_{\rho}, [3,6]_{\rho} \right\}$ with $a=1$ and $b=3$. In such case, $\mathcal D_{[1,3]_{\rho}}^{\mathrm{Lang}}(\mathfrak m)=\left\{ [2,5]_{\rho}, [4]_{\rho}, [2]_{\rho}, [3,6]_{\rho} \right\}$. Thus, the segments $[2,5]_{\rho}$ and $[3,6]_{\rho}$ in $\mathcal D_{[1,3]_{\rho}}^{\mathrm{Lang}}(\mathfrak m)$ participate in the $\mathfrak{tds}(\mathcal D_{[1,3]_{\rho}}^{\mathrm{Lang}}(\mathfrak m), 2)$-process. On the other hand, the segments $[2,4]_{\rho}$ and $[3,6]_{\rho}$ participate in the $\mathfrak{tds}(\mathfrak m, 2)$-process. So the role of $[2,4]_{\rho}$ in $\mathfrak m$ for  $\mathfrak{tds}(\mathfrak m, 2)$-process is now replaced by the truncated $[2,5]_{\rho}$ for $\mathfrak{tds}(\mathcal D_{[1,3]_{\rho}}^{\mathrm{Lang}}(\mathfrak m), 2)$-process.
\end{example}

\begin{example}
Consider $\mathfrak m=\left\{ [1,6]_{\rho}, [2,3]_{\rho}, [2,5]_{\rho}, [3,4]_{\rho}, [3,7]_{\rho} \right\}$ with $a=1$ and $b=3$. In this case, $\mathcal D^{\mathrm{Lang}}_{[1,3]_{\rho}}(\mathfrak m)=\left\{ [2,6]_{\rho}, [2,3]_{\rho}, [3,5]_{\rho}, [4]_{\rho}, [3,7]_{\rho}\right\}$. Note that the segments participating in the removal steps of $\mathfrak{tus}(\mathfrak m, 2)$ are $[2,5]_{\rho}$, $[3,7]_{\rho}$, $[2,3]_{\rho}$ and $[3,4]_{\rho}$. On the other hand, the segments participating in the removal steps of $\mathfrak{tus}(\mathcal D^{\mathrm{Lang}}_{[1,3]_{\rho}}(\mathfrak m), 2)$ are $[2,6]_{\rho}$, $[3,7]_{\rho}$, $[2,3]_{\rho}$ and $[3,5]_{\rho}$. Here the truncated $[2,6]_{\rho}$ in $\mathcal D^{\mathrm{Lang}}_{[1,3]_{\rho}}(\mathfrak m)$ replaces the role $[2,5]_{\rho}$ in $\mathfrak m$, and the truncated $[3,5]_{\rho}$ in $\mathcal D^{\mathrm{Lang}}_{[1,3]_{\rho}}(\mathfrak m)$ replaces the role of $[3,4]_{\rho}$ in $\mathfrak m$.
\end{example}

\begin{example}
Consider $\mathfrak m=\left\{ [1,5]_{\rho}, [4, 7]_{\rho}, [2,3]_{\rho}, [2,5]_{\rho}, [3,4]_{\rho}\right\}$ with $a=1$ and $b=3$. In this case, $\mathcal D^{\mathrm{Lang}}_{[1,3]_{\rho}}(\mathfrak m)=\left\{ [3,5]_{\rho}, [4,7]_{\rho}, [2,3]_{\rho}, [2,5]_{\rho}, [4]_{\rho}\right\}$. Note that the segments participating in the removal steps of $\mathfrak{tus}(\mathfrak m, 2)$ are $[2,3]_{\rho}$ and $[3,4]_{\rho}$, and the segments participating in the removal steps of $\mathfrak{tus}(\mathcal D^{\mathrm{Lang}}_{[1,3]_{\rho}}(\mathfrak m), 2)$ are $[2,3]_{\rho}$ and $[3,5]_{\rho}$. Here the role of $[3,4]_{\rho}$ in $\mathfrak m$ is replaced by the truncated $[3,5]_{\rho}$ in $\mathcal D^{\mathrm{Lang}}_{[1,3]_{\rho}}(\mathfrak m)$. 
\end{example}

\section{Proof of Proposition \ref{prop dual derivative integral}} \label{appendix prop duality}

 To facilitate discussions, we define a natural bijective map:
$\Psi: \mathfrak m \rightarrow \mathbb D_r(\mathfrak m)$ determined by $\Psi([a',b']_{\rho})=[-r+b'+1,a'-1]_{\rho}$. We first make the following two simple observations:
\begin{enumerate}
\item[(a)] Two segments $\Delta$ and $\Delta'$ in $\mathfrak m$ are linked if and only if $\Psi(\Delta)$ and $\Psi(\Delta')$ in $\mathbb D_r(\mathfrak m)$ are linked. 
\item[(b)] The map
\[  \Delta \in \mathfrak m \mapsto \ell_{rel}(\Delta)+\ell_{rel}(\Psi(\Delta)) 
\]
is a constant map equal to $r$.
\end{enumerate}

We also have the following facts, whose proofs are elementary.
 Let $\Delta \in \mathfrak m$. 
 \begin{itemize}
     \item  By using the map $\Psi$, one sees that $a_1$ is the largest integer such that $a_1 \leq s(\Delta)$ and $\Delta' \prec \Delta$ for some $\Delta' \in \mathfrak m[a_1]$ if and only if $a_1$ is also the largest integer $a_1-1 \leq e(\Psi(\Delta))$ and $\Delta'' \prec \Psi(\Delta)$ for some $\Delta'' \in \mathbb D_r(\mathfrak m)\langle a_1-1\rangle$. 
     \item Moreover, for such $a_1$, by the second observation above, $\Delta'$ is the shortest choice in $\mathfrak m[a_1]$ such that $\Delta'\prec \Delta$ if and only if $\Psi(\Delta')$ is the longest choice in $\mathbb D_r(\mathfrak m)\langle a_1-1 \rangle$ such that $\Psi(\Delta') \prec \Psi(\Delta)$.
\end{itemize}

We now apply the map $\Theta$ on $\mathbb D_r(\mathfrak m)$ (resp. $\mathbb D^{[a,b]_\rho}_r(\mathfrak m)$) to apply Algorithm \ref{alg:der:Lang}. Let $\mathfrak n=\Theta( \mathbb D_r(\mathfrak m))$ (resp. $\Theta( \mathbb D^{[a,b]_\rho}_r(\mathfrak m))$). Now, it follows from the above two bullets: for $k\geq 1$ (resp. $k\geq 2$), the $k$-th upward sequence $\underline{\mathfrak{Us}}(\mathfrak n_k)$ (here $\mathfrak n_k$ is defined in an obvious way as in Algorithm \ref{alg:der:Lang}) maps naturally, under $(\Theta\circ \Psi)^{-1}$, to the downward sequence $\underline{\mathfrak{Ds}}(\mathfrak m_{k})$ (resp. $\underline{\mathfrak{Ds}}(\mathfrak m_{k-1})$), where $\mathfrak m_{k}$ and $\mathfrak m_{k-1}$ are defined as in Algorithm \ref{alg:int:Lang}. Moreover, for any $\Delta \in \underline{\mathfrak{Ds}}(\mathfrak m_{k-1})$, 
\[  \Theta\circ \Psi(\mathfrak{af}(\Delta))=\mathfrak{rf}(\Theta\circ \Psi(\Delta)) .
\]

 If we consider $\mathfrak n= \Theta (\mathbb D_r^{[a,b]_{\rho}}(\mathfrak m))=\Theta \left(\mathbb D_r(\mathfrak m)+[-r+b+1,b]_{\rho}\right)$, the first upward sequence $\underline{\mathfrak{Us}}(\mathfrak n_1)$ in Algorithm \ref{alg:der:Lang} contains only the segment $\Theta([b-r+1,b]_{\rho})$. The remaining upward sequences are exactly those from that for $\mathcal D^{\mathrm{Lang}}_{[-b,-a]_{\rho^\vee}}\left(\Theta \circ\mathbb D_r(\mathfrak m)\right)$.

Thus, running the two algorithms, if $\Delta_{p_1,q_1}, \ldots, \Delta_{p_{\ell}, q_{\ell}}$ are all segments participating in the extension process for $\mathcal I^{\mathrm{Lang}}_{[a,b]_{\rho}}(\mathfrak m)$, then 
\begin{enumerate}
    \item[(a)] if $s(\Delta_{p_{\ell}, q_{\ell}})<b+1$ (equivalently $|\mathcal I_{[a,b]_{\rho}}^{\mathrm{Lang}}(\mathfrak m)|>|\mathfrak m|$), then 
    $\Theta\circ\Psi(\Delta_{p_1,q_1}), \ldots, \Theta\circ \Psi(\Delta_{p_{\ell}, q_{\ell}})$ together with $[-b, r-b-1]_{\rho^\vee}$ participate in the truncation process for $\mathcal D^{\mathrm{Lang}}_{[-b,-a]_{\rho^\vee}}\left(\Theta \circ\mathbb D_r^{[a,b]_{\rho}}(\mathfrak m)\right)$. 
    \item[(b)] if $s(\Delta_{p_{\ell},q_{\ell}})=b+1$ (equivalently $|\mathcal I_{[a,b]_{\rho}}^{\mathrm{Lang}}(\mathfrak m)|=|\mathfrak m|$), then  $\Theta\circ\Psi(\Delta_{p_1,q_1}), \ldots, \Theta\circ \Psi(\Delta_{p_{\ell}, q_{\ell}})$ are all segments participating in the truncation process for the derivative $\mathcal D^{\mathrm{Lang}}_{[-b,-a]_{\rho^\vee}}\left(\Theta \circ\mathbb D_r(\mathfrak m)\right)$. 
\end{enumerate}

Now, using notations in Algorithm \ref{alg:int:Lang}, one can verify, for $k=1,\ldots \ell$
\[   \Theta\circ \Psi(\Delta_{p_k,q_k}^{\mathrm{ex}})=\Psi(\Delta_{p_k,q_k})^{\mathrm{trc}} ;
\]
and if $s(\Delta_{p_{\ell},q_{\ell}})<b+1$ (equivalently $|\mathcal I_{[a,b]_{\rho}}^{\mathrm{Lang}}(\mathfrak m)|>|\mathfrak m|$), then we also have:
\[  \Theta \circ \Psi([s(\Delta_{\ell}), b]_{\rho})=[-s(\Delta_{\ell})+1, r-b-1,]_{\rho^{\vee}}=([-b, r-b-1]_{\rho^{\vee}})^{\mathrm{trc}}
\]
Now, one verifies the formulas (i) and (ii) in the proposition by using the above equations, (\ref{alg_der_lang}) and (\ref{eq:int_Lang_[a,b]}).

For (iii), it is similar to the above discussion of (a) and (b) on the segments participating in the truncation process.

\section{Gluing minimally linked multisegments} \label{s glue minimal multiseg}
\begin{lemma} \label{lem combine minimal}
Let $\mathfrak m \in \mathrm{Mult}_{\rho}$ and $k \in \mathbb Z$. Let $\mathfrak p_k, \mathfrak p_k' \in \mathrm{Mult}_{\rho}$ such that all segments $\Delta$ in $\mathfrak p_k, \mathfrak p'_k$ satisfy $e(\Delta)=k$, and let $\mathfrak p_{k+1}, \mathfrak p'_{k+1} \in \mathrm{Mult}_{\rho}$ such that all segments $\Delta$ in $\mathfrak p_{k+1}, \mathfrak p'_{k+1}$ satisfy $e(\Delta)=k+1$. Suppose $|\mathfrak p_k|\leq |\mathfrak p_{k+1}|$, $|\mathfrak p_k'|\leq |\mathfrak p_{k+1}'|$, and furthermore $\mathfrak p_k$ and $\mathfrak p_{k+1}$ (resp. $\mathfrak p_k'$ and $\mathfrak p_{k+1}'$) are minimally linked in $\mathfrak m+\mathfrak p_k+\mathfrak p_{k+1}$ (resp. $\mathfrak m+\mathfrak p_k'+\mathfrak p_{k+1}'$), then $\mathfrak p_k+\mathfrak p_{k}'$ and $\mathfrak p_{k+1}+\mathfrak p_{k+1}'$ are minimally linked in $\mathfrak m+\mathfrak p_k+\mathfrak p_{k+1}+\mathfrak p_k'+\mathfrak p_{k+1}'$.
\end{lemma}

\begin{proof}
We shall first consider the case that $|\mathfrak p_k'|=|\mathfrak p_{k+1}'|=1$, and so let $\widetilde{\Delta}_k \in \mathfrak p'_k$ and let $\widetilde{\Delta}_{k+1}\in \mathfrak p'_{k+1}$. Let $r=|\mathfrak p_k|$, and let $s=|\mathfrak p_{k+1}|$. We write the segments in $\mathfrak p_k$ (resp. $\mathfrak p_{k+1}$) in the increasing order:
\[  \Delta_{k,1}, \ldots, \Delta_{k,r}, \quad (\mbox{resp.}\  \Delta_{k+1,1}, \ldots, \Delta_{k+1,s}) .
\]
Let $j^*_i$ $(i=k,k+1)$ be the smallest integer such that $s(\Delta_{i, j^*_i-1}) >s(\widetilde{\Delta}_i) \geq s(\Delta_{i,j^*_i})$.
For $1\leq x \leq r+1$, let
\[ \overline{\Delta}_{k,x}= \left\{ \begin{array}{cc}  \Delta_{k,x}   &   1\leq x \leq j_k^*-1 \\
\widetilde{\Delta}_{k} & x=j_k^* \\
\Delta_{k, x-1} &  j_k^*+1 \leq x \leq r  
\end{array} \right.
\]

We now obtain another ordering. Let $\left\{ \underline{\Delta}_{k,1}, \ldots, \underline{\Delta}_{k,r+1}\right\}$ be the submultisegment minimally linked to $\mathfrak p_{k+1}+\widetilde{\Delta}_{k+1}$ in $\mathfrak m+\mathfrak p_k+\mathfrak p_{k+1}+\widetilde{\Delta}_k+\widetilde{\Delta}_{k+1}$, written in the increasing order. For $1\leq x \leq r+1$, we have to show that $\underline{\Delta}_{k,x}=\overline{\Delta}_{k,x}$. 

\begin{enumerate}
\item Case 1: $j_k^*< j_{k+1}^*$. 
\begin{itemize}
\item For $1 \leq x \leq j_{k+1}^*-1$, if $\underline{\Delta}_{k, j_k^*}$ is in , then it contradicts that the segment $\mathfrak p_k$ is linked to the segment in $\mathfrak p_{k+1}$. 
\item For $x=j_{k+1}^*$, if $\underline{\Delta}_{k,j_{k+1}^*}$ is in $\mathfrak m$, then that segment is linked to $\widetilde{\Delta}_k$ and so, by our choice of indices, is also linked $\overline{\Delta}_{k+1, j_{k+1}^*-1}=\Delta_{k+1,j_{k+1}^*-1}$. This then contradicts the minimal linkedness of $\mathfrak p_k$ and $\mathfrak p_{k+1}$.
\item For $j_{k+1}^*+1\leq x \leq r+1$, the case is similar to above two cases.
%\item It remains to show $s=s'$. However, if it is not the case, one can show it contradicts the minimal linkedness of $\mathfrak p_k$ and $\mathfrak p_{k+1}$.
\end{itemize}
\item Case 2: $j_k^*\geq j_{k+1}^*$.
\begin{itemize}
    \item For $1\leq x\leq  j_{k+1}^*-1$, $\underline{\Delta}_{k,x}=\overline{\Delta}_{k,x}$ again comes from the minimal linkedness of $\mathfrak p_k$ and $\mathfrak p_{k+1}$.
    \item We consider $j_{k+1}^* \leq x \leq j_k^*$. Suppose $\underline{\Delta}_{k,x}$ is in $\mathfrak n$. Then $\underline{\Delta}_{k,x}$ must have to be shorter than $\widetilde{\Delta}_{k}$. This contradicts that $\widetilde{\Delta}_k$ and $ \widetilde{\Delta}_{k+1}$ are minimally linked in $\mathfrak m+\widetilde{\Delta}_k+\widetilde{\Delta}_{k+1}$.
    \item For $j_k^*<x\leq r+1$, one again uses the minimal linkedness between $\mathfrak p_k$ and $\mathfrak p_{k+1}$.
\end{itemize}
\end{enumerate}

The case that $|\mathfrak p_{k+1}'|=1$ and $|\mathfrak p_k'|=0$ is similar to the consideration of Case 2 above, and we omit the details.

For the general case, we find the first segment $\Delta_1$ in the increasing order of $\mathfrak p_{k+1}'$ and the first segment in the increasing order $\Delta_1'$ of $\mathfrak p_k'$. Then, one uses the given minimal linkedness to deduce that $\Delta_1$ and $\Delta_1'$ are minimally linked in $\mathfrak m+\Delta_1+\Delta_1'$; and $\mathfrak p-\Delta_1$ and $\mathfrak p'-\Delta_1'$ are minimally linked in $\mathfrak m+(\mathfrak p-\Delta_1)+(\mathfrak p'-\Delta_1')$. Then, one proceeds inductively by using the above two basic cases.
\end{proof}

%%%%%%%%%%%%%%%%%%%%%%%%%%%%%%%%%%%%%%%%%%%%%%%%%%%%%%%%%%%%%%%%%%%%%%%%%%%%%%%%%%%


\begin{thebibliography}{99999999}

\bibitem[Ato20]{Ato}  Hiraku Atobe,  \emph{On an algorithm to compute derivatives}. Manuscripta Math. 167 (2022), no. 3-4, 721-763. 


\bibitem[BLM13]{BLM13} Alexandru Ioan Badulescu, Erez Lapid, Alberto M\'inguez: \emph{Une condition suffisante pour l'irr\'eductibilit\'e d'une induite parabolique de $GL(m, D)$}. Ann. Inst. Fourier (Grenoble) 63(6), 2239-2266 (2013)

\bibitem[Cha22]{Ch22} Kei Yuen Chan, \emph{Restriction for general linear groups: The local non-tempered Gan-Gross-Prasad conjecture (non-Archimedean case)}. Crelles Journal, vol. 2022, no. 783, 2022, pp. 49-94. 

\bibitem[Cha23]{Cha_qbl} \bysame, \emph{Quotient branching law for p-adic $(\mathrm{GL}_{n+1}, \mathrm{GL}_{n})$ I: generalized GGP relevant pair}. arXiv:2212.05919 (v2, 2023).


\bibitem[Cha24a]{Cha_csq_ii} \bysame, \emph{Construction of simple quotients of Bernstein-Zelevinsky derivatives and highest derivative multisegments II: minimal sequences}. preprint (2024).



\bibitem[Cha24b]{Cha_csq_iii} \bysame, \emph{Construction of simple quotients of Bernstein-Zelevinsky derivatives and highest derivative multisegments III: properties of minimal sequences}. preprint (2024).

\bibitem[Cha24c]{Cha_duality} \bysame, \emph{Duality for generalized Gan-Gross-Prasad relevant pairs for $p$-adic $\mathrm{GL}_n$}, preprint, arXiv:2210.17249


\bibitem[Cha24d]{Ch24} \bysame, \emph{On the Product Functor on Inner Forms of the General Linear Group Over A Non-Archimedean Local Field}. Transformation Groups (2024).

\bibitem[Cha25]{Cha_csq} \bysame, \emph{Construction of simple quotients of Bernstein-Zelevinsky derivatives and highest derivative multisegments I: reduction to combinatorics}, Trans. Amer. Math. Soc. Ser. B 12 (2025), 851-909.  


\bibitem[CW25]{CW25} Kei Yuen Chan and Daniel Kayue Wong, \emph{On the Lefeschetz principle for $\mathrm{GL}(n, \mathbb C)$ and $\mathrm{GL}(m, \mathbb Q_p)$}. to appear in Israel Journal of Mathematics.


\bibitem[CT12]{CT12} Dan Ciubotaru and Peter Trapa, \emph{Duality for $GL(n, \mathbb{R})$, $GL(n, \mathbb{Q}_p)$, and the degenerate affine Hecke algebra for $gl(n)$}. Amer. J. Math. {\bf 134} (2012), 1-30.

\bibitem[Gur21]{Gu21} Maxim Gurevich, \emph{Quantum invariants for decomposition problems in type A rings of representations}.
Journal of Combinatorial Theory, Series A, {\bf 180} (2021).

\bibitem[Jan07]{Jan} Chris Jantzen, \emph{Jacquet modules of p-adic general linear groups}. Representation Theory: an electronic journal of the American Mathematical Society {\bf 11} (2007), 45-83. 

\bibitem[Jan14]{Jan14} \bysame, \emph{Tempered representations for classical p-adic groups}. Manuscripta. Math. 145(3-4), 319-387 (2014).

\bibitem[Jan18]{Jan18} \bysame, \emph{Jacquet modules and irreducibility of induced representations for classical p-adic groups}. Manuscripta Math. 156(1-2), 23-55 (2018). 

\bibitem[LM16]{LM16} Erez Lapid and Alberto M\'inguez, \emph{On parabolic induction on inner forms of the general linear group over a non-archimedean local field}, Select. Math., 22 No. 4 163-183 (2016), 2347-2400. 

\bibitem[LM18]{LM19} \bysame, \emph{Geometric conditions for $\square$-irreducibility of certain representations of the general linear group over a non-archimedean local field}, Advances in Mathematics 339(2018), 113-190.

\bibitem[LM20]{LM_inven} \bysame, \emph{Conjectures and results about parabolic induction of representations of $\mathrm{GL}_n(F)$}, Invent. Math. 222 (2020), no. 3, 695-747. 

\bibitem[LM22]{LM_pamq} \bysame, \emph{A binary operation on irreducible components of Lusztig's nilpotent
varieties II: applications and conjectures for representations of $GL_n$ over a non-archimedean local field}, (2022) to appear in Pure and Applied Mathematics Quarterly.


\bibitem[M\'in08]{Mi08} Alberto M\'inguez, \emph{ Correspondance de Howe explicite : paires duales de type II}, Annales scientifiques de l'\'Ecole Normale Sup\'erieure, S\'erie 4, Tome 41 (2008) no. 5, pp. 717-741.

\bibitem[M\'in09]{Min} \bysame, \emph{ Surl'irr\'educitibilit\'e d'une induite parabolique}. (French) J. Reine Angew. Math. 629 (2009), 107-131.  

\bibitem[MW86]{MW} Colette M\oe glin and Jean-Loup Waldspurger, \emph{ Sur l\'involution de Zelevinski}. (French) [The Zelevinskiĭ involution] J. Reine Angew. Math. 372 (1986), 136-177. 

\bibitem[Tad86]{Ta86} Marko Tadi\'c, \emph{On the classification of irreducible unitary representations of $GL(n)$ and the conjectures of Bernstein and Zelevinsky}, Ann. Sci. \'Ecole Norm. Sup., {\bf 19} (1986), 335-382.

\bibitem[Tad15]{Ta15} \bysame, \emph{On the reducibility points beyond the ends of complementary series of p-adic general linear groups}. J. Lie Theory 25(1), 147-183 (2015)


\bibitem[Tad22]{Ta22} \bysame, On unitarizability and Arthur packets. Manuscripta Math. 169, 327–367 (2022). 


\bibitem[Xu17]{Xu17} Bin Xu, \emph{On M\oe glin's parametrization of Arthur packets for p-adic quasisplit $Sp(N)$ and $SO(N)$}. Canad. J. Math. 69 (2017), no. 4, 890-960.

\bibitem[Zel80]{Zel} A.V. Zelevinsky, \emph{Induced representations of reductive p-adic groups. II. On irreducible representations of $GL(n)$}, Ann. Sci. École Norm. Sup. (4) 13(2)(1980) 165-210. 

\end{thebibliography}
\end{document}